\documentclass[12pt,reqno]{amsart}
\usepackage{amsmath,amssymb,amscd,amsthm,amscd,amsthm,amsfonts,mathrsfs,array}
\usepackage{graphicx,color}
\usepackage[all]{xy}
\usepackage{tikz}
\usetikzlibrary{intersections,calc,arrows.meta}
\usepackage[hidelinks]{hyperref}

\newcommand{\R}{{\mathbb{R}}}
\newcommand{\Z}{{\mathbb{Z}}}
\newcommand{\C}{{\mathbb{C}}}

\newcommand{\F}{{\mathbb{F}}}
\newcommand{\Bl}{{\mathrm{Bl}_2(\C P^2)}}
\newcommand{\Bll}{{\mathrm{Bl}_3(\C P^2)}}
\newcommand{\ii}{{\mathbf{i}}}
\newcommand{\ee}{{\mathbf{e}}}
\newcommand{\fm}{{\mathfrak{m}}}
\newcommand{\cO}{{\mathcal{O}}}
\newcommand{\cE}{{\mathcal{E}}}

\newcommand{\cI}{{\mathcal{I}}}

\newcommand{\cV}{{\mathcal{V}}}

\newcommand{\del}{{\partial}}
\newcommand{\pr}{{\mathrm{pr}}}
\newcommand{\grad}{{\mathrm{grad}}}
\newcommand{\Int}{{\mathrm{Int}}}

 \newtheorem{thm}{Theorem}[section]
 \newtheorem{lem}[thm]{Lemma}%[section]
 \newtheorem{cor}[thm]{Corollary}%[section]
\theoremstyle{definition}

\begin{document}

\title{Homological mirror symmetry of toric Fano surfaces via Morse homotopy}
\author{Hayato Naknishi}
%\thanks{*Department of Mathematics and Informatics, Graduate School of Science and Engineering, Chiba University, 1-33 Yayoicho, Inage, Chiba, 263-8522 Japan. E-mail : hayato\underline{ }nakanishi@chiba-u.jp. }
\address{Department of Mathematics and Informatics, Graduate School of Science and Engineering, Chiba University,
1-33 Yayoicho, Inage, Chiba, 263-8522 Japan.}
\email{hayato\underline{ }nakanishi@chiba-u.jp}
\date{}

\maketitle

\begin{abstract}
Strominger-Yau-Zaslow (SYZ) proposed a way of constructing mirror pairs as pairs of torus fibrations. We apply this SYZ construction to toric Fano surfaces as complex manifolds, and discuss the homological mirror symmetry, where we consider Morse homotopy of the moment polytope instead of the Fukaya category.
\end{abstract}

\tableofcontents

\section{Introduction.}
Mirror symmetry has fascinated many physicists and mathematicians. Mirror symmetry was first observed in Calabi-Yau manifolds. It has been extended to cases of non Calabi-Yau, e.g., Fano manifolds. Homological mirror symmetry conjectured by Kontsevich \cite{kon94} is a categorical formulation of mirror symmetry. This conjecture is a statement about the equivalence of two categories: the Fukaya category $Fuk(X)$ of a Calabi-Yau manifold $X$ and the derived category $D^b(Coh(\check{X}))$ of the coherent sheaves on a mirror Calabi-Yau manifold $\check{X}$. More precisely, the triangulated category induced by the Fukaya category is equivalent to $D^b(Coh(\check{X}))$. In general, we can construct a triangulated category from a given $A_\infty$ category \cite{BK, kon94}. When two triangulated categories are induced by two $A_\infty$-quasi-isomorphic $A_\infty$ categories, then the triangulated categories are equivalent to each other \cite{seidel}. Therefore, when $D^b(Coh(\check{X}))$ is obtained from an $A_\infty$ category, we can discuss the homological mirror symmetry conjecture by considering whether underlying $A_\infty$ categories are $A_\infty$-quasi-isomorphic to each other.\par
Strominger-Yau-Zaslow proposed a construction of mirror pair in \cite{SYZ}. When a given Calabi-Yau manifold is equipped with a torus fibration, we can obtain a mirror Calabi-Yau manifold as a dual torus fibration. Kontsevich-Soibelman \cite{KoSo:torus} proposed a framework to systematically prove the homological mirror symmetry via Morse homotopy for dual torus fibrations over a closed manifold without singular fibers.
%when the base of torus fibration without singular fibers is a closed manifold.
 Fukaya-Oh \cite{FO} proved that the category $Mo(B)$ of Morse homotopy on $B$ is equivalent to the Fukaya category $Fuk(T^*B)$ of the cotangent bundle $T^*B$, and Kontsevich-Soibelman's approach is based on this result.\par
Fano surfaces are also called del Pezzo surfaces, and the homological mirror symmetry of toric del Pezzo surfaces is discussed by Auroux-Katszarkov-Orlov \cite{AKO} and Ueda \cite{U}. In these papers, they consider a Fukaya-Seidel category \cite{seidel} corresponding to a Landau-Ginzburg potential of toric del Pezzo surfaces as the symplectic side. \par
In this paper, we apply SYZ construction to toric Fano surfaces as complex manifolds and discuss the homological mirror symmetry. In this situation, we consider Morse homotopy of the moment polytope instead of the Fukaya category. SYZ picture is useful to express the mirror functor explicitly, which is why we use SYZ construction instead of the Landau-Ginzburg potential.  Futaki-Kajiura \cite{fut-kaj1,fut-kaj2} proved the homological mirror symmetry for $\C P^2,\ \C P^1\times\C P^1$ and $\F_1$. They defined the category $Mo(P)$ of weighted Morse homotopy as a generalization of the weighted Fukaya-Oh category given in \cite{KoSo:torus} to the case that the base manifold has boundaries and critical points may be degenerate. However, for other toric Fano surfaces, the cleanness of the Lagrangian intersections and the conditions of the grading are not satisfied. So we need to improve the notion of $Mo(P)$ so that it can be applied well to other toric Fano surfaces. In this paper, we give a modification of the category $Mo(P)$ and prove the homological mirror symmetry for the blow-ups of $\C P^2$ at two and three points. We expect that we can discuss the homological mirror symmetry for many other toric Fano manifolds by using this modified category $Mo(P)$.\par
This paper is organized as follows. In section 2, we recall the SYZ fibration set-up and explain a plan for proving the homological mirror symmetry of toric Fano surfaces. We explicitly prove the homological mirror symmetry of the blow-up of $\C P^2$ at two points in section 3 and of the blow-up of $\C P^2$ at three points in section 4. In Appendix, we give lists of compositions of morphisms in $Mo(P)$ and non-trivial gradient trees.
\par
\noindent
{\bf Acknowledgments.}
I would like to thank my supervisor Professor Hiroshige Kajiura for various advices in writing this paper. I am also grateful to Professor Masahiro Futaki for helpful discussions. This work was supported by JST SPRING, Grant Number JPMJSP2109.

\section{Preliminaries.}
In this section, we briefly review the SYZ torus fibration set-up \cite{SYZ,LYZ,leung05} and the category $Mo(P)$ of weighted Morse homotopy \cite{KoSo:torus,fut-kaj1,fut-kaj2}. We follow the convention of \cite{hk:fukayadeform,fut-kaj1,fut-kaj2}. In subsection \ref{MM}, we construct a torus bundle and its dual torus bundle. We explain Lagrangian sections in subsection \ref{glag} and holomorphic line bundles in subsection \ref{line}. In subsection \ref{lineDG}, we define a DG categories consisting of holomorphic line bundles given in subsection \ref{line}. In subsection \ref{Mo(P)}, we modify the category $Mo(P)$ of weighted Morse homotopy given in \cite{fut-kaj1}. In subsection \ref{plan}, we explain a plan of proving the homological mirror symmetry of toric Fano surfaces.

\subsection{Torus bundle $M$ and dual torus bundle $\check{M}$.}\label{MM}
Let $B$ be an $n$-dimensional smooth manifold. A smooth manifold $B$ is called {\em affine} if $B$ has an open covering $\{U_{\lambda}\}_{\lambda \in \Lambda}$ such that the coordinate transformation is affine. This means that, for any $U_{\lambda}$ and $U_{\mu}$ such that $U_{\lambda} \cap U_{\mu} \neq \emptyset$, the coordinate systems $x_{(\lambda)} := (x_{(\lambda)}^1,\dots,x_{(\lambda)}^n)^t$ and $x_{(\mu)} := (x_{(\mu)}^1,\dots,x_{(\mu)}^n)^t$ are related to each other by
\begin{equation*}
x_{(\mu)} = \varphi_{\lambda\mu}x_{(\lambda)} + \psi_{\lambda\mu}
\end{equation*}
with some $\varphi_{\lambda\mu}\in GL(n;\R)$ and $\psi_{\lambda\mu}\in\R^n$. If in particular $\varphi_{\lambda\mu}\in GL(n;\Z)$ for any $U_{\lambda} \cap U_{\mu}$ then $B$ is called {\em tropical}. We omit the index $(\lambda)$ when no confusion may occur. For simplicity, we assume that all nonempty intersections of $U_\lambda$'s are contractible. We call $B$ {\em Hessian} when it is equipped with metric $g$ locally expressed as
\begin{equation*}
g_{ij} = \frac{\del^2 \phi}{\del x^i \del x^j}
\end{equation*}
for some smooth function $\phi$. Hereafter we assume that $(B,g)$ is a tropical and Hessian manifold.\par
Using the metric $g$ on $B$, we define the dual affine coordinates as follows: 
since  $\sum_{j=1}^ng_{ij}dx^j$ is a closed form, there exists a function $x_i:=\phi_i$ of $x$ for each $i=1,\dots,n$ such that
\begin{equation*}
dx_i=\sum_{j=1}^ng_{ij}dx^j.
\end{equation*}
We thus obtain the dual affine coordinates $\check{x}^{(\lambda)}=(x_1^{(\lambda)},\dots,x_n^{(\lambda)})^t$.\par
For each open set $U=U_\lambda$, let us denote by $(y^1,\dots,y^n)^t$ the fiber coordinates of $TB|_{U}$ and $(y_1,\dots,y_n)$ the fiber coordinates of $T^*B|_U$. The tangent bundle $TB$ is a complex manifold, where $z^i := x^i + \ii y^i$'s form the complex coordinates. The cotangent bundle $T^*B$ has a canonical symplectic form $\omega_{T^*B}=\sum_{i=1}^ndx^i\wedge dy_i$. Furthermore, we can equip $TB$ with the symplectic form $\omega_{TB}:=\sum_{i,j=1}^ng_{ij}dx^i\wedge dy^j$ and $T^*B$ with the complex structure given by $z_i:=x_i + \ii y_i$'s. These structures give the K\"ahler structures on both $TB$ and $T^*B$. Also, the symplectic form $\omega_{TB}$ coincides with the pullback of the symplectic form $\omega_{T^*B}$ by the diffeomorphism $TB\simeq T^*B$.

For an $n$-dimensional tropical Hessian manifold $B$, we consider $\Z^n$-action on $TB$ and $T^*B$. This action of $(i_1,\dots,i_n)\in\Z^n$ is defined by $y^j\mapsto y^i + 2\pi i_j$ and $y_j\mapsto y_j +2\pi i_j$, respectively. This action is well-defined because the transition functions of $TB$ and $T^*B$ are the elements of $GL(n;\Z)$. Therefore we can consider the quotients of this action to get a pair of K\"ahler manifolds $M:=TB/\Z^n$ and $\check{M}:=T^*B/\Z^n$ and dual torus fibrations ( see Figure \ref{tfdtf}).
\begin{figure}[h]
\begin{equation*}
  \xymatrix{
   M \ar[rd]_\pi& & \check{M} \ar[ld]^{\check{\pi}}  \\
   & B &  \\   
  } 
\end{equation*}
\caption{The torus fibration and the dual torus fibration}
\label{tfdtf}
\end{figure}

\subsection{Lagrangian sections of $M$.}\label{glag}
We fix a tropical open covering $\{U_\lambda\}_{\lambda\in\Lambda}$. Let $\underline{s}:B\to M$ be a section of $M \to B$. Locally, we may regard $\underline{s}$ as a section of $TB \simeq T^*B$. Then, $\underline{s}$ is locally described by a collection of functions as 
\begin{equation*}
y^i_{(\lambda)} = s^i_{(\lambda)}(x)
\end{equation*}
on each $U_\lambda$. On $U_\lambda\cap U_\mu$, these local expressions are related to each other by
\begin{equation}\label{s}
s^i_{(\mu)}(x) = s^i_{(\lambda)}(x) + 2\pi I_{\lambda\mu}
\end{equation}
for some $I_{\lambda\mu}\in\Z^n$. Here, $x$ may be identified with either $x_{(\lambda)}$ or $x_{(\mu)}$. Also, $s^i_{(\lambda)}(x)$ and $s^i_{(\mu)}(x)$ are expressed by the common coordinates $y_{(\lambda)}$ or $y_{(\mu)}$. This transformation rule automatically satisfies the cocycle condition
\begin{equation}\label{I}
I_{\lambda\mu} + I_{\mu\nu} + I_{\nu\lambda} =0
\end{equation}
for $U_\lambda\cap U_\mu\cap U_\nu\neq\emptyset$. We denote by $s$ such a collection $\{s_{(\lambda)}:U_\lambda\to TB|_{U_\lambda}\}_{\lambda\in\Lambda}$ which satisfies the transformation rule (\ref{s}) and the cocycle condition (\ref{I}). \par
Now, under our assumption $TB\simeq T^*B$, we can check whether the graph of a section $\underline{s}:B\to M$ is Lagrangian or not in $T^*B$. The section $\underline{s}:B\to M$ is regarded as a section of $T^*B$ by setting $y_i=\sum_{j=1}^ng_{ij}y^j=\sum_{j=1}^ng_{ij}s^j$, from which one has
\begin{equation*}
\sum_{i=1}^n y_idx^i = \sum_{i=1}^n\left(\sum_{j=1}^ng_{ij}s^j\right)dx^i = \sum_{j=1}^ns^jdx_j.
\end{equation*}
Thus, the graph of the section $\underline{s}:B\to M$ is Lagrangian if and only if there exists a function $f$ such that $df=\sum_{j=1}^ns^jdx_j$ locally. The gradient vector field of such function $f$ is of the form:
\begin{equation}\label{gra}
\grad(f) = \sum_{i,j}\frac{\del f}{\del x^j}g^{ji}\frac{\del}{\del x^i} = \sum_i\frac{\del f}{\del x_i}\frac{\del}{\del x^i}.
\end{equation}

\subsection{Holomorphic line bundles over $\check{M}$.}\label{line}
Consider a section $\underline{s}:B\to M$ and express it as a collection $\{s_{(\lambda)}:U_\lambda\to TB|_{U_\lambda}\}_{\lambda\in\Lambda}$ of local functions. We define a line bundle $V$ with a $U(1)$-connection on the mirror manifold $\check{M}$ associated to $\underline{s}$. We set the covariant derivative locally as
\begin{equation*}
D := d - \frac{\ii}{2\pi}\sum_{i=1}^ns^i(x)dy_i.
\end{equation*}
The curvature of this covariant derivative is 
\begin{equation*}
D^2 = \frac{\ii}{2\pi}\sum_{i,j=1}^n \frac{\del s^i}{\del x_j}dx_j\wedge dy_i.
\end{equation*}
The $(0,2)$-part vanishes if and only if the matrix $\{\frac{\del s^i}{\del x_j}\}$ is symmetric, which is the case when there exists a locally function $f$ such that $df=\sum_{i=1}s^idx_i$. Thus, the condition that $D$ defines a holomorphic line bundle on $\check{M}$ is equivalent to that the graph of $\underline{s}$ is Lagrangian in $M$. We define the transition functions of $V$ by
\begin{equation*}
\psi_{(\mu)} = e^{\ii I_{\lambda\mu}\cdot\check{y}}\psi_{(\lambda)}
\end{equation*}
for local expressions $\psi_{(\lambda)},\psi_{(\mu)}$ of a smooth section of $V$ on $U_\lambda\cap U_\mu\neq\emptyset$, where $I_{\lambda\mu}\cdot\check{y}:=\sum_i^ni_jy_j$ for $I_{\lambda\mu}=(i_1,\dots,i_n)^t$. Then the covariant derivative is defined globally.

\subsection{DG categories $\cV$ and $DG(\check{X})$ of holomorphic line bundles.}\label{lineDG}
Now, let $\check{X}$ be a smooth compact toric manifold as the complex side. Consider the complement of toric divisors $\check{M} := \check{X}\backslash\mu^{-1}(\del P)$ for the moment map $\mu:\check{X}\to P\subset\R^n$. By fixing an appropriate structure of the Hessian manifold on $B:=\Int(P)$, we get an affine torus fibration on $\check{M}\to B$ whose K\"ahler structure is induced by that of $\check{X}$. Applying the construction in the previous section, we get the dual torus fibration $M\to B$ with a K\"ahler structure on the total space.\par
First, we define the DG category $\cV$ associated with $\check{M}$. The objects are holomorphic line bundles $V$ with $U(1)$-connections
\begin{equation}\label{connection}
D := d - \frac{\ii}{2\pi}\sum_{i=1}^ny^i(x)dy_i.
\end{equation}
These line bundles $(V,D)$ are associated to the Lagrangian section $\underline{y}:B\to M$. For two objects $(V_a,D_a),(V_b,D_b)\in\cV$, the space $\cV((V_a,D_a),(V_b,D_b))$ of morphisms is defined by
\begin{equation*}
\cV((V_a,D_a),(V_b,D_b)) := \Gamma(V_a,V_b)\otimes_{C^\infty(\check{M})}\Omega^{0,*}(\check{M}),
\end{equation*}
where $\Gamma(V_a,V_b)$ is the space of homomorphisms from $V_a$ to $V_b$ and $\Omega^{0,*}(\check{M})$ is the space of anti-holomorphic differential forms on $\check{M}$. The space $\cV((V_a,D_a),(V_b,D_b))$ is a $\Z$-graded vector space whose $\Z$-grading is defined by the degree of the anti-holomorphic differential forms. Let  $\cV^r((V_a,D_a),(V_b,D_b))$ denote the degree $r$ part. We define a linear map $d_{ab}:\cV^r((V_a,D_a),(V_b,D_b))\to\cV^{r+1}((V_a,D_a),(V_b.D_b))$ as follows. We decompose $D_a$ into its holomorphic part and anti-holomorphic part $D_a= D_a^{(1,0)} + D_a^{(0,1)}$, and set a $d_a:=2D_a^{(0,1)}$. Then, for $\psi_{ab}\in\cV^r((V_a,D_a),(V_b,D_b))$, we set
\begin{equation*}
d_{ab}(\psi_{ab}) := d_b\psi_{ab} -(-1)^r\psi_{ab}d_a.
\end{equation*}
This linear map $d_{ab}:\cV^r((V_a,D_a),(V_b.D_b))\to\cV^{r+1}((V_a,D_a),(V_b,D_b))$ satisfies $d_{ab}^2=0$ since $(V_a,D_a)$ and $(V_b,D_b)$ are holomorphic line bundles. The product structure $m:\cV((V_a,D_a),(V_b,D_b))\otimes\cV((V_b,D_b),(V_c,D_c))\to\cV((V_a,D_a),(V_b,D_c))$ is defined by the composition of homomorphisms of line bundles together with the wedge product for the anti-holomorphic differential forms.\par
Next, we define the DG category $DG(\check{X})$ consisting of holomorphic line bundles on the toric manifold $\check{X}$. For line bundle $V$ on $\check{X}$, we take a holomorphic connection $D$ whose restriction to $\check{M}$ is isomorphic to a line bundle on $\check{M}$ with connection of the form
\begin{equation*}
d-\frac{\ii}{2\pi}\sum_{i=1}^n y^i(x)dy_i.
\end{equation*}
We set the objects of $DG(\check{X})$ as such pairs $(V,D)$. The space $DG(\check{X})((V_a,D_a),(V_b,D_b))$ of morphisms is defined as the $\Z$-graded vector space whose degree $r$ part is given by
\begin{equation*}
DG^r(\check{X})((V_a,D_a),(V_b,D_b)) := \Gamma(V_a,V_b)\otimes_{C^\infty(\check{X})}\Omega^{0,r}(\check{X}),
\end{equation*}
where $\Gamma(V_a,V_b)$ is the space of smooth bundle maps from $V_a$ to $V_b$. The composition of morphisms and the DG structure is defined in a similar way as that in $\cV$ above.\par
We then have a faithful embedding $\cI:DG(\check{X})\to\cV$ by restring line bundle on $\check{X}$ to $\check{M}$. We define $\cV'$ to be the image $\cI(DG(\check{X}))$ of $DG(\check{X})$ under $\cI$.

\subsection{The category $Mo(P)$ of weighted Morse homotopy.}\label{Mo(P)}
Next, we consider the symplectic side. From the homological mirror symmetry view point, what we should discuss is not $Fuk(M)$, but a kind of Fukaya category $Fuk(\bar{M})$ of a torus fibration over $P=\bar{B}$. As an intermediate step, we consider the category $Mo(P)$ of weighted Morse homotopy for the moment polytope $P$. This $Mo(P)$ is the generalization of the weighted Fukaya-Oh category given in \cite{KoSo:torus} to the case where the base manifold has boundaries and critical points may be degenerate.\par
Before we give a modification of the category $Mo(P)$ defined in \cite{fut-kaj1}, we define a slightly weakened notion of the clean intersection. Let $L$ and $L'$ be Lagrangian submanifolds. We say that the intersection $L\cap L'$ is {\em generically clean} if $L$ and $L'$ intersect with each other cleanly at any point in an open dense subset of $L\cap L'$. Next, we define the category $Mo(P)$.
\begin{itemize}
\item The objects : The objects are the Lagrangian sections $L$ of $\pi:M\to B$ which corresponds to objects of $\cI(DG(\check{X}))\subset\cV$ described in subsection \ref{lineDG}\footnote{We can also interpret this condition as a condition of potential functions $f_L$ called the {\em growth condition} defined in \cite[Definition 3.1]{Chan}.}. We can extend $L$ to a section on an open subset $P\subset\tilde{B}\subset\R^n$ smoothly. Note that there exists a function $f_L$, at least in $B$, such that $L$ is the graph of $df_L$, where the Lagrangian section $L$ is regarded as the gradient vector field $\sum\frac{\del f}{\del x_i}\frac{\del}{\del x^i}=\grad(f_L)\in\Gamma(TB)$ of $f_L$ as we see in (\ref{gra}).
\item The space of morphisms : For a given two objects $L,L'\in Mo(P)$, we assume that the intersection $L\cap L'$ forms a topological manifold and is generically clean. The space $Mo(P)(L,L')$ is the $\Z$-grading vector space spanned by the connected components $V$ of $\pi(L\cap L')\subset P$ which satisfies the following conditions:
\begin{description}
\item[(M1)] For each connected component $V\subseteq\pi(L\cap L')$, the dimension of the stable manifold $S_v\subset\tilde{B}$ of the gradient vector field $-\grad(f_L-f_{L'})$ with a generic point $v\in V$ is constant. Then we define the degree of $V$ by $|V|:=\dim(S_v)$
\item[(M2)] There exists a point $v\in V$ which is generic and is an interior point of $S_v\cap P\subset S_v$.
\end{description}

\item $A_\infty$-structure : We only explain $\fm_2$ because of the following reasons: firstly, the Morse homotopy for the toric Fano surfaces is minimal. Secondly, the set of objects $\cE$ we compute later forms strongly exceptional collection in $Tr(Mo_\cE(P))$ and therefore we do not need to compute $\fm_k$ with $k\geq3$ to compute $Tr(Mo_\cE(P))$. For more details see \cite{fut-kaj1,hk:hpt-hms}. Take a triple $(L_1,L_2,L_3)$, connected components of the intersections $V_{12}\subseteq \pi(L_1\cap L_2), V_{23}\subseteq\pi(L_2\cap L_3), V_{13}\subseteq\pi(L_1\cap L_3)$. Let $\mathcal{GT}(v_{12},v_{23};v_{13})$ be the set of trivalent gradient trees starting at $v_{12}\in V_{12}, v_{23}\in V_{23}$ and ending at $v_{13}\in V_{13}$. Define $\mathcal{GT}(V_{12},V_{23};V_{13}) := \cup_{v_{12}\in V_{12},v_{23}\in V_{23},v_{13}\in V_{13}}\mathcal{GT}(v_{12},v_{23};v_{13})$ and $\mathcal{HGT}(V_{12},V_{23};V_{13}):=\mathcal{GT}(V_{12},V_{23};V_{13})/\sim$, where $\sim$ is smooth homotopy. This set become a finite set when $|V_{13}|=|V_{12}|+|V_{23}|$ and therefore we define the composition $\fm_2$ of morphisms by
\begin{equation*}
\begin{split}
\fm_2:&Mo(P)(L_1,L_2)\otimes Mo(P)(L_2,L_3)\longrightarrow Mo(P)(L_1,L_3)\\
&(V_{12},V_{23}) \longmapsto \sum_{\substack{V_{13}\in Mo(P)(L_1,L_3)\\|V_{13}|=|V_{12}|+|V_{23}|}}\sum_{[\gamma]\in \mathcal{HGT}(V_{12},V_{23};V_{13})}e^{-A(\Gamma)}V_{13},
\end{split}
\end{equation*}
where $A(\Gamma)\in[0,\infty]$ is the symplectic area of the piecewise smooth disc in $\pi^{-1}(\Gamma(T))$ \cite{KoSo:torus}.
\end{itemize}

\subsection{Homological mirror symmetry of toric Fano surfaces.}\label{plan}
For a toric Fano surface $\check{X}$, the derived category $D^b(Coh(\check{X}))$ of the coherent sheaves has a full strongly exceptional collection $\cE$ of line bundles\cite{hille-perling11,EL}. This means that $\cE$ generates $D^b(Coh(\check{X}))$ in the sense that
\begin{equation*}
D^b(Coh(\check{X})) \simeq Tr(DG_\cE(\check{X})),
\end{equation*}
where $DG_\cE(\check{X})$ is the full DG subcategory of $DG(\check{X})$ consisting of $\cE$ and $Tr$ is the twisted complexes construction by Bondal-Kapranov \cite{BK} and Kontsevich \cite{kon94}. Also, we have a DG-quasi-isomorphism
\begin{equation*}
DG_\cE(\check{X})\overset{\sim}\to\cV_\cE'=\cI(DG_\cE(\check{X})),
\end{equation*}
where $\cI:DG(\check{X})\to\cV$ is the faithful functor which preserves the smoothness of the bundle maps on toric divisor. We denote the collection of Lagrangian sections corresponding to the exceptional collection $\cE$ by the same symbol $\cE$. We denote by $Mo_\cE(P)$ the full subcategory of $Mo(P)$ consisting of $\cE$.\par
Then the main result of this paper is the following:
\begin{thm}\label{main}
Let $\check{X}$ be a toric Fano surface, $P$ be the moment polytope of $\check{X}$ and $\cE$ be the full strongly exceptional collection of the derived category $D^b(Coh(\check{X}))$ of coherent sheaves on $\check{X}$.
Then we have a DG-quasi-isomorphisms
\begin{equation*}
Mo_\cE(P) \simeq \cV_\cE' \simeq DG_\cE(\check{X}),
\end{equation*}
where $Mo_\cE(P)$ is the full subcategory of $Mo(P)$ consisting of the collection  of Lagrangian sections mirror to $\cE$.
\end{thm}

Futaki and Kajiura prove this theorem for $\C P^n,\C P^m\times\C P^n$ and $\F_1$\cite{fut-kaj1,fut-kaj2}. In this paper, we will prove this theorem for the blow-up of $\C P^2$ at two points and three points.\par
The next corollary immediately follows from the main theorem (Theorem \ref{main}) because the triangulated categories induced from DG-quasi-isomorphic categories are isomorphic as triangulated categories.

\begin{cor}\label{main cor}
For the toric Fano surfaces $\check{X}$, we have an equivalence of triangulated categories
\begin{equation*}
Tr(Mo_\cE(P)) \simeq D^b(Coh(\check{X})),
\end{equation*}
where $P$ is the moment polytope of $\check{X}$ and $\cE$ is the collection of Lagrangian sections mirror to the full strongly exceptional collection of holomorphic line bundles on $\check{X}$.
\end{cor}

For a compact manifold $B$ without boundaries, the Fukaya category $Fuk(T^*B)$ consisting of the Lagrangian sections of $T^*B$ is $A_\infty$-equivalent to the category $Mo(B)$ of Morse homotopy on $B$ \cite{FO}. The category $Mo(P)$ is a generalization of the category $Mo(B)$ to the situations of $B$ with boundaries. Therefore, the corollary \ref{main cor} can be thought of as a  homological mirror symmetry.

\section{Blow-up of $\C P^2$ at two points.}\label{2}
In this section, we consider the homological mirror symmetry for the blow-up of $\C P^2$ at two points $\Bl$. In subsection \ref{Bl2}, we explain how to apply the SYZ construction to $\Bl$. In subsection \ref{bdl2}, we discuss line bundles on $\Bl$ constructed from the toric divisors. In subsection \ref{DGCP22}, we define the DG category $DG(\Bl)$ consisting of these line bundles and consider a full strongly exceptional collection of $D^b(Coh(\Bl))$. In subsection \ref{lagsec}, we construct Lagrangian sections which are SYZ mirror dual to the line bundles in $DG(\Bl)$. These Lagrangian sections will be the objects of $Mo(P)$. In subsection \ref{HU'}, we translate the cohomology $H(DG(\Bl))$ to the cohomology $H(\cV')$. In subsection \ref{morse homotopy}, we compute the space of morphisms in $Mo_\cE(P)$. In subsection \ref{composite}, we prove the main theorem (Theorem \ref{main}) for $\Bl$.

\subsection{Blow-up of $\C P^2$ at two points $\Bl$.}\label{Bl2}
The blow-up of $\C P^2$ at two points $\Bl$ is defined by
\begin{equation*}
\Bl:=\left\{\left([r_0:r_1],[s_0:s_1],[t_0:t_1:t_2]\right) \in \C P^1\times \C P^1\times \C P^2\ \middle|
\begin{array}{l} s_0t_1 = s_1t_0 \\ r_0t_2 = r_1t_0
\end{array}\right\}.
\end{equation*}
We take the following open covering $\{U_i\}_i$: 
\begin{align*}
U_0  =  \left\{([r_0:r_1],[s_0:s_1],[t_0:t_1:t_2]) \in \Bl\ \middle|\ t_0 \neq 0\ ,\ r_0 \neq \ 0\ ,\ s_0 \neq 0\right\},\\
U_1  =  \left\{([r_0:r_1],[s_0:s_1],[t_0:t_1:t_2]) \in \Bl\ \middle|\ t_1 \neq 0\ ,\ r_0 \neq \ 0\ ,\ s_1 \neq 0\right\},\\
U_2  =  \left\{([r_0:r_1],[s_0:s_1],[t_0:t_1:t_2]) \in \Bl\ \middle|\ t_1 \neq 0\ ,\ r_1 \neq \ 0\ ,\ s_1 \neq 0\right\},\\
U_3  =  \left\{([r_0:r_1],[s_0:s_1],[t_0:t_1:t_2]) \in \Bl\ \middle|\ t_2 \neq 0\ ,\ r_1 \neq \ 0\ ,\ s_1 \neq 0\right\},\\
U_4  =  \left\{([r_0:r_1],[s_0:s_1],[t_0:t_1:t_2]) \in \Bl\ \middle|\ t_2 \neq 0\ ,\ r_1 \neq \ 0\ ,\ s_0 \neq 0\right\}.
\end{align*}
The local coordinates in $U_i$ are
\begin{equation*}
\begin{split}
&(u_0,v_0) := \left(\frac{t_1}{t_0},\frac{t_2}{t_0}\right),\ 
(u_1,v_1) := \left(\frac{t_0}{t_1},\frac{r_1}{r_0}\right),\ 
(u_2,v_2) := \left(\frac{t_2}{t_1},\frac{r_0}{r_1}\right),\\
&(u_3,v_3) := \left(\frac{s_0}{s_1},\frac{t_1}{t_2}\right),\ 
(u_4,v_4) := \left(\frac{s_1}{s_0},\frac{t_0}{t_2}\right),
\end{split}
\end{equation*}
and the coordinate transformations turn out to be
\begin{equation}\label{cd2}
\begin{array}{cc}
(u_1,v_1) = \left(\frac{1}{u_0},v_0\right), & (u_2,v_2) = \left(\frac{v_0}{u_0},\frac{1}{v_0}\right)\\
(u_3,v_3) = \left(\frac{1}{u_0},\frac{u_0}{v_0}\right), & (u_4,v_4) = \left(u_0,\frac{1}{v_0}\right).
\end{array}
\end{equation}
There exists natural projections $\ \pr_1,\pr_2:\Bl\to\C P^1$ and $\pr_3:\Bl\to\C P^2$.
Using these projections, we define the K\"ahler form of $\Bl$ by
\begin{equation*}
\check{\omega}:= C_1\pr_1^*\left(\omega_{\C P^1}\right) + C_2\pr_2^*\left(\omega_{\C P^1}\right) + C_3\pr_3^*\left(\omega_{\C P^2}\right),
\end{equation*}
where each $C_i$ is real positive constant and $\omega_{\C P^n}$ is the Fubini-Study form on $\C P^n$. Correspondingly, the moment map $\mu:\Bl\to \R^2$ is given by
\begin{equation*}
\begin{split}
\mu([r_0:r_1],&[s_0:s_1],[t_0:t_1:t_2])\\
&:=\left(\frac{2C_2|s_1|^2}{|s_0|^2+|s_1|^2}+\frac{2C_3|t_1|^2}{|t_0|^2+|t_1|^2+|t_2|^2} , \frac{2C_1|r_1|^2}{|r_0|^2+|r_1|^2}+\frac{2C_3|t_2|^2}{|t_0|^2+|t_1|^2+|t_2|^2}\right).
\end{split}
\end{equation*}
The image $\mu(\Bl)$ is called the moment polytope associated to $\check{\omega}$, which is a pentagon. Namely, the moment polytope is given by
\begin{equation*}
P:=\left\{(x^1,x^2)\in\R^2\ \middle|\ \begin{array}{l}0\leq x^1 \leq 2(C_2+C_3)\\0\leq x^2 \leq 2(C_1+C_3)\\x^1 + x^2 \leq 2(C_1+C_2+C_3)\end{array}\right\}.
\end{equation*}
We denote each edge of $P$ by $E_i$  (see Figure \ref{mp2}).
\begin{figure}[h]
\center
\begin{tikzpicture}
\draw(0,0)-- node[auto=right]{$E_2$} (4,0)-- node[auto=right]{$E_3$} (4,2)-- node[auto=right]{$E_4$} (2,4)-- node[auto=right]{$E_5$} (0,4)-- node[auto=right]{$E_1$} cycle;
\end{tikzpicture}
\caption{The moment polytope of $\Bl.$}\label{mp2}
\end{figure}
Now, we set
\begin{equation*}
\check{M}:=U_0\cap U_1\cap U_2\cap U_3\cap U_4,\ \ \ B:=\mathrm{Int}P,
\end{equation*}
and we treat $\check{M}$ as a torus fibration $\mu|_{\check{M}}:\check{M}\to B$. Hereafter, we fix $U:=U_0,\ u:=u_0,\ v:=v_0$. Then, $\check{M}$ is equipped with an affine structure by $u=e^{x_1+\ii y_1}$ and $v=e^{x_2+\ii y_2}$, where $y_1$ and $y_2$ are the fiber coordinates of $\check{M}$. The K\"ahler form $\omega$ is expressed as
\begin{equation*}
\begin{split}
\omega &= 2\ii C_3\frac{(1+v\overline{v})^2du\wedge d\overline{u} - \overline{u}vdu\wedge d\overline{v} - u\overline{v}dv\wedge d\overline{u} +(1+u\overline{u})dv\wedge d\overline{v}}{(1+u\overline{u}+v\overline{v})^2} \\
&\ \ \ +2\ii C_2\frac{du\wedge d\overline{u}}{(1+u\overline{u})^2} +2\ii C_1\frac{dv\wedge d\overline{v}}{(1+v\overline{v})^2} \\
 &= 4 C_3\frac{(1+t)sdx_1\wedge dy_1 -stdx_1\wedge dy_2 - stdx_2\wedge dy_1 +(1+s)tdx_2\wedge dy_2}{(1+s+t)^2}\\
 &\ \ \ +4 C_2\frac{sdx_1\wedge dy_1}{(1+s)^2} +4 C_1\frac{tdx_2\wedge dy_2}{(1+t)^2},
\end{split}
\end{equation*} 
 on $U$, where $s:=|u|^2 = e^{2x_1}$ and $t:=|v^2|= e^{2x_2}$. By this expression, the inverse matrix $\{g^{ij}\}$ of the metric $\{g_{ij}\}$ on $B$ is given by
\begin{equation*}
\begin{pmatrix}
g^{11} & g^{12}\\
g^{21} & g^{22}
\end{pmatrix}
=
\begin{pmatrix}
C_2\frac{4 s}{(1+s)^2} + C_3\frac{4(1+t)s}{(1+s+t)^2} & C_3\frac{-4 st}{(1+s+t)^2} \\
C_3\frac{-4 st}{(1+s+t)^2} & C_1\frac{4 t}{(1+t)^2} + C_3\frac{4(1+s)t}{(1+s+t)^2}
\end{pmatrix}.
\end{equation*}
Now, we put $\psi:=C_3\log(1+e^{2 x_1}+e^{2 x_2}) + C_2\log(1+e^{2 x_1}) + C_1\log(1+e^{2 x_2})$ and then $\frac{\del^2 \psi}{\del x_i\del x_j}=g^{ij}$. Thus, the dual coordinate $(x^1,x^2)$ is obtained by
\begin{equation*}
\begin{split}
(x^1,x^2)
&:=
\left(\frac{\del\psi}{\del x_1} , \frac{\del\psi}{\del x_2}\right)\\
&=
\left(C_2\frac{2e^{2 x_1}}{1+e^{2 x_1}} + C_3\frac{2e^{2 x_1}}{1+e^{2 x_1}+e^{2 x_2}}, C_1\frac{2e^{2 x_2}}{1+e^{2 x_2}} + C_3\frac{2e^{2 x_2}}{1+e^{2 x_1}+e^{2 x_2}}\right)\\
&=
\mu\left([1:e^{x_1+\ii y_1}],[1:e^{x_2+\ii y_2}],[1:e^{x_1+\ii y_1}:e^{x_2+\ii y_2}]\right).
\end{split}
\end{equation*}
For simplicity, we fix $C_1=C_2=C_3=1$ since the structure of the category $Mo(P)$ we shall construct is independent of these constants.
Hereafter, we regard $M$ as the dual torus fibration of $\mu|_{\check{M}}:\check{M}\to B$.

\subsection{Holomorphic line bundles $\cO(a,b,c)$ over $\Bl$.}\label{bdl2}
Any line bundle over $\Bl$ is constructed from a toric divisor, which is a linear combination of the following divisors:
\begin{align*}
 D_1&=(t_1=s_1=0) & D_2&=(t_2=r_1=0) & D_3&=(t_0=t_2=0)\\
 D_4&=(t_0=s_0=r_0=0) & D_5&=(t_0=t_1=0). & 
\end{align*}
The corresponding Cartier divisors are as follows:
\begin{align*}
D_1 :& \{(U_0,u_0),(U_1,1),(U_2,1),(U_3,1),(U_4,u_4)\}\\
D_2 :& \{(U_0,v_0),(U_1,v_1),(U_2,1),(U_3,1),(U_4,1)\}\\
D_3 :& \{(U_0,1),(U_1,u_1),(U_2,u_2),(U_3,1),(U_4,1)\}\\
D_4 :& \{(U_0,1),(U_1,1),(U_2,v_2),(U_3,u_3),(U_4,1)\}\\
D_5 :& \{(U_0,1),(U_1,1),(U_2,1),(U_3,v_3),(U_4,v_4)\}.
\end{align*}
Using the coordinate transformations (\ref{cd2}), the transition functions of the line bundle $\cO(D_i)$ corresponding to Cartier divisor $D_i$ are as follows:
\begin{align*}
D_1:& \phi_{01}=u_0,  &  &\phi_{02}=u_0,  &  &\phi_{03}=u_0,  &  &\phi_{04}=\frac{u_0}{u_4}=1\\
D_2:& \phi_{01}=\frac{v_0}{v_1}=1,  &  &\phi_{02}=v_0,  &  &\phi_{03}=v_0,  &  &\phi_{04}=v_0\\
D_3:& \phi_{01}=\frac{1}{u_1}=u_0,  &  &\phi_{02}=\frac{1}{u_2}=\frac{u_0}{v_0},  &  &\phi_{03}=1,  &  &\phi_{04}=1\\
D_4:& \phi_{01}=1,  &  &\phi_{02}=\frac{1}{v_2}=v_0,  &  &\phi_{03}=\frac{1}{u_3}=u_0,  &  &\phi_{04}=1\\
D_5:& \phi_{01}=1,  &  &\phi_{02}=1,  &  &\phi_{03}=\frac{1}{v_3}=\frac{v_0}{u_0},  &  &\phi_{04}=\frac{1}{v_4}=v_0.
\end{align*}
Thus, we see that $\cO(D_1)=\cO(D_3+D_4)$ and $\cO(D_2)=\cO(D_4+D_5)$, then any line bundle over $\Bl$ is generated by $(D_1,D_2,D_3+D_4+D_5)$. For more details, see \cite{fulton93toric,coxtor}. \par
On the other hand, we have line bundles $\pr_1^*(\cO_{\C P^1}(1)),\ \pr_2^*(\cO_{\C P^1}(1))$ and $\pr_3^*(\cO_{\C P^2}(1))$. By comparing the transition functions of these line bundles with those of $\cO(D_i)$, we can identify $\pr_1^*(\cO_{\C P^1}(1))=\cO(D_2),\ \pr_2^*(\cO_{\C P^1}(1))=\cO(D_1)$ and $\pr_3^*(\cO_{\C P^2}(1))=\cO(D_3+D_4+D_5)$. Furthermore, the connection one-forms of $\pr_1^*(\cO_{\C P^1}(1)),\ \pr_2^*(\cO_{\C P^1}(1))$ and $\pr_3^*(\cO_{\C P^2}(1))$ are expressed as
\begin{equation*}
\begin{split}
\pr_1^*(A_{\C P^1}) &= -\frac{\overline{v}dv}{1+v\overline{v}}=-\frac{t(dx_2+\ii dy_2)}{1+t}\\
\pr_2^*(A_{\C P^1}) &= -\frac{\overline{u}du}{1+u\overline{u}}=-\frac{s(dx_1+\ii dy_1)}{1+s}\\
\pr_3^*(A_{\C P^2}) &= -\frac{\overline{u}du+\overline{v}dv}{1+u\overline{u}+v\overline{v}}=-\frac{s(dx_1+\ii dy_1) + t(dx_2+\ii dy_2)}{1+s+t}
\end{split}
\end{equation*}
on $U$. Thus, for $a,b,c\in\Z$, the connection one-form $A_{(a,b,c)}$ of $\cO(a,b,c):=\cO(aD_1 + bD_2 + c(D_3+D_4+D_5))$ is given by
\begin{equation}\label{Aabc}
A_{(a,b,c)}:=-a\frac{s(dx_1+\ii dy_1)}{1+s} - b\frac{t(dx_2+\ii dy_2)}{1+t} - c\frac{s(dx_1+\ii dy_1) + t(dx_2+\ii dy_2)}{1+s+t}.
\end{equation}

\subsection{DG category $DG(\Bl)$ and full strongly exceptional collection $\cE$.}\label{DGCP22}
We consider $DG(\Bl)$ as defined in subsection 2.4, where the objects are the line bundles $\cO(a,b,c)$ with the connection one-form $A_{(a,b,c)}$. The DG structure of $DG(\Bl)$ is given by the way described in subsection \ref{lineDG}. Since each $\cO(a,b,c)$ is a line bundle, we have
\begin{equation*}
\begin{split}
DG(\Bl)&\left(\cO(a_1,b_1,c_1),\cO(a_2,b_2,c_2)\right) \\
&\simeq DG(\Bl)\left(\cO,\cO(a_2-a_1,b_2-b_1,c_2-c_1)\right).
 \end{split}
\end{equation*}
In particular,  the zero-th cohomology of $DG(\Bl)(\cO(0,0,0),\cO(a,b,c))$ is the space $\Gamma(\Bl,\cO(a,b,c))$ of holomorphic global sections. 
\begin{equation*}
\begin{split}
H^0(DG(\Bl)(&\cO(a_1,b_1,c_1),\cO(a_2,b_2,c_2))) 
\\&\simeq H^0(DG(\Bl)(\cO(0,0,0),\cO(a_2-a_1,b_2-b_1,c_2-c_1))\\
&\simeq \Gamma(\Bl,\cO(a_2-a_1,b_2-b_1,c_2-c_1)) .
\end{split}
\end{equation*}
Using the coordinates $(u,v)$ for $U$, the generators of $\Gamma(\Bl,\cO(a,b,c))$ are expressed explicitly as
\begin{equation*}
\psi_{(i_1,i_2)}:=u^{i_1}v^{i_2},
\end{equation*}
where $0\leq i_1\leq a+c,\ 0\leq i_2\leq b+c$ and $i_1+i_2\leq a+b+c$.\par
Next, we consider a full strongly exceptional collection of the derived category of coherent sheaves $D^b(Coh(\Bl))\simeq Tr(DG(\Bl))$. For Fano surfaces, Hille and Perling proposed the construction of such collection \cite{hille-perling11}, where the collection is obtained by the generator of $\mathrm{Pic}(\C P^2)$ and the invertible sheaves corresponding to toric divisors\footnote{Elagin-Lunts proved that any exceptional collection of line bundles on Fano surfaces is given by the Hille-Perling's construction in \cite{EL}.}.
 Here, the exceptional divisors of blow-ups are $D_3$ and $D_5$. The corresponding line bundles are
\begin{equation*}
\begin{split}
\cO(D_3) &= \cO(-D_2 + D_3 + D_4 + D_5) = \cO(0,-1,1)\\
\cO(D_5) &= \cO(-D_1 + D_3 + D_4 + D_5) = \cO(-1,0,1),
\end{split}
\end{equation*}
respectively. Also, the generator of $\mathrm{Pic}(\C P^2)$ is $\cO(0,0,1)$.
We put
\begin{equation*}
\cE:=(\cO,\cO(0,-1,1),\cO(-1,0,1),\cO(0,0,1),\cO(0,0,2)).
\end{equation*}
We can identify this collection $\cE$ with a full strongly exceptional collection of Hille-Perling's construction in \cite[Theorem 5.7]{hille-perling11}.
We denote by $DG_\cE(\Bl)$ be the full subcategory of $DG(\Bl)$ consisting of $\cE$.

\subsection{Lagrangian sections $L(a,b,c)$.}\label{lagsec}
We discuss Lagrangian sections $L(a,b,c)$ of the dual torus fibration $M\to B$ corresponding to the line bundle $\cO(a,b,c)$. Twisting the fibers of $\cO(a,b,c)$ by the isomorphisms $\Psi_{(a,b,c)}:=(1+s)^{\frac{a}{2}}(1+t)^{\frac{b}{2}}(1+s+t)^{\frac{c}{2}}$, we can remove the $dx$ term from $A_{(a,b,c)}$ as follows:
\begin{equation*}
\Psi_{(a,b,c)}^{-1}(d+A_{(a,b,c)})\Psi_{(a,b,c)} = d - \ii\left( a\frac{s}{1+s} + c\frac{s}{1+s+t}\right)dy_1 - \ii\left( b\frac{t}{1+t} + c\frac{t}{1+s+t}\right)dy_2,
\end{equation*}
where $s=e^{2x_1}$ and $t=e^{2x_2}$.
Thus, the Lagrangians section $L(a,b,c)$ corresponding to $\cO(a,b,c)$ is obtained as the coefficients of $dy$ terms:
\begin{equation}\label{lag}
\begin{pmatrix}
y^1 \\ y^2
\end{pmatrix}
=2\pi
\begin{pmatrix}
 a\frac{s}{1+s} +  c\frac{s}{1+s+t} \\  b\frac{t}{1+t} +  c\frac{t}{1+s+t}
\end{pmatrix}.
\end{equation}
By this expression, we can extend $L(a,b,c)$ on $B$ to that on $P$ smoothly.
The potential function of this Lagrangian section is given by
\begin{equation}\label{fabc}
f =  \pi a\log(1+s) +  \pi b\log(1+t) +  \pi c\log(1+s+t).
\end{equation}
The collection of the Lagrangian sections corresponding to the full strongly exceptional collection $\cE$ given in subsection \ref{DGCP22} is denoted by the same symbol $\cE$, i.e.,
\begin{equation*}
\cE=\left(L(0,0,0),L(0,-1,1),L(-1,0,1),L(0,0,1),L(0,0,2)\right).
\end{equation*}

\subsection{Cohomologies $H(\cV')$.}\label{HU'}
Let $\cV$ be a DG category of holomorphic line bundles over $\check{M}$. We consider the faithful functor $\cI:DG(\Bl)\to\cV$. Let $\widetilde{\cO}(a,b,c)$ denote the line bundle $\cO(a,b,c)$ twisted by $\Psi_{(a,b,c)}$. This functor assigns $\cO(a,b,c)$ to $\widetilde{\cO}(a,b,c)$. Also, each generator $\psi_{(i_1,i_2)}\in DG(\Bl)(\cO,\cO(a,b,c))$ is sent to be $\Psi^{-1}_{(a,b,c)}\psi_{(i_1,i_2)}\in\cV(\widetilde{\cO},\widetilde{\cO}(a,b,c))$, i.e.,
\begin{equation}\label{pre-e}
\begin{split}
\Psi^{-1}_{(a,b,c)}\psi_{(i_1,i_2)} &= (1+s)^{\frac{-a}{2}}(1+t)^{\frac{-b}{2}}(1+s+t)^{\frac{-c}{2}}u^{i_1}v^{i_2}\\
 &=(1+s)^\frac{-a}{2}(1 + t)^\frac{-b}{2}(1 + s + t)^\frac{-c}{2}s^\frac{i_1}{2}t^\frac{i_2}{2}e^{\ii(i_1y_1 + i_2y_2)}.
\end{split}
\end{equation}
The images are denoted by $\cV':=\cI(DG(\Bl))$ and $\cV_\cE':=\cI(DG_\cE(\Bl))$, respectively. Then, the basis of $H^0(\cV(\widetilde{\cO},\widetilde{\cO}(a,b,c))$ are given by (\ref{pre-e}). If the functions $\Psi^{-1}_{(a,b,c)}\psi_{(i_1,i_2)}$ on $B$ extend to that on $P$ smoothly, then we rescale each basis $\Psi^{-1}_{(a,b,c)}\psi_{(i_1,i_2)}$ by multiplying a positive number and denote it by $\ee_{(a,b,c);(i_1,i_2)}$ so that
\begin{equation*}
\max_{x\in P}|\ee_{(a,b,c);(i_1,i_2)}(x)|=1.
\end{equation*}

\subsection{Weighted Morse homotopy $Mo_{\cE}(P)$.}\label{morse homotopy}
For the moment polytope $P$ of $\Bl$, we construct the full subcategory $Mo_\cE(P)\subset Mo(P)$ consisting of $\cE$, where $Mo(P)$ is defined in subsection \ref{Mo(P)}. The objects of $Mo(P)$ are Lagrangian sections $L(a,b,c)$ obtained in subsection \ref{lagsec}.  Since we have
\begin{equation*}
Mo(P)(L(a_1,b_1,c_1),L(a_2,b_2,c_2)) \simeq Mo(P)(L(0,0,0),L(a_2-a_1,b_2-b_1,c_2-c_1)),
\end{equation*}
we concentrate to computing the space $Mo_\cE(P)(L(0,0,0),L(a,b,c))$.
As considered in subsection \ref{glag}, the intersections of $L(0,0,0)$ and $L(a,b,c)$ are expressed as
\begin{equation}\label{eq}
\begin{split}
2\pi
\begin{pmatrix}
i_1 \\ i_2
\end{pmatrix}
=
2\pi\begin{pmatrix}
 a\frac{s}{1+s} +  c\frac{s}{1+s+t} \\  b\frac{t}{1+t} +  c\frac{t}{1+s+t}
\end{pmatrix}
\end{split}
\end{equation} 
in the covering space of $\pi:\bar{M}\to P$, where $(i_1,i_2)\in\Z^2$, $s=e^{2x_1}$ and $t=e^{2x_2}$. If there exists a nonempty intersection, then we set $V_{(a,b,c);(i_1,i_2)}:=\pi(L(0,0,0)\cap L(a,b,c))$. We check  that $V_{(a,b,c);(i_1,i_2)}$ satisfies the conditions (M1) and (M2) given in subsection \ref{Mo(P)}. In this case, the gradient vector field associated to $V_{(a,b,c);(i_1,i_2)}$ is of the form
\begin{equation}\label{grad2}
2\pi\left(a\frac{s}{1+s} + c\frac{s}{1+s+t} - i_1\right)\frac{\del}{\del x^1} + 2\pi\left(b\frac{t}{1+t} + c\frac{t}{1+s+t} - i_2 \right)\frac{\del}{\del x^2}.
\end{equation}\par
On the other hand, for the space of the opposite directional morphisms, we have
\begin{equation*}
Mo(P)(L(a,b,c),L(0,0,0)) \cong Mo(P)(L(0,0,0),L(-a,-b,-c)).
\end{equation*}
Thus, the connected component $V_{(a,b,c);(i_1,i_2)}$ coincides with the connected component $V_{(-a,-b,-c);(-i_1,-i_2)}$. The gradient vector field associated with $V_{(-a,-b,-c);(-i_1,-i_2)}$ is the opposite direction of (\ref{grad2}).\par
\begin{figure}[h]
\center
\begin{tikzpicture}
\draw(0,0)--node[auto=left]{$E_2$} node[auto=right]{$t=0$} (4,0)-- node[auto=left]{$E_3$}node[auto=right]{$s=\infty,\ t\leq\infty$} (4,2)--node[auto=left]{$E_4$} node[auto=right]{$s,t=\infty$} (2,4)--node[auto=left]{$E_5$} node[auto=right]{$s\leq\infty,t=\infty$} (0,4)--node[auto=left]{$E_1$} node[auto=right]{$s=0$} cycle;
\draw(4,2)node[right]{$\frac{t}{s}=0$};
\draw(2,4)node[right]{$\frac{t}{s}=\infty$};
\end{tikzpicture}
\caption{The moment polytope of $\Bl$.}
\label{poly2}
\end{figure}
For all $L$ and $L'$, we compute generators of $Mo_\cE(P)(L,L')$.

\begin{itemize}
\item $Mo_\cE(P)\left(L(0,0,0),L(0,-1,1)\right)$ and $Mo_\cE(P)\left(L(0,0,0),L(-1,0,1)\right)$:\ \par
The intersection of $L(0,0,0)$ and $L(0,-1,1)$ is expressed as
\begin{equation*}
i_1 = \frac{s}{1+s+t},\ \ \ 
i_2 = -\frac{t}{1+t} + \frac{t}{1+s+t}.
\end{equation*}
Then, since $(x^1,x^2)=(\frac{2s}{1+s}+\frac{2s}{1+s+t},\frac{2t}{1+t}+\frac{2t}{1+s+t})$, we obtain the connected components
\begin{equation*}
V_{(0,-1,1);(0,0)}  = E_1 \cup E_5,\ \ \ 
V_{(0,-1,1);(1,0)}  = \{(4,0)\},\ \ \ 
V_{(0,-1,1);(1,-1)} = \{(4,2)\},
\end{equation*}
where $E_i$ is the edge of the moment polytope associated to each toric divisor $D_i$ (see Figure \ref{poly2}). For each $(i_1,i_2)$, the gradient vector field is given by
\begin{equation*}
2\pi\left(\frac{s}{1+s+t} - i_1\right)\frac{\del}{\del x^1} + 2\pi\left(-\frac{t}{1+t} + \frac{t}{1+s+t} - i_2 \right)\frac{\del}{\del x^2}.
\end{equation*}
The degrees of the connected components are $|V_{(0,-1,1);(0,0)}|=0$ and $|V_{(0,-1,1);(1,0)}|=|V_{(0,-1,1);(1,-1)}|=1$. In fact, for $V_{(0,-1,1);(1,0)}$ and $V_{(0,-1,1);(1,-1)}$, the stable manifold is $\{(4,x^2)\}$. Then, since $V_{(0,-1,1);(1,0)}$ and $ V_{(0,-1,1);(1,-1)}$ are vertices of the moment polytope $P$, they do not satisfy the condition (M2) in subsection \ref{Mo(P)} and can not be the generators. Thus, the connected component $V_{(0,-1,1);(0,0)}$ is the only generator of $Mo_\cE(P)\left(L(0,0,0),L(0,-1,1)\right)$ of degree zero:
\begin{equation*}
Mo_\cE(P)\left(L(0,0,0),L(0,-1,1)\right) = \C\cdot V_{(0,-1,1);(0,0)},\ \ \ |V_{(0,-1,1);(0,0)}|=0.
\end{equation*}\par
On the other hand, for the space $Mo_\cE(P)\left(L(0,-1,1),L(0,0,0)\right)$, the degree of any connected components is positive. For this reason, we have
\begin{equation*}
Mo_\cE(P)\left(L(0,-1,1),L(0,0,0)\right) = 0.
\end{equation*}\par
By exchanging $i_1$ for $i_2$ and $s$ for $t$, we have
\begin{equation*}
\begin{split}
V_{(-1,0,1);(0,0)} &= E_2 \cup E_3, \\
|V_{(-1,0,1);(0,0)}| &= 0 ,\\
Mo_\cE(P)\left(L(0,0,0),L(-1,0,1)\right) &= \C\cdot V_{(-1,0,1);(0,0)},\\
Mo_\cE(P)\left(L(-1,0,1),L(0,0,0)\right) &= 0.
\end{split}
\end{equation*}

\item $Mo_\cE(P)\left(L(0,0,0),L(0,0,1)\right)$ and $Mo_\cE(P)\left(L(0,0,1),L(0,0,2)\right)$:\ \par
The intersection of $L(0,0,0)$ and $L(0,0,1)$ is expressed as
\begin{equation*}
i_1 = \frac{s}{1+s+t},\ \ \ 
i_2 = \frac{t}{1+s+t}.
\end{equation*}
Then, we obtain the connected components
\begin{equation*}
V_{(0,0,1);(0,0)}  = \{(0,0)\},\ \ \ 
V_{(0,0,1);(1,0)}  = E_3,\ \ \ 
V_{(0,0,1);(0,1)}  = E_5,
\end{equation*}
and the corresponding gradient vector field
\begin{equation*}
2\pi\left(\frac{s}{1+s+t} - i_1\right)\frac{\del}{\del x^1} + 2\pi\left(\frac{t}{1+s+t} - i_2 \right)\frac{\del}{\del x^2}.
\end{equation*}
Since the degrees of these connected components are zero, we have
\begin{equation*}
Mo_\cE(P)\left(L(0,0,0),L(0,0,1)\right) = \C\cdot V_{(0,0,1);(0,0)} \oplus \C\cdot V_{(0,0,1);(1,0)} \oplus \C\cdot V_{(0,0,1);(0,1)}.
\end{equation*}
Furthermore, we have
\begin{equation*}
\begin{split}
Mo_\cE(P)\left(L(0,0,1),L(0,0,2)\right) &\cong Mo(P)\left(L(0,0,0),L(0,0,1)\right)\\
&= \C\cdot V_{(0,0,1);(0,0)} \oplus \C\cdot V_{(0,0,1);(1,0)} \oplus \C\cdot V_{(0,0,1);(0,1)}.
\end{split}
\end{equation*}\par
On the other hand, the space $Mo_\cE(P)\left(L(0,0,1),L(0,0,0)\right)$ is trivial. In fact, the degree of $V_{(0,0,-1);I}$ is greater than 0 and none of the point $v_{(0,0,-1);I}\in V_{(0,0,-1);I}$ is an interior point of $P$. Thus, we have
\begin{equation*}
\begin{split}
Mo_\cE(P)\left(L(0,0,1),L(0,0,0)\right) &= 0,\\
Mo_\cE(P)\left(L(0,0,2),L(0,0,1)\right) &= 0.
\end{split}
\end{equation*}

\item $Mo_\cE(P)\left(L(0,0,0),L(0,0,2)\right)$:\ \par
The intersection of $L(0,0,0)$ and $L(0,0,2)$ is expressed as
\begin{equation*}
i_1 = \frac{2s}{1+s+t},\ \ \ 
i_2 = \frac{2t}{1+s+t}.
\end{equation*}
Then, we obtain the connected components
\begin{align*}
V_{(0,0,2);(0,2)} &= E_5\\
V_{(0,0,2);(0,1)} &= \{(0,2)\} & V_{(0,0,2);(1,1)} &= \{(3,3)\}\\ 
V_{(0,0,2);(0,0)} &= \{(0,0)\} & V_{(0,0,2);(1,0)} &= \{(2,0)\} & V_{(0,0,2);(2,0)} &= E_3,
\end{align*}
and the corresponding gradient vector field
\begin{equation*}
2\pi\left(\frac{2s}{1+s+t} - i_1\right)\frac{\del}{\del x^1} + 2\pi\left(\frac{2t}{1+s+t} - i_2 \right)\frac{\del}{\del x^2}.
\end{equation*}
For each $V_{(0,0,2);I}$, the degree is zero. Thus, we have
\begin{equation*}
Mo_\cE(P)\left(L(0,0,0),L(0,0,2)\right) = \bigoplus_{\substack{0\leq i_1 \leq 2\\0\leq i_2\leq2\\0\leq i_1+i_2\leq2}} \C\cdot V_{(0,0,2);(i_1,i_2)}.
\end{equation*}
For the space $Mo_\cE(P)\left(L(0,0,2),L(0,0,0)\right)$, since each $V_{(0,0,2);I}$ is included in $\del P$ and its degree is positive, it can not be generators. Thus, we have
\begin{equation*}
Mo_\cE(P)\left(L(0,0,2),L(0,0,0)\right) = 0.
\end{equation*}

\item $Mo_\cE(P)\left(L(0,-1,1),L(-1,0,1)\right)$:\ \par
Since $Mo_\cE(P)\left(L(0,-1,1),L(-1,0,1)\right) \cong Mo_\cE(P)\left(L(0,0,0),L(-1,1,0)\right)$, we consider $Mo_\cE(P)\left(L(0,0,0),L(-1,1,0)\right)$. The intersection of $L(0,0,0)$ and $L(-1,1,0)$ is expressed as
\begin{equation*}
i_1 = -\frac{s}{1+s},\ \ \ i_2 = \frac{t}{1+t}.
\end{equation*}
Then, we obtain the connected components
\begin{align*}
V_{(-1,1,0);(0,1)} &= \{(0,4)\},  &   V_{(-1,1,0);(-1,1)} &= E_4,\\
V_{(-1,1,0);(0,0)} &= \{(0,0)\},  &   V_{(-1,1,0);(-1,0)} &= \{(4,0)\},
\end{align*}
and the corresponding gradient vector field
\begin{equation*}
2\pi\left(-\frac{s}{1+s} - i_1\right)\frac{\del}{\del x^1} + 2\pi\left(\frac{t}{1+t} - i_2 \right)\frac{\del}{\del x^2}.
\end{equation*}
These connected components are not generators of $Mo_\cE(P)\left(L(0,0,0),L(-1,1,0)\right)$. In fact, when $(i_1,i_2)$ is $(0,0),\ (0,1)$ or $(-1,0)$, the corresponding stable manifold is as follows:
\begin{equation*}
S_{(0,4)} = \{(x^1,4)\},\ \ \  S_{(0,0)} = \{(4,x^2)\}  \ \ \  \mbox{or}\ \ \   S_{(4,0)} = \{(4,x^2)\},
\end{equation*}
respectively.
For $(i_1,i_2)=(-1,1)$, the dimension of the stable manifold is not constant. Now, we decompose $V_{(-1,1,0);(-1,1)}$ into the upper half and the lower half:
\begin{equation*}
\begin{split}
V_{(-1,1,0);(-1,1)} &= V_{(-1,1,0);(-1,1)}^+ \cup V_{(-1,1,0);(-1,1)}^- ,\\
V_{(-1,1,0);(-1,1)}^- &:= \{(x^1,x^2)\in P\ |\ x^1+x^2=6,\ 2\leq x^2\leq3\},\\
V_{(-1,1,0);(-1,1)}^+ &:= \{(x^1,x^2)\in P\ |\ x^1+x^2=6,\ 3<x^2\leq4\}.
\end{split}
\end{equation*}
Then, we have $\dim S_v =1$ (resp. $\dim S_v =0$) if $v\in V_{(-1,1,0);(-1,1)}^+$ (resp. $v\in V_{(-1,1,0);(-1,1)}^-$).\par
For the space of opposite directional morphisms, the degree of each connected components is similar to those above. Thus, we have
\begin{align*}
Mo_\cE(P)\left(L(0,-1,1),L(-1,0,1)\right) &= 0, \\
Mo_\cE(P)\left(L(-1,0,1),L(0,-1,1)\right) &= 0.
\end{align*}

\item $Mo_\cE(P)\left(L(0,-1,1),L(0,0,1)\right)$ and $Mo_\cE(P)\left(L(-1,0,1),L(0,0,1)\right)$:\ \par
Since $Mo_\cE(P)\left(L(0,-1,1),L(0,0,1)\right) \cong Mo_\cE(P)\left(L(0,0,0),L(0,1,0)\right)$, we consider $Mo_\cE(P)\left(L(0,0,0),L(0,1,0)\right)$. The intersection of $L(0,0,0)$ and $L(0,1,0)$ is expressed as
\begin{equation*}
i_1 = 0,\ \ \ i_2 = \frac{t}{1+t}.
\end{equation*}
Then, we obtain the connected components
\begin{equation*}
V_{(0,1,0);(0,0)} = E_2,\ \ \ V_{(0,1,0);(0,1)} = E_4 \cup E_5,
\end{equation*}
and the corresponding gradient vector field
\begin{equation*}
2\pi\left(\frac{t}{1+t} - i_2 \right)\frac{\del}{\del x^2}.
\end{equation*}
The degree of these $V_{(0,1,0);I}$ is zero. Also, since $|V_{(0,-1,0);J}|\geq1$, we have
\begin{align*}
Mo_\cE(P)\left(L(0,-1,1),L(0,0,1)\right) &= \C\cdot V_{(0,1,0);(0,0)} \oplus \C\cdot V_{(0,1,0);(0,1)}, \\
Mo_\cE(P)\left(L(0,0,1),L(0,-1,1)\right) &= 0.
\end{align*} 
By exchanging $i_1$ for $i_2$ and $s$ for $t$, we have
\begin{align*}
V_{(1,0,0);(0,0)} &= E_1, \\
V_{(1,0,0);(1,0)} &= E_3 \cup E_4, \\
Mo_\cE(P)\left(L(-1,0,1),L(0,0,1)\right) &= \C\cdot V_{(1,0,0);(0,0)} \oplus \C\cdot V_{(1,0,0);(1,0)}, \\
Mo_\cE(P)\left(L(0,0,1),L(-1,0,1)\right) &= 0.
\end{align*}

\item $Mo_\cE(P)\left(L(0,-1,1),L(0,0,2)\right)$ and $Mo_\cE(P)\left(L(-1,0,1),L(0,0,2)\right)$:\ \par
Since $Mo_\cE(P)\left(L(0,-1,1),L(0,0,2)\right) \cong Mo_\cE(P)\left(L(0,0,0),L(0,1,1)\right)$, we consider $Mo_\cE(P)\left(L(0,0,0),L(0,1,1)\right)$. The intersection of $L(0,0,0)$ and $L(0,1,1)$ is expressed as
\begin{equation*}
i_1 = \frac{s}{1+s+t},\ \ \ 
i_2 = \frac{t}{1+t} + \frac{t}{1+s+t} = \frac{x^2}{2}.
\end{equation*}
Then, we obtain the connected components
\begin{align*}
V_{(0,1,1);(0,2)} &= E_5 ,\\
V_{(0,1,1);(0,1)} &= \{(0,2)\}, & V_{(0,1,1);(1,1)} &= \{(4,2)\}, \\
V_{(0,1,1);(0,0)} &= \{(0,0)\}, & V_{(0,1,1);(1,0)} &= \{(4,0)\},
\end{align*}
and the corresponding gradient vector field
\begin{equation*}
2\pi\left(\frac{s}{1+s+t} - i_1\right)\frac{\del}{\del x^1} + 2\pi\left(\frac{t}{1+t} + \frac{t}{1+s+t} - i_2 \right)\frac{\del}{\del x^2}.
\end{equation*}
The degree of these $V_{(0,1,1);I}$ is zero. Also, since $|V_{(0,-1,-1);J}|\geq1$, we have
\begin{align*}
Mo_\cE(P)\left(L(0,-1,1),L(0,0,2)\right) &= \bigoplus_{\substack{0\leq i_1 \leq 1\\ 0\leq i_2 \leq 2\\ 0 \leq i_1+i_2 \leq 2}}\C\cdot V_{(0,1,1);(i_1,i_2)},\\
Mo_\cE(P)\left(L(0,0,2),L(0,-1,1)\right) &= 0
\end{align*}
By exchanging $i_1$ for $i_2$ and $s$ for $t$, we have
\begin{align*}
V_{(1,0,1);(0,1)} &= \{(0,4)\}, & V_{(1,0,1);(1,0)} &= \{(2,4)\}, \\
V_{(1,0,1);(0,0)} &= \{(0,0)\}, & V_{(1,0,1);(1,0)} &= \{(2,0)\}, & V_{(1,0,1);(2,0)} &= E_3,
\end{align*}
and the space of morphisms
\begin{align*}
Mo_\cE(P)\left(L(-1,0,1),L(0,0,2)\right) &= \bigoplus_{\substack{0\leq i_1 \leq 2\\ 0\leq i_2 \leq 1\\ 0 \leq i_1+i_2 \leq 2}}\C\cdot V_{(1,0,1);(i_1,i_2)},\\
Mo_\cE(P)\left(L(0,0,2),L(-1,0,1)\right) &= 0.
\end{align*}

\end{itemize}
Summarizing the above results, we obtain the following.

\begin{lem}
Each generator of the space of morphisms in $Mo_\cE(P)$ belongs to $\del P$ and its degree is zero.
\end{lem}

\subsection{Proof of the main theorem for $\check{X}=\Bl$.}\label{composite}
In order to prove the main theorem (Theorem \ref{main}), we construct a quasi-isomorphism $\iota:Mo_\cE(P)\to\cV_\cE'$ explicitly. For the objects, we assign $L(a,b,c)$ to $\widetilde{\cO}(a,b,c)$. For the morphisms, we obtain the following.

\begin{lem}\label{lem2-2}
The basis $\ee_{(a,b,c);I}$ of $H^0(\cV'_\cE(\widetilde{\cO}(a_1,b_1,c_1),\widetilde{\cO}(a_1+a,b_1+b,c_1+c)))$ are expressed as the form
\begin{equation*}
\ee_{(a,b,c);I}(x)=e^{-f_I}e^{\ii Iy},
\end{equation*}
where $e^{-f_I}$ is continuous on $P$ and smooth on $B$. Furthermore, the function $f_I$ satisfies
\begin{equation}\label{dfI}
df_I = \sum_{j=1}^2\frac{\del f_I}{\del x_j}dx_j,\ \ \ \frac{\del f_I}{\del x_j} = \frac{y^j_{(a,b,c)} - 2\pi i_j}{2\pi}
\end{equation}
in $B$ and $\min_{x\in P}f_I=0$. In particular, we have 
\begin{equation}\label{VI2}
\{x\in P\ |\ f_I(x)=0\}=V_{(a,b,c);I}.
\end{equation}
Thus, the correspondence $\iota:V_{(a,b,c);I}\mapsto \ee_{(a,b,c);I}$ gives a quasi-isomorphism
\begin{equation*}
\iota : Mo_\cE(P)(L(a_1,b_1,c_1),L(a_1+a,b_1+b,c_1+c))\to \cV'_\cE(\widetilde{\cO}(a_1,b_1,c_1),\widetilde{\cO}(a_1+a,b_1+b,c_1+c))
\end{equation*}
of complexes.
\end{lem}
\begin{proof}
By the expression (\ref{pre-e}), we see that the basis $\ee_{(a,b,c);I}$ of $H^0(\cV'_\cE(\widetilde{\cO}(a_1,b_1,c_1),\widetilde{\cO}(a_1+a,b_1+b,c_1+c)))$ are expressed as the form
\begin{equation*}
\ee_{(a,b,c);I} =c_{(a,b,c);(i_1,i_2)}(1+s)^\frac{-a}{2}(1 + t)^\frac{-b}{2}(1 + s + t)^\frac{-c}{2}s^\frac{i_1}{2}t^\frac{i_2}{2}e^{\ii(i_1y_1 + i_2y_2)}
\end{equation*}
where $s = e^{2 x_1},\ t= e^{2 x_2}$ and $c_{(a,b,c);(i_1,i_2)}$ is a constant. Then, the functions $e^{-f_I}$ and $f_I$ are given by
\begin{equation}\label{fI2}
\begin{split}
e^{-f_I}& = c_{(a,b,c);I}(1+s)^{-\frac{a}{2}}(1 + t)^{-\frac{b}{2}}(1 + s + t)^{-\frac{c}{2}}s^\frac{i_1}{2}t^\frac{i_2}{2},\\
f_I  &= \log\left((1+s)^\frac{a}{2}(1 + t)^\frac{b}{2}(1 + s + t)^\frac{c}{2}s^\frac{-i_1}{2}t^\frac{-i_2}{2}\right) + \mbox{const}  \\
&= \frac{a}{2}\log(1+e^{2 x_1}) + \frac{b}{2}\log(1+e^{2 x_2}) + \frac{c}{2}\log(1+e^{2 x_1}+e^{2 x_2}) -  i_1 x_1 -  i_2 x_2 + \mbox{const}.
\end{split}
\end{equation}
By this expression, the function $f_I$ satisfy (\ref{dfI}). For $(a,b,c)=(-1,0,1)$ and $(0,-1,1)$, we have
\begin{align*}
e^{-f_I} &= c_{(-1,0,1);I}\left(\frac{1+s}{1+s+t}\right)^{\frac{1}{2}} = c_{(-1,0,1);I}\left(\frac{1}{1+\frac{t}{1+s}}\right)^{\frac{1}{2}},\\
e^{-f_I} &= c_{(0,-1,1);I}\left(\frac{1+t}{1+s+t}\right)^{\frac{1}{2}} = c_{(0,-1,1);I}\left(\frac{1}{1+\frac{s}{1+t}}\right)^{\frac{1}{2}},
\end{align*}
since $i_1=i_2=0$.
Thus, the function $e^{-f_I}$ is continuous on $P$.
For $a,b,c\geq0$, since $0\leq i_1 \leq a+c,\ 0\leq i_2 \leq b+c$, the function $e^{-f_I}$ is also continuous on $P$. Furthermore, we obtain (\ref{VI2}) by evaluating $f_I$ in (\ref{fI2}) in each case directly. Since the DG structure of $Mo_\cE(P)$ is minimal, we obtain the last statement.
\end{proof}

\begin{lem}\label{lem2-3}
The correspondence $\iota$ forms a functor from $Mo_\cE(P)$ to $\cV'_\cE$.
\end{lem}
\begin{proof}
By lemma \ref{lem2-2}, the identity morphism in $Mo_\cE(P)$ is maped to that of in $\cV_\cE'$.\par
The composition in $\cV_\cE'$ or $H^0(\cV_\cE')$ is simply a product between functions, i.e.,
\begin{equation}\label{prod-e-pre}
e^{-f_{(\alpha,\beta,\gamma);I}}e^{\ii Iy} \otimes e^{-f_{(\alpha',\beta',\gamma');J}}e^{\ii Jy} \longmapsto e^{-\left(f_{(\alpha,\beta,\gamma);I}+f_{(\alpha',\beta',\gamma');J}\right)}e^{\ii (I+J)y},
\end{equation}
where $\alpha:=a_2-a_1,\ \beta:=b_2-b_1,\ \gamma:=c_2-c_1,\ \alpha':=a_3-a_2,\ \beta':=b_3-b_2,\ \gamma':=c_3-c_2$.
Here, let $v$ be a point where the function $f_{(\alpha,\beta,\gamma);I}+f_{(\alpha',\beta',\gamma')}$ is minimum. Then, we can rewrite (\ref{prod-e-pre}) by \begin{equation}\label{prod-e}
\ee_{(\alpha,\beta,\gamma);I} \otimes \ee_{(\alpha',\beta',\gamma');J} \mapsto e^{-\left(f_{(\alpha,\beta,\gamma);I}(v) + f_{(\alpha',\beta',\gamma');J}(v)\right)}\ee_{(\alpha+\alpha',\beta+\beta',\gamma+\gamma');I+J}.
\end{equation}\par
On the other hand, the composition in $Mo_\cE(P)$ is given by
\begin{equation*}
V_{(\alpha,\beta,\gamma);I}\otimes V_{(\alpha',\beta',\gamma');J} \mapsto e^{-A(\Gamma)}V_{(\alpha+\alpha',\beta+\beta',\gamma+\gamma');I+J},
\end{equation*}
where $\Gamma$ is the unique gradient tree starting at two points $v_{(\alpha,\beta,\gamma);I}\in V_{(\alpha,\beta,\gamma);I}$ and $v_{(\alpha',\beta',\gamma');J}\in V_{(\alpha',\beta',\gamma');J}$ and ending at a point $v_{(\alpha+\alpha',\beta+\beta',\gamma+\gamma');I+J}\in V_{(\alpha+\alpha',\beta+\beta',\gamma+\gamma');I+J}$, and $A(\Gamma)$ is the symplectic area associated to the lift of $\Gamma$. Since the Lagrangian $L(a,b,c)$ is locally the graph of $df_{(a,b,c);I}$ and $f_{(a,b,c);I}=0$ on $V_{(a,b,c)+I}$, we have
\begin{equation}\label{area}
A(\Gamma) = f_{(\alpha,\beta,\gamma);I}(v_{(\alpha+\alpha',\beta+\beta',\gamma+\gamma');I+J}) + f_{(\alpha',\beta',\gamma');J}(v_{(\alpha+\alpha',\beta+\beta',\gamma+\gamma');I+J}).
\end{equation}
By subsection \ref{morse homotopy}, for each generator $V$ of the space of the morphisms in $Mo_\cE(P)$, the gradient trajectories are lines on the adjacent edges of $V$. Using this fact, for $V_{(\alpha,\beta,\gamma);I},\ V_{(\alpha',\beta',\gamma');J}$ and $V_{(\alpha+\alpha',\beta+\beta',\gamma+\gamma');I+J}$, the gradient trees are as follows:
\begin{itemize}
\item In the case of  $V_{(\alpha,\beta,\gamma);I} \cap V_{(\alpha',\beta',\gamma');J} = \emptyset$, there exists a gradient tree if $v_{(\alpha,\beta,\gamma);I}$ and $v_{(\alpha',\beta',\gamma');J}$ are in the same edge. The image of the root edge is $\{v_{(\alpha+\alpha',\beta+\beta',\gamma+\gamma');I+J}\}$.

\item In the case of $V_{(\alpha,\beta,\gamma);I} \cap V_{(\alpha',\beta',\gamma');J} \neq \emptyset$, there exists a gradient tree only if $v_{(\alpha,\beta,\gamma);I} = v_{(\alpha',\beta',\gamma');J} = v_{(\alpha+\alpha',\beta+\beta',\gamma+\gamma');I+J}$. This gradient tree is trivial and the symplectic area equals zero.
\end{itemize}
In Appendix \ref{appendix1}, we give the list of compositions of morphisms and gradient trees. Since $v_{(\alpha+\alpha',\beta+\beta',\gamma+\gamma');I+J}$ coincide with the point $v$ in (\ref{prod-e}), the correspondence $\iota:V_{(a,b,c);I}\mapsto \ee_{(a,b,c);I}$ is compatible with the composition of morphisms:
\begin{equation*}
\fm_2 \circ (\iota\otimes\iota)\left(V_{(\alpha,\beta,\gamma);I}\otimes V_{(\alpha',\beta',\gamma');J}\right) =  \iota \circ \fm_2 \left(V_{(\alpha,\beta,\gamma);I}\otimes V_{(\alpha',\beta',\gamma');J}\right).
\end{equation*}
\end{proof}
By Lemma \ref{lem2-2} and Lemma \ref{lem2-3}, the correspondence $\iota:Mo_\cE(P)\to\cV_\cE'$ is a quasi-isomorphism between DG categories. On the other hand, the DG category $DG_\cE(\Bl)$ is quasi-isomorphic to $\cV_\cE'$ by subsection \ref{HU'}. Thus, the proof of the main theorem (Theorem \ref{main}) for $\Bl$ is completed.

\section{Blow-up of $\C P^2$ at three points.}\label{3}
In this section, we consider the homological mirror symmetry for the blow-up of $\C P^2$ at three points $\Bll$. The organization of this section is parallel to the section 3. For the toric Fano surfaces, we can discuss the homological mirror symmetry by using the Fukaya-Seidel category of the corresponding Landau-Ginzburg potential \cite{AKO,U,seidel}. In particular, for the blow-up of $\C P^2$ at three points, the homological mirror symmetry conjecture is proved by Ueda \cite{U}.

\subsection{Blow-up of $\C P^2$ at three points $\Bll$.}
The blow-up of $\C P^2$ at three points $\Bll$ is defined by
\begin{equation*}
\Bll:=\left\{([\alpha_0:\alpha_1],[\beta_0:\beta_1],[\gamma_0:\gamma_1],[\delta_0:\delta_1:\delta_2]) \in \left(\C P^1\right)^3\times \C P^2\ \middle|
\begin{array}{l}\ \gamma_0\delta_1 = \gamma_1\delta_0 \\ \beta_0\delta_2 = \beta_1\gamma_0 \\ \alpha_0\gamma_2 = \alpha_1\gamma_1
\end{array}\right\}
\end{equation*}
We take the following open covering $\{U_i\}_i$:
\footnotesize
\begin{align*}
U_0  =  \left\{([\alpha_0:\alpha_1],[\beta_0:\beta_1],[\gamma_0:\gamma_1],[\delta_0:\delta_1:\delta_2]) \in \Bll\ \middle|\ \delta_0 \neq 0\ ,\ \gamma_0 \neq \ 0\ ,\ \beta_0 \neq 0 ,\ \alpha_1 \neq 0\right\},\\
U_1  =  \left\{([\alpha_0:\alpha_1],[\beta_0:\beta_1],[\gamma_0:\gamma_1],[\delta_0:\delta_1:\delta_2]) \in \Bll\ \middle|\ \delta_0 \neq 0\ ,\ \gamma_0 \neq \ 0\ ,\ \beta_0 \neq 0 ,\ \alpha_0 \neq 0\right\},\\
U_2  =  \left\{([\alpha_0:\alpha_1],[\beta_0:\beta_1],[\gamma_0:\gamma_1],[\delta_0:\delta_1:\delta_2]) \in \Bll\ \middle|\ \delta_1 \neq 0\ ,\ \gamma_1 \neq \ 0\ ,\ \beta_0 \neq 0 ,\ \alpha_0 \neq 0\right\},\\
U_3  =  \left\{([\alpha_0:\alpha_1],[\beta_0:\beta_1],[\gamma_0:\gamma_1],[\delta_0:\delta_1:\delta_2]) \in \Bll\ \middle|\ \delta_1 \neq 0\ ,\ \gamma_1 \neq \ 0\ ,\ \beta_1 \neq 0 ,\ \alpha_0 \neq 0\right\},\\
U_4  =  \left\{([\alpha_0:\alpha_1],[\beta_0:\beta_1],[\gamma_0:\gamma_1],[\delta_0:\delta_1:\delta_2]) \in \Bll\ \middle|\ \delta_2 \neq 0\ ,\ \gamma_1 \neq \ 0\ ,\ \beta_1 \neq 0 ,\ \alpha_1 \neq 0\right\},\\
U_5  =  \left\{([\alpha_0:\alpha_1],[\beta_0:\beta_1],[\gamma_0:\gamma_1],[\delta_0:\delta_1:\delta_2]) \in \Bll\ \middle|\ \delta_2 \neq 0\ ,\ \gamma_0 \neq \ 0\ ,\ \beta_1 \neq 0 ,\ \alpha_1 \neq 0\right\}.
\end{align*}
\normalsize
The local coordinates in $U_i$ are
\begin{equation*}
\begin{split}
&(u_0,v_0) := \left(\frac{\alpha_0}{\alpha_1},\frac{\beta_1}{\beta_0}\right),\ 
(u_1,v_1) := \left(\frac{\alpha_1}{\alpha_0},\frac{\gamma_1}{\gamma_0}\right),\ 
(u_2,v_2) := \left(\frac{\gamma_0}{\gamma_1},\frac{\beta_1}{\beta_0}\right),\\
&(u_3,v_3) := \left(\frac{\alpha_1}{\alpha_0},\frac{\beta_0}{\beta_1}\right),\ 
(u_4,v_4) := \left(\frac{\gamma_0}{\gamma_1},\frac{\alpha_0}{\alpha_1}\right),\ 
(u_5,v_5) := \left(\frac{\beta_0}{\beta_1},\frac{\gamma_1}{\gamma_0}\right),
\end{split}
\end{equation*}
and the coordinate transformations turn out to be
\begin{equation}\label{ct3}
\begin{array}{ccc}
(u_1,v_1) = \left(\frac{1}{u_0},u_0v_0\right), & (u_2,v_2) = \left(\frac{1}{u_0v_0},v_0\right), & (u_3,v_3)=\left(\frac{1}{u_0},\frac{1}{v_0}\right),\\
(u_4,v_4) = \left(\frac{1}{u_0v_0},u_0\right), & (u_5,v_5) = \left(\frac{1}{v_0},u_0v_0\right). & 
\end{array}
\end{equation}
There exists natural projections $\ \pr_1,\pr_2,\pr_3:\Bll\to\C P^1$ and $\pr_4:\Bll\to\C P^2$.
Using these projections, we define the K\"ahler form of $\Bll$ by
\begin{equation*}
\check{\omega}:= C_1\pr_1^*\left(\omega_{\C P^1}\right) + C_2\pr_2^*\left(\omega_{\C P^1}\right) +  C_3\pr_3^*\left(\omega_{\C P^1}\right), + C_4\pr_4^*\left(\omega_{\C P^2}\right),
\end{equation*}
where each $C_i$ is real positive constant and $\omega_{\C P^n}$ is the Fubini-Study form on $\C P^n$. Correspondingly, the moment map $\mu:\Bll\to \R^2$ is given by
\footnotesize
\begin{equation*}
\begin{split}
&\mu\left([\alpha_0:\alpha_1],[\beta_0:\beta_1],[\gamma_0:\gamma_1],[\delta_0:\delta_1:\delta_2]\right):=\\
&\left(\frac{2C_1|\alpha_0|^2}{|\alpha_0|^2+|\alpha_1|^2} + \frac{2C_3|\gamma_1|}{|\gamma_0|^2 + |\gamma|^2} + \frac{2C_4|\delta_1|^2}{|\delta_0|^2+|\delta_1|^2+|\delta_2|^2} , \frac{2C_1|\alpha_1|^2}{|\alpha_0|^2+|\alpha_1|^2} + \frac{2C_2|\beta_1|^2}{|\beta_0|^2+|\beta_1|^2} + \frac{2C_4|\delta_2|^2}{|\delta_0|^2+|\delta_1|^2+|\delta_2|^2}\right)
\end{split}
\end{equation*}
\normalsize
The image $\mu(\Bll)$ is called the moment polytope associated to $\check{\omega}$, which is a hexagon. Namely, the moment polytope is
\begin{equation*}
P:=\left\{(X,Y)\in\R^2\ \middle|\begin{array}{l}\ 0\leq X \leq 2(C_1+C_3+C_4\\ 0\leq Y \leq 2(C_1+C_2+C_4)\\ 2C_1\leq X+Y \leq 2(C_1+C_2+C_3+C_4)\end{array}\right\}.
\end{equation*}
We denote each edge of $P$ by $E_i$  (see Figure \ref{mp3}). For simplicity, we fix $C_1=C_2=C_3=C_4=1$ since the structure of the category $Mo(P)$ we shall construct is independent of these constants.
\begin{figure}[h]
\center
\begin{tikzpicture}
\draw(0,2)-- node[auto=right]{$E_2$} (2,0) -- node[auto=right]{$E_3$} (6,0)-- node[auto=right]{$E_4$} (6,2)-- node[auto=right]{$E_5$} (2,6)-- node[auto=right]{$E_6$} (0,6)-- node[auto=right]{$E_1$} cycle;
\end{tikzpicture}
\caption{The moment polytope of $\Bll.$}\label{mp3}
\end{figure}

Now, we set
\begin{equation*}
\check{M}:=U_0\cap U_1\cap U_2\cap U_3\cap U_4\cap U_5,\ \ \  B:=\mathrm{Int}P,
\end{equation*}
and we treat $\check{M}$ as a torus fibration $\mu|_{\check{M}}:\check{M}\to B$. Hereafter, we fix $U:=U_0,\ u:=u_0,\ v:=v_0$. Then, $\check{M}$ is equipped with an affine structure by $u=e^{x_1+\ii y_1}$ and $v=e^{x_2+\ii y_2}$, where $y_1$ and $y_2$ are the fiber coordinates of $\check{M}$. The K\"ahler form $\omega$ is expressed as
\begin{equation*}
\begin{split}
\omega = &\frac{4(st(1+t)dx_1\wedge dy_1 - stdx_1\wedge dy_2 - stdx_2\wedge dy_1 +(1+s)tdx_2\wedge dy_2)}{(1+st+t)^2} + \frac{4 sdx_1\wedge dy_1}{(1+s)^2} \\&+ \frac{4 tdx_2\wedge dy_2}{(1+t)^2} + \frac{4 st(dx_1\wedge dy_1 + dx_1\wedge dy_1 + dx_1\wedge dy_2 + dx_2\wedge dy_1 + dx_2\wedge dy_2)}{(1+st)^2} 
\end{split}
\end{equation*}
on $U$, where $s:=|u|^2 = e^{2x_1}$ and $t:=|v^2|= e^{2x_2}$. By this expression, the inverse matrix $\{g^{ij}\}$ of the metric $\{g_{ij}\}$ on $B$ is given by
\begin{equation*}
\begin{pmatrix}
g^{11} & g^{12}\\
g^{21} & g^{22}
\end{pmatrix}
=
\begin{pmatrix}
\frac{4 s}{(1+s)^2} +\frac{4 st}{(1+st)^2} + \frac{4 (1+t)st}{(1+st+t)^2} & -\frac{4 st}{(1+st)^2} -\frac{4 st}{(1+st+t)^2} \\
-\frac{4 st}{(1+st)^2} -\frac{4 st}{(1+st+t)^2} & \frac{4 t}{(1+t)^2} +\frac{4 st}{(1+st)^2} + \frac{4 (1+s)t}{(1+st+t)^2}
\end{pmatrix}.
\end{equation*}
Now, we put $\psi:=\log(1+e^{2 (x_1+x_2)}+e^{2 x_2}) + \log(1+e^{2 x_1}) + \log(1+e^{2 x_2}) + \log(1+e^{2 (x_1+x_2)})$ and then $\frac{\del^2 \psi}{\del x_i\del x_j}=g^{ij}$. Thus the dual coordinate $(x^1,x^2)$ is obtained by
\footnotesize
\begin{equation*}
\begin{split}
(&x^1,x^2)
:=
\left(\frac{\del\psi}{\del x_1} , \frac{\del\psi}{\del x_2}\right)\\
&=
\left(\frac{2e^{2 x_1}}{1+e^{2 x_1}} + \frac{2e^{2(x_1 + x_2)}}{1+e^{2(x_1 + x_2)}} + \frac{2e^{2 (x_1 + x_2)}}{1+e^{2 (x_1 + x_2)}+e^{2 x_2}}, \frac{2e^{2 x_2}}{1+e^{2 x_2}} + \frac{2e^{2 (x_1 + x_2)}}{1+e^{2 (x_1 + x_2)}}  + \frac{2e^{2 x_2}}{1+e^{2 (x_1 + x_2)}+e^{2 x_2}}\right).
\end{split}
\end{equation*}
\normalsize

\subsection{Holomorphic line bundles $\cO(a,b,c,d)$ over $\Bll$.}
Any line bundle over $\Bll$ is constructed from a toric divisor, which is a linear combination of the following divisors:
\begin{align*}
 D_1&=(\alpha_0=\gamma_1=\delta_1=0) ,&
 D_2&=(\beta_1=\gamma_1=\delta_1=\delta_2=0) ,\\
 D_3&=(\alpha_1=\beta_1=\delta_2=0),&
 D_4&=(\alpha_1=\gamma_0=\delta_0=\delta_2=0) ,\\
 D_5&=(\beta_0=\gamma_0=\delta_0=0),& 
 D_6&=(\alpha_0=\beta_0=\delta_0=\delta_1=0).
\end{align*}
The corresponding Cartier divisors are as follows:
\begin{align*}
D_1 :& \{(U_0,u_0),(U_1,1),(U_2,1),(U_3,1),(U_4,1),(U_5,v_5)\},\\
D_2 :& \{(U_0,v_0),(U_1,v_1),(U_2,1),(U_3,1),(U_4,1),(U_5,1)\},\\
D_3 :& \{(U_0,1),(U_1,u_1),(U_2,v_2),(U_3,1),(U_4,1),(U_5,1)\},\\
D_4 :& \{(U_0,1),(U_1,1),(U_2,u_2),(U_3,u_3),(U_4,1),(U_5,1)\},\\
D_5 :& \{(U_0,1),(U_1,1),(U_2,1),(U_3,v_3),(U_4,u_4),(U_5,1)\},\\
D_6 :& \{(U_0,1),(U_1,1),(U_2,1),(U_3,1),(U_4,v_4),(U_5,u_5)\}.
\end{align*}
Using the coordinate transformations (\ref{ct3}), the transition functions of the line bundle $\cO(D_i)$ corresponding to Cartier divisor $D_i$ are as follows:
\begin{align*}
D_1:& \phi_{01}=u_0,  &  &\phi_{02}=u_0,  &  &\phi_{03}=u_0,  &  &\phi_{04}=u_0, & &\phi_{05}=\frac{u_0}{v_5}=\frac{1}{v_0},\\
D_2:& \phi_{01}=\frac{v_0}{v_1}=\frac{1}{u_0},  &  &\phi_{02}=v_0,  &  &\phi_{03}=v_0,  &  &\phi_{04}=v_0, & &\phi_{05}=v_0,\\
D_3:& \phi_{01}=\frac{1}{u_1}=u_0,  &  &\phi_{02}=\frac{1}{v_2}=\frac{1}{v_0},  &  &\phi_{03}=1,  &  &\phi_{04}=1, & &\phi_{05}=1,\\
D_4:& \phi_{01}=1,  &  &\phi_{02}=\frac{1}{u_2}=u_0v_0,  &  &\phi_{03}=\frac{1}{u_3}=u_0,  &  &\phi_{04}=1, & &\phi_{05}=1,\\
D_5:& \phi_{01}=1,  &  &\phi_{02}=1,  &  &\phi_{03}=\frac{1}{v_3}=v_0,  &  &\phi_{04}=\frac{1}{u_4}=u_0v_0, & &\phi_{05}=1,\\
D_6:& \phi_{01}=1,  &  &\phi_{02}=1,  &  &\phi_{03}=1,  &  &\phi_{04}=\frac{1}{v_4}=\frac{1}{u_0}, & &\phi_{05}=\frac{1}{u_5}=v_0.
\end{align*}
Thus, we see that $\cO(D_1+D_2)=\cO(D_4+D_5)$ and $\cO(D_2+D_3)=\cO(D_5+D_6)$, then any line bundle over $\Bll$ is generated by $(D_1+D_6,D_2+D_3,D_4+D_5,D_4+D_5+D_6)$. For more details, see \cite{fulton93toric,coxtor}. \par
On the other hand, we have line bundles $\pr_1^*(\cO_{\C P^1}(1)),\ \pr_2^*(\cO_{\C P^1}(1)),\ \pr_3^*(\cO_{\C P^1}(1))$ and $\pr_4^*(\cO_{\C P^2}(1))$. By comparing the transition functions of these line bundles with those of $\cO(D_i)$, we can identify $ \pr_1^*(\cO_{\C P^1}(1))=\cO(D_1+D_6),\ \pr_2^*(\cO_{\C P^1}(1))=\cO(D_2+D_3),\ \pr_3^*(\cO_{\C P^1}(1))=\cO(D_4+D_5)$ and $\pr_4^*(\cO_{\C P^2}(1))=\cO(D_4+D_5+D_6)$. Furthermore, the connection one-forms of $\pr_1^*(\cO_{\C P^1}(1)),\ \pr_2^*(\cO_{\C P^1}(1)),\ \pr_3^*(\cO_{\C P^1}(1))$ and $\pr_4^*(\cO_{\C P^2}(1))$ are expressed as
\begin{equation*}
\begin{split}
\pr_1^*(A_{\C P^1}) &= -\frac{\overline{u}du}{1+u\overline{u}}=-\frac{s( dx_1+\ii dy_1)}{1+s},\\
\pr_2^*(A_{\C P^1}) &= -\frac{\overline{v}dv}{1+v\overline{v}}=-\frac{t( dx_2+\ii dy_2)}{1+t},\\
\pr_3^*(A_{\C P^1}) &= -\frac{v\overline{uv}du+u\overline{uv}dv}{1+uv\overline{uv}}=-\frac{st((dx_1 + dx_2)+\ii( dy_1 + dy_2))}{1+st},\\
\pr_4^*(A_{\C P^2}) &= -\frac{v\overline{uv}du+(u\overline{uv}+\overline{v})dv}{1+uv\overline{uv}+v\overline{v}}=-\frac{st( dx_1+\ii dy_1) + (1+s)t( dx_2+\ii dy_2)}{1+st+t},
\end{split}
\end{equation*}
on $U$. Thus, for $a,b,c,d\in\Z$, the connection one-form $A_{(a,b,c,d)}$ of $\cO(a,b,c,d):=\cO(a(D_1+D_6) + b(D_2+D_3) +c(D_4+D_5) + d(D_4+D_5+D_6))$ is given by
\begin{equation}\label{Aabcd}
\begin{split}
A_{(a,b,c)}:=&-a\frac{s( dx_1+\ii dy_1)}{1+s} - b\frac{t( dx_2+\ii dy_2)}{1+t} - c\frac{st((dx_1 + dx_2)+\ii( dy_1 + dy_2))}{1+st}\\ &- d\frac{st( dx_1+\ii dy_1) + (1+s)t( dx_2+\ii dy_2)}{1+st+t}
\end{split}
\end{equation}

\subsection{DG category $DG(\Bll)$ and full strongly exceptional collection $\cE$.}\label{DGCP23}
We consider $DG(\Bll)$ as defined in subsection \ref{lineDG}, where the objects are the line bundles $\cO(a,b,c,d)$ with the connection one-form $A_{(a,b,c,d)}$. The DG structure of $DG(\Bll)$ is given by the way described in subsection \ref{lineDG}. Since each $\cO(a,b,c,d)$ is a line bundle, we have
\begin{equation*}
\begin{split}
DG(\Bll)&(\cO(a_1,b_1,c_1,d_1),\cO(a_2,b_2,c_2,d_2)\\
&\simeq DG(\Bll)\left(\cO,\cO(a_2-a_1,b_2-b_1,c_2-c_1,d_2-d_1)\right).
 \end{split}
\end{equation*}
In particular,  the zero-th cohomology of $DG(\Bll)(\cO(0,0,0,0),\cO(a,b,c,d))$ is the space $\Gamma(\Bll,\cO(a,b,c,d))$ of holomorphic global sections.
\begin{equation*}
\begin{split}
H^0(DG(\Bll)&(\cO(a_1,b_1,c_1,d_1),\cO(a_2,b_2,c_2,d_2)))\\
&\simeq H^0(DG(\Bll)(\cO(0,0,0,0),\cO(a_2-a_1,b_2-b_1,c_2-c_1,d_2-d_1))\\
&\simeq \Gamma(\Bll,\cO(a_2-a_1,b_2-b_1,c_2-c_1,d_2-d_1)) .
\end{split}
\end{equation*}
Using the coordinates $(u,v)$ for $U$, the generators of $\Gamma(\Bll,\cO(a,b,c,d))$ are expressed as
\begin{equation*}
\psi_{(i_1,i_2)}:=u^{i_1}v^{i_2},
\end{equation*}
where $0\leq i_1\leq a+c+d,\ 0\leq i_2\leq b+c+d$ and $-a\leq-i_1+i_2\leq b+d$. \par
As in subsection \ref{DGCP22}, we next consider a full strongly exceptional collection of the derived category of coherent sheaves $D^b(Coh(\Bll))\simeq Tr(DG(\Bll))$. The exceptional divisors of blow-ups are $D_2,D_4$ and $D_6$. The corresponding line bundles are
\begin{equation*}
\begin{split}
\cO(D_2) &= \cO(-D_1 + D_4 + D_5) = \cO(-1,0,0,1),\\
\cO(D_4) &= \cO( D_4 + D_5 + D_6 - (D_5 + D_6)) = \cO(0,0,-1,1),\\
\cO(D_6) &= \cO( D_4 + D_5 + D_6 - (D_4 + D_5)) = \cO(0,-1,0,1),
\end{split}
\end{equation*}
respectively. Also, the generator of $\mathrm{Pic}(\C P^2)$ is $\cO(0,0,0,1)$.
We put
\begin{equation*}
\cE:=(\cO,\cO(-1,0,0,1),\cO(0,-1,0,1),\cO(0,0,-1,1),\cO(0,0,0,1),\cO(0,0,0,2)).
\end{equation*}
We can identify this collection $\cE$ with a full strongly exceptional collection of Hille-Perling's construction in \cite[Theorem 5.7]{hille-perling11}. We denote by $DG_\cE(\Bll)$ be the full subcategory of $DG(\Bll)$ consisting of $\cE$.

\subsection{Lagrangian sections $L(a,b,c,d)$.}\label{lagabcd}
We discuss Lagrangian sections $L(a,b,c,d)$ of the dual torus fibration $M\to B$ corresponding to the line bundle $\cO(a,b,c,d)$. Twisting the fibers of $\cO(a,b,c,d)$ by the isomorphisms$\Psi_{(a,b,c,d)}:=(1+s)^{\frac{a}{2}}(1+t)^{\frac{b}{2}}(1+st)^{\frac{c}{2}}(1+st+t)^{\frac{d}{2}}$, we can remove the $dx$ term from $A_{(a,b,c,d)}$ as follows:
\begin{equation*}
\begin{split}
&\Psi_{(a,b,c,d)}^{-1}(d+A_{(a,b,c,d)})\Psi_{(a,b,c,d)}\\ &\ = d - \ii\left( a\frac{s}{1+s} + c\frac{st}{1+st} + d\frac{st}{1+st+t}\right)dy_1 - \ii\left( b\frac{t}{1+t} + c\frac{st}{1+st} + d\frac{st+t}{1+st+t}\right)dy_2,
\end{split}
\end{equation*}
where $s=e^{2x_1}$ and $t=e^{2x_2}$. Thus, the Lagrangians section $L(a,b,c,d)$ corresponding to $\cO(a,b,c,d)$ is obtained as the coefficients of $dy$ terms:
\begin{equation}\label{Lagrange-section3}
\begin{pmatrix}
y^1 \\ y^2
\end{pmatrix}
=2\pi
\begin{pmatrix}
 a\frac{s}{1+s} +  c\frac{st}{1+st} + d\frac{st}{1+st+t} \\  b\frac{t}{1+t} +  c\frac{t}{1+st} + d\frac{st+t}{1+st+t}
\end{pmatrix}.
\end{equation}
By this expression, we can extend $L(a,b,c,d)$ on $B$ to that on $P$ smoothly.
The potential function of this Lagrangian section is given by
\begin{equation*}
f = \pi a\log(1+s) + \pi b\log(1+t) + \pi c\log(1+st) + \pi d\log(1+st+t)
\end{equation*}
The collection of the Lagrangian sections corresponding to the full strongly exceptional collection $\cE$ given in subsection \ref{DGCP23} is denoted by the same symbol $\cE$, i.e.,
\begin{equation*}
\cE=\left(L(0,0,0,0),L(-1,0,0,1),L(0,-1,0,1),L(0,0,-1,1),L(0,0,0,1),L(0,0,0,2)\right).
\end{equation*}

\subsection{Cohomologies $H(\cV')$.}\label{HV'}
Let $\cV$ be a DG category of holomorphic line bundles over $\check{M}$. We consider the faithful functor $\cI:DG(\Bll)\to\cV$. Let $\widetilde{\cO}(a,b,c,d)$ denote the line bundle $\cO(a,b,c,d)$ twisted by $\Psi_{(a,b,c,d)}$. This functor assigns $\cO(a,b,c,d)$ to $\widetilde{\cO}(a,b,c,d)$. Also, each generator $\psi_{(i_1,i_2)}\in DG(\Bll)(\cO,\cO(a,b,c,d))$ is sent to be $\Psi^{-1}_{(a,b,c,d)}\psi_{(i_1,i_2)}\in\cV(\widetilde{\cO},\widetilde{\cO}(a,b,c,d))$, i.e.,
\begin{equation}\label{pre-e2}
\begin{split}
\Psi^{-1}_{(a,b,c,d)}\psi_{(i_1,i_2)} &= (1+s)^{-\frac{a}{2}}(1+t)^{-\frac{b}{2}}(1+st)^{-\frac{c}{2}}(1+st+t)^{-\frac{d}{2}}u^{i_1}v^{i_2}\\
 &=(1+s)^{-\frac{a}{2}}(1 + t)^{-\frac{b}{2}}(1 + st)^{-\frac{c}{2}}(1+st+t)^{-\frac{d}{2}}s^\frac{i_1}{2}t^\frac{i_2}{2}e^{\ii(i_1y_1 + i_2y_2)}.
\end{split}
\end{equation}
The images are denoted by $\cV':=\cI(DG(\Bll))$ and $\cV_\cE':=\cI(DG_\cE(\Bll))$, respectively. Then, the basis of $H^0(\cV(\widetilde{\cO},\widetilde{\cO}(a,b,c,d))$ are given by (\ref{pre-e2}). If the functions $\Psi^{-1}_{(a,b,c,d)}\psi_{(i_1,i_2)}$ on $B$ extend to that on $P$ smoothly, then we rescale each basis $\Psi^{-1}_{(a,b,c,d)}\psi_{(i_1,i_2)}$ by multiplying a positive number and denote it by $\ee_{(a,b,c,d);(i_1,i_2)}$ so that
\begin{equation*}
\max_{x\in P}|\ee_{(a,b,c,d);(i_1,i_2)}(x)|=1.
\end{equation*}

\subsection{Weighted Morse homotopy $Mo_{\cE}(P)$.}\label{morse homotopy2}
For the moment polytope $P$ of $\Bll$, we construct the full subcategory $Mo_\cE(P)\subset Mo(P)$ consisting of $\cE$, where $Mo(P)$ is defined in subsection \ref{Mo(P)}. The objects of $Mo(P)$ are Lagrangian sections $L(a,b,c,d)$ obtained in subsection \ref{lagabcd}. Since we have
\begin{equation*}
Mo(P)(L(a_1,b_1,c_1,d_1),L(a_2,b_2,c_2,d_2)) \simeq Mo(P)(L(0,0,0),L(a_2-a_1,b_2-b_1,c_2-c_1,d_2-d_1)),
\end{equation*}
we concentrate on computing the space $Mo_\cE(P)(L(0,0,0,0),L(a,b,c,d))$. As considered in subsection \ref{glag}, the intersections of $L(0,0,0,0)$ and $L(a,b,c,d)$ are expressed as
\begin{equation}\label{eq3}
\begin{split}
2\pi
\begin{pmatrix}
i_1 \\ i_2
\end{pmatrix}
=
2\pi\begin{pmatrix}
 a\frac{s}{1+s} +  c\frac{st}{1+st} + d\frac{st}{1+st+t} \\  b\frac{t}{1+t} +  c\frac{t}{1+st} + d\frac{st+t}{1+st+t}
\end{pmatrix}
\end{split}
\end{equation}
in the covering space of $\pi:\bar{M}\to P$, where $(i_1,i_2)\in\Z^2$, $s=e^{2x_1}$ and $t=e^{2x_2}$. If there exists a nonempty intersection, then we set $V_{(a,b,c,d);(i_1,i_2)}:=\pi(L(0,0,0,0)\cap L(a,b,c,d))$. We check  that $V_{(a,b,c,d);(i_1,i_2)}$ satisfies the conditions (M1) and (M2) given in subsection \ref{Mo(P)}. In this cace, the gradient vector field associated to $V_{(a,b,c,d);(i_1,i_2)}$ is of the form
\footnotesize
\begin{equation}\label{grad3}
2\pi\left(a\frac{s}{1+s} +  c\frac{st}{1+st} + d\frac{st}{1+st+t} - i_1\right)\frac{\del}{\del x^1} + 2\pi\left(b\frac{t}{1+t} +  c\frac{st}{1+st} + + d\frac{st+t}{1+st+t} - i_2 \right)\frac{\del}{\del x^2}.
\end{equation}
\normalsize
On the other hand, for the space of the opposite directional morphisms, we have
\begin{equation*}
Mo(P)(L(a,b,c,d),L(0,0,0,0)) \cong Mo(P)(L(0,0,0,0),L(-a,-b,-c,-d)).
\end{equation*}
Thus, the connected component $V_{(a,b,c,d);(i_1,i_2)}$ coincides with the connected component $V_{(-a,-b,-c,-d);(-i_1,-i_2)}$. The gradient vector field associated with $V_{(-a,-b,-,c,-d);(-i_1,-i_2)}$ is the opposite direction of (\ref{grad3}).\par

\begin{figure}[h]
\center
\begin{tikzpicture}
\draw(0,2)-- node[auto=left]{$E_2$} node[auto=right]{$0\leq s \leq\infty,t=0$} (2,0) --node[auto=left]{$E_3$} node[auto=right]{$s=\infty,t=0,0\leq st\leq\infty$} (6,0)-- node[auto=left]{$E_4$} node[auto=right]{$s=\infty,t\leq\infty,st=\infty$} (6,2)-- node[auto=left]{$E_5$} node[auto=right]{$s\leq\infty,t=\infty,st=\infty$} (2,6)--node[auto=left]{$E_6$} node[auto=right]{$s=0,t=\infty,0\leq st\leq\infty$} (0,6)--node[auto=left]{$E_1$} node[auto=right]{$s=0,0\leq t\leq\infty$} cycle;
\draw(6,2) node[right]{$s=t=st=\infty$};
\end{tikzpicture}
\caption{The moment polytope$\Bll$.}
\label{poly3}
\end{figure}
For all $L$ and $L'$, we compute generators of $Mo_\cE(P)(L,L')$.

\begin{itemize}
\item $Mo_\cE(P)\left(L(0,0,0,0),L(-1,0,0,1)\right)$:\ \par
The intersection of $L(0,0,0,0)$ and $L(-1,0,0,1)$ is expressed as
\begin{equation*}
i_1 = -\frac{s}{1+s} + \frac{st}{1+st+t},\ \ \ 
i_2 = \frac{st+t}{1+st+t}.
\end{equation*}
Then, since $(x^1,x^2)=\left(\frac{2s}{1+s} + \frac{2st}{1+st} + \frac{2st}{1+st+t}, \frac{2t}{1+t} + \frac{2st}{1+st}  + \frac{2(st+t)}{1+st+t}\right)$,
we obtain the connected components
\begin{equation*}
V_{(-1,0,0,1);(0,0)}   = \{(0,2)\},\ \ \ 
V_{(-1,0,0,1);(-1,0)}  = \{(2,0)\},\ \ \ 
V_{(-1,0,0,1);(0,1)}   = E_4 \cup E_5 \cup E_6,
\end{equation*}
where $E_i$ is the edge of the moment polytope associated to each toric divisor $D_i$ (see Figure \ref{poly3}). For each $(i_1,i_2)$, the gradient vector field is given by
\begin{equation*}
2\pi\left(-\frac{s}{1+s} + \frac{st}{1+st+t} - i_1\right)\frac{\del}{\del x^1} + 2\pi\left(\frac{st+t}{1+st+t} - i_2 \right)\frac{\del}{\del x^2}.
\end{equation*}
The degrees of the connected components are $|V_{(-1,0,0,1);(0,1)}|=0$ and $|V_{(-1,0,0,1);(0,0)}|=|V_{(-1,0,0,1);(-1,0)}|=1$. In fact, for $V_{(-1,0,0,1);(0,0)}$ and $V_{(-1,0,0,1);(-1,0)}$, the stable manifold is $\{(X,Y)\in\R^2\ |\ X+Y=2\}$. Since the condition (M2) in subsection \ref{Mo(P)} the components $V_{(-1,0,0,1);(0,0)}$ and $V_{(-1,0,0,1);(-1,0)}$ can not be the generators. Thus, the connected component $V_{(-1,0,0,1);(0,1)}$ is the only generator of $Mo_\cE(P)\left(L(0,0,0,0),L(-1,0,0,1)\right)$ of degree zero:
\begin{equation*}
Mo_\cE(P)\left(L(0,0,0,0),L(-1,0,0,1)\right) = \C\cdot V_{(-1,0,0,1);(0,1)},\ \ \ |V_{(-1,0,0,1);(0,1)}|=0.
\end{equation*}
On the other hand, for the space $Mo_\cE(P)\left(L(-1,0,0,1),L(0,0,0,0)\right)$, the degree of any connected components is positive. For this reason, we have
\begin{equation*}
Mo_\cE(P)\left(L(-1,0,0,1),L(0,0,0,0)\right) = 0.
\end{equation*}\par

\item $Mo_\cE(P)\left(L(0,0,0,0),L(0,-1,0,1)\right)$:\ \par
The intersection of $L(0,0,0,0)$ and $L(-1,0,0,1)$ is expressed as
\begin{equation*}
i_1 = \frac{st}{1+st+t},\ \ \ 
i_2 = -\frac{t}{1+t} + \frac{st+t}{1+st+t}.
\end{equation*}
Then, we obtain the connected components
\begin{equation*}
V_{(0,-1,0,1);(0,0)}  = E_1 \cup E_2 \cup E_6,\ \ \ 
V_{(0,-1,0,1);(1,0)}  = \{(6,2)\},\ \ \ 
V_{(0,-1,0,1);(1,1)}  = \{(6,0)\},
\end{equation*}
and the corresponding gradient vector field
\begin{equation*}
2\pi\left(\frac{st}{1+st+t} - i_1\right)\frac{\del}{\del x^1} + 2\pi\left(-\frac{t}{1+t} + \frac{st+t}{1+st+t} - i_2 \right)\frac{\del}{\del x^2}.
\end{equation*}
The degrees of the connected components are $|V_{(0,-1,0,1);(0,0)}|=0$ and $|V_{(0,-1,0,1);(1,0)}|=|V_{(0,-1,0,1);(1,1)}|=1$. The components $V_{(0,-1,0,1);(1,0)}$ and $V_{(0,-1,0,1);(1,1)}$ do not satisfy the condition (M2) and can not be the generators. Thus, the connected component $V_{(0,-1,0,1);(0,0)}$ is the only generator of $Mo_\cE(P)\left(L(0,0,0,0),L(0,-1,0,1)\right)$ of degree zero:
\begin{equation*}
Mo_\cE(P)\left(L(0,0,0,0),L(0,-1,0,1)\right) = \C\cdot V_{(0,-1,0,1);(0,1)},\ \ \ |V_{(0,-1,0,1);(0,0)}|=0.
\end{equation*}
On the other hand, for the space $Mo_\cE(P)\left(L(0,-1,0,1),L(0,0,0,0)\right)$, the degree of any connected components is positive. For this reason, we have
\begin{equation*}
Mo_\cE(P)\left(L(0,-1,0,1),L(0,0,0,0)\right) = 0.
\end{equation*}\par

\item $Mo_\cE(P)\left(L(0,0,0,0),L(0,0,-1,1)\right)$:\ \par
The intersection of $L(0,0,0,0)$ and $L(0,0,-1,1)$ is expressed as
\begin{equation*}
i_1 = -\frac{st}{1+st} + \frac{st}{1+st+t},\ \ \ 
i_2 = -\frac{st}{1+st} + \frac{st+t}{1+st+t}.
\end{equation*}
Then, we obtain the connected components
\begin{equation*}
V_{(0,0,-1,1);(0,0)}   = E_2 \cup E_3 \cup E_4,\ \ \ 
V_{(0,0,-1,1);(0,1)}   = \{(0,6)\},\ \ \ 
V_{(0,0,-1,1);(-1,0)}  = \{(2,6)\},
\end{equation*}
and the corresponding gradient vector field
\begin{equation*}
2\pi\left( - \frac{st}{1+st} + \frac{st}{1+st+t} - i_1\right)\frac{\del}{\del x^1} + 2\pi\left(-\frac{st}{1+st} + \frac{st+t}{1+st+t} - i_2 \right)\frac{\del}{\del x^2}.
\end{equation*}
The degrees of the connected components are $|V_{(0,0,-1,1);(0,0)}|=0$ and $|V_{(0,0,-1,1);(-1,0)}|=|V_{(0,0,-1,1);(0,1)}|=1$. The components $V_{(0,0,-1,1);(-1,0)}$ and $V_{(0,0,-1,1);(0,1)}$ do not satisfy the condition (M2) and can not be the generators. Thus, the connected component $V_{(0,0,-1,1);(0,0)}$ is the only generator of $Mo_\cE(P)\left(L(0,0,0,0),L(0,0,-1,1)\right)$ of degree zero:
\begin{equation*}
Mo_\cE(P)\left(L(0,0,0,0),L(0,0,-1,1)\right) = \C\cdot V_{(0,0,-1,1);(0,0)},\ \ \ |V_{(0,0,-1,1);(0,0)}|=0,
\end{equation*}
On the other hand, for the space $Mo_\cE(P)\left(L(0,-1,0,1),L(0,0,0,0)\right)$, the degree of any connected components is positive. For this reason, we have
\begin{equation*}
Mo_\cE(P)\left(L(0,0,-1,1),L(0,0,0,0)\right) = 0.
\end{equation*}\par

\item $Mo_\cE(P)\left(L(0,0,0,0),L(0,0,0,1)\right)$ and $Mo_\cE(P)\left(L(0,0,0,1),L(0,0,0,2)\right)$:\ \par
The intersection of $L(0,0,0,0)$ and $L(0,0,0,1)$ is expressed as
\begin{equation*}
i_1 = \frac{st}{1+st+t},\ \ \ 
i_2 = \frac{st+t}{1+st+t}.
\end{equation*}
Then, we obtain the connected components
\begin{equation*}
V_{(0,0,0,1);(0,0)}   = E_2,\ \ \ 
V_{(0,0,0,1);(0,1)}   = E_6,\ \ \ 
V_{(0,0,0,1);(1,1)}   = E_4,
\end{equation*}
and the corresponding gradient vector field
\begin{equation*}
2\pi\left( \frac{st}{1+st+t} - i_1\right)\frac{\del}{\del x^1} + 2\pi\left( \frac{st+t}{1+st+t} - i_2 \right)\frac{\del}{\del x^2}.
\end{equation*}
Since the degrees of these connected components are zero, we have
\begin{equation*}
Mo_\cE(P)\left(L(0,0,0,0),L(0,0,0,1)\right) = \C\cdot V_{(0,0,0,1);(0,0)} \oplus \C\cdot V_{(0,0,0,1);(0,1)} \oplus \C\cdot V_{(0,0,0,1);(1,1)}.
\end{equation*}
Furthermore, we have
\begin{equation*}
\begin{split}
Mo_\cE(P)\left(L(0,0,0,1),L(0,0,0,2)\right) &\cong Mo_\cE(P)\left(L(0,0,0,0),L(0,0,0,1)\right)\\
&\cong\C\cdot V_{(0,0,0,1);(0,0)} \oplus \C\cdot V_{(0,0,0,1);(0,1)} \oplus \C\cdot V_{(0,0,0,1);(1,1)}.
\end{split}
\end{equation*}\par
On the other hand, the space $Mo_\cE(P)\left(L(0,0,0,1),L(0,0,0,0)\right)$ is trivial. In fact, the degree of $V_{(0,0,0,-1);I}$ is greater than 0 and none of the point $v_{(0,0,0,-1);I}\in V_{(0,0,0,-1);I}$ is an interior point of $P$. Thus, we have
\begin{equation*}
\begin{split}
Mo_\cE(P)\left(L(0,0,0,1),L(0,0,0,0)\right) &= 0,\\
Mo_\cE(P)\left(L(0,0,0,2),L(0,0,0,1)\right) &= 0.
\end{split}
\end{equation*}

\item $Mo_\cE(P)\left(L(0,0,0,0),L(0,0,0,2)\right)$:\ \par
The intersection of $L(0,0,0,0)$ and $L(0,0,0,2)$ is expressed as
\begin{equation*}
i_1 = \frac{2st}{1+st+t},\ \ \ 
i_2 = \frac{2(st+t)}{1+st+t}.
\end{equation*}
Then, we obtain the connected components
\begin{align*}
V_{(0,0,0,2);(0,2)} &= E_6,\\
V_{(0,0,0,2);(0,1)} &= \{(0,4)\} ,& V_{(0,0,0,2);(1,2)} &= \{(4,4)\},\\ 
V_{(0,0,0,2);(0,0)} &= E_1, & V_{(0,0,0,2);(1,1)} &= \{(4,0)\} ,& V_{(0,0,2);(2,2)} &= E_4,
\end{align*}
and the corresponding gradient vector field
\begin{equation*}
2\pi\left( \frac{2st}{1+st+t} - i_1\right)\frac{\del}{\del x^1} + 2\pi\left( \frac{2(st+t)}{1+st+t} - i_2 \right)\frac{\del}{\del x^2}.
\end{equation*}
Since the degrees of these connected components are zero, we have
\begin{equation*}
Mo_\cE(P)\left(L(0,0,0,0),L(0,0,0,2)\right) =  \bigoplus_{\substack{ 0 \leq i_1 \leq i_2 \leq 2}}\C\cdot V_{(0,0,0,2);(i_1,i_2)}.
\end{equation*}\par
On the other hand, the space $Mo_\cE(P)\left(L(0,0,0,2),L(0,0,0,0)\right)$ is trivial. In fact, the degree of $V_{(0,0,0,-2);I}$ is greater than 0 and none of the point $v_{(0,0,0,-2);I}\in V_{(0,0,0,-2);I}$ is an interior point of $P$. Thus, we have
\begin{equation*}
Mo_\cE(P)\left(L(0,0,0,2),L(0,0,0,0)\right) = 0.
\end{equation*}

\item $Mo_\cE(P)\left(L(-1,0,0,1),L(0,-1,0,1)\right)$:\ \par
Since $Mo_\cE(P)\left(L(-1,0,0,1),L(0,-1,0,1)\right)\cong Mo_\cE(P)\left(L(0,0,0,0),L(1,-1,0,0)\right)$, we consider $Mo_\cE(P)\left(L(0,0,0,0),L(1,-1,0,0)\right)$. The intersection of $L(0,0,0,0)$ and $L(1,-1,0,0)$ is expressed as
\begin{equation*}
i_1 = \frac{s}{1+s},\ \ \ 
i_2 = -\frac{t}{1+t}.
\end{equation*}
Then, we obtain the connected components
\begin{align*}
V_{(1,-1,0,0);(0,-1)} &= E_6  ,    & V_{(1,-1,0,0);(1,-1)} &= \{(6,2)\},\\
V_{(1,-1,0,0);(0,0)} &= \{(0,2)\}, & V_{(1,-1,0,0);(1,0)} &= E_3,
\end{align*}
and the corresponding gradient vector field
\begin{equation*}
2\pi\left(\frac{s}{1+s} - i_1\right)\frac{\del}{\del x^1} + 2\pi\left(-\frac{t}{1+t} - i_2 \right)\frac{\del}{\del x^2}.
\end{equation*}
For $(i_1,i_2)=(0,0)$ and $(1,-1)$, the corresponding stable manifolds is as follows:
\begin{equation*}
S_{(0,2)} = \{(0,Y)\},\ \ \ S_{(6,2)}=\{(6,Y)\},
\end{equation*}
respectively. These connected components are degree one and do not satisfy the condition (M2). For $(i_1,i_2)=(1,0)$ and $(0,-1)$, the dimensions of the stable manifolds are not constant. Now, we decompose $V_{(1,-1,0,0);(1,0)}$ and $V_{(1,-1,0,0);(0,-1)}$ into the left half and the right half:
\begin{equation*}
\begin{split}
V_{(1,-1,0,0);(1,0)} &= V_{(-1,1,0,0);(1,0)}^- \cup V_{(-1,1,0,0);(1,0)}^+ ,\\
V_{(1,-1,0,0);(1,0)}^- &:= \{(X,0)\in P\ |\ 2\leq X\leq4\},\\
V_{(1,-1,0,0);(1,0)}^+ &:= \{(X,0)\in P\ |\ 4<X\leq6\},\\
V_{(1,-1,0,0);(0,-1)} &= V_{(-1,1,0,0);(0,-1)}^- \cup V_{(-1,1,0,0);(0,-1)}^+ ,\\
V_{(1,-1,0,0);(0,-1)}^- &:= \{(X,6)\in P\ |\ 0\leq X\leq1\},\\
V_{(1,-1,0,0);(0,-1)}^+ &:= \{(X,6)\in P\ |\ 1<X\leq2\}.
\end{split}
\end{equation*}
Then,we have $\dim S_v = 0$ (resp. $\dim S_v =1$) if $v\in V_{(1,-1,0,0);(1,0)}^-$ or $v\in V_{(1,-1,0,0);(0,-1)}^+$ (resp. $v\in V_{(1,-1,0,0);(1,0)}^+$ or $v\in V_{(1,-1,0,0);(0,-1)}^-$).\par
For the space of opposite directional morphisms, the degree of each connected components is similar to these above. Thus, we have
\begin{equation*}
\begin{split}
Mo_\cE(P)\left(L(-1,0,0,1),L(0,-1,0,1)\right) &= 0,\\
Mo_\cE(P)\left(L(0,-1,0,1),L(-1,0,0,1)\right) &= 0.
\end{split}
\end{equation*}

\item $Mo_\cE(P)\left(L(-1,0,0,1),L(0,0,-1,1)\right)$:\ \par
Since $Mo_\cE(P)\left(L(-1,0,0,1),L(0,0,-1,1)\right)\cong Mo_\cE(P)\left(L(0,0,0,0),L(1,0,-1,0)\right)$, we consider $Mo_\cE(P)\left(L(0,0,0,0),L(1,0,-1,0)\right)$. The intersection of $L(0,0,0,0)$ and $L(1,0,-1,0)$ is expressed as
\begin{equation*}
i_1 = \frac{s}{1+s} - \frac{st}{1+st},\ \ \ 
i_2 = -\frac{st}{1+st}.
\end{equation*}
Then, we obtain the connected components
\begin{align*}
V_{(1,0,-1,0);(0,0)} &= E_1 ,      & V_{(1,0,-1,0);(-1,-1)} &= \{(2,6)\},\\
V_{(1,0,-1,0);(1,0)} &= \{(2,0)\}, & V_{(1,0,-1,0);(0,-1)} &= E_4,
\end{align*}
and the corresponding gradient vector field
\begin{equation*}
2\pi\left(\frac{s}{1+s} - \frac{st}{1+st} - i_1\right)\frac{\del}{\del x^1} + 2\pi\left(- \frac{st}{1+st} - i_2 \right)\frac{\del}{\del x^2}.
\end{equation*}
For $(i_1,i_2)=(1,0)$ and $(-1,-1)$, the corresponding stable manifolds is as follows:
\begin{equation*}
S_{(2,0)} = \{(X,0)\},\ \ \ S_{(2,6)}=\{(X,6)\},
\end{equation*}
respectively. These connected components are degree one and do not satisfy the condition (M2). For $(i_1,i_2)=(1,0)$ and $(-1,-1)$, the dimensions of the stable manifolds are not constant. Now, we decompose $V_{(1,0,-1,0);(0,0)}$ and $V_{(1,0,-1,0);(0,-1)}$ into the upper half and the lower half:
\begin{equation*}
\begin{split}
V_{(1,0,-1,0);(0,0)} &= V_{(1,0,-1,0);(0,0)}^- \cup V_{(1,0,-1,0);(0,0)}^+ ,\\
V_{(1,0,-1,0);(0,0)}^- &:= \{(0,Y)\in P\ |\ 2\leq X\leq4\},\\
V_{(1,0,-1,0);(0,0)}^+ &:= \{(0,Y)\in P\ |\ 4<X\leq6\},\\
V_{(1,0,-1,0);(0,-1)} &= V_{(1,0,-1,0);(0,-1)}^- \cup V_{(1,0,-1,0);(0,-1)}^+, \\
V_{(1,0,-1,0);(0,-1)}^- &:= \{(6,Y)\in P\ |\ 0\leq Y\leq1\},\\
V_{(1,0,-1,0);(0,-1)}^+ &:= \{(6,Y)\in P\ |\ 1<Y\leq2\}.
\end{split}
\end{equation*}
Then,we have $\dim S_v = 0$ (resp. $\dim S_v =1$) if $v\in V_{(1,0,-1,0);(0,0)}^-$ or $v\in V_{(1,0,-1,0);(0,-1)}^+$ (resp. $v\in V_{(1,0,-1,0);(0,0)}^+$ or $v\in V_{(1,0,-1,0);(0,-1)}^-$).\par
For the space of opposite directional morphisms, the degree of each connected components is similar to these above. Thus, we have
\begin{equation*}
\begin{split}
Mo_\cE(P)\left(L(-1,0,0,1),L(0,-1,0,1)\right) &= 0,\\
Mo_\cE(P)\left(L(0,-1,0,1),L(-1,0,0,1)\right) &= 0.
\end{split}
\end{equation*}

\item $Mo_\cE(P)\left(L(-1,0,0,1),L(0,0,0,1)\right)$:\ \par
Since $Mo_\cE(P)\left(L(-1,0,0,1),L(0,0,0,1)\right) \cong Mo_\cE(P)\left(L(0,0,0,0),L(1,0,0,0)\right)$, we consider $Mo_\cE(P)\left(L(0,0,0,0),L(1,0,0,0)\right)$. The intersection of $L(0,0,0,0)$ and $L(1,0,0,0)$ is expressed as
\begin{equation*}
i_1 = \frac{s}{1+s},\ \ \ 
i_2 = 0.
\end{equation*}
Then, we obtain the connected components
\begin{align*}
V_{(1,0,0,0);(0,0)} &= E_1 \cup E_6, & V_{(1,0,0,0);(1,0)}  &= E_3 \cup E_4 ,
\end{align*}
and the corresponding gradient vector field
\begin{equation*}
2\pi\left(\frac{s}{1+s} - i_1\right)\frac{\del}{\del x^1}.
\end{equation*}
Since the degrees of these connected components are zero, we have
\begin{equation*}
\begin{split}
Mo_\cE(P)\left(L(-1,0,0,1),L(0,0,0,1)\right) &\cong Mo_\cE(P)\left(L(0,0,0,0),L(1,0,0,0)\right)\\
&= \C\cdot V_{(1,0,0,0);(0,0)} \oplus \C\cdot V_{(1,0,0,0);(1,0)}.
\end{split}
\end{equation*}\par
On the other hand, the space $Mo_\cE(P)\left(L(0,0,0,1),L(-1,0,0,1)\right)$ is trivial. In fact, the degree of $V_{(-1,0,0,0);I}$ is greater than 0 and none of the point $v_{(-1,0,0,0);I}\in V_{(-1,0,0,0);I}$ is an interior point of $P$. Thus, we have
\begin{equation*}
Mo_\cE(P)\left(L(0,0,0,1),L(-1,0,0,1)\right) = 0.
\end{equation*}

\item $Mo_\cE(P)\left(L(-1,0,0,1),L(0,0,0,2)\right)$:\ \par
Since $Mo_\cE(P)\left(L(-1,0,0,1),L(0,0,0,2)\right) \cong Mo_\cE(P)\left(L(0,0,0,0),L(1,0,0,1)\right)$, we consider $Mo_\cE(P)\left(L(0,0,0,0),L(1,0,0,1)\right)$. The intersection of $L(0,0,0,0)$ and $L(1,0,0,1)$ is expressed as
\begin{equation*}
i_1 = \frac{s}{1+s} + \frac{st}{1+st+t},\ \ \ 
i_2 = \frac{st+t}{1+st+t}.
\end{equation*}
Then, we obtain the connected components
\begin{align*}
V_{(1,0,0,1);(0,1)} &= E_6   ,     & V_{(1,0,0,1);(1,1)} &= \{(4,4)\}, \\
V_{(1,0,0,1);(0,0)} &= \{(0,2)\}  ,& V_{(1,0,0,1);(1,0)} &= \{(2,0)\}, &  V_{(1,0,0,1);(2,1)} &= E_4,
\end{align*}
and the corresponding gradient vector field
\begin{equation*}
2\pi\left(\frac{s}{1+s} + \frac{st}{1+st+t} - i_1\right)\frac{\del}{\del x^1} + 2\pi\left( \frac{st+t}{1+st} - i_2 \right)\frac{\del}{\del x^2}.
\end{equation*}
Since the degrees of these connected components are zero, we have
\begin{equation*}
\begin{split}
Mo_\cE(P)\left(L(-1,0,0,1),L(0,0,0,2)\right) &\cong Mo_\cE(P)\left(L(0,0,0,0),L(1,0,0,1)\right)\\
&= \bigoplus_{\substack{0\leq i_2 \leq1 \\ 0 \leq i_1 \leq 2-i_2}}\C\cdot V_{(1,0,0,1);(i_1,i_2)}.
\end{split}
\end{equation*}\par
On the other hand, the space $Mo_\cE(P)\left(L(0,0,0,2),L(-1,0,0,1)\right)$ is trivial. In fact, the degree of $V_{(-1,0,0,-1);I}$ is greater than 0 and none of the point $v_{(-1,0,0,-1);I}\in V_{(-1,0,0,-1);I}$ is an interior point of $P$. Thus, we have
\begin{equation*}
Mo_\cE(P)\left(L(0,0,0,2),L(-1,0,0,1)\right) = 0.
\end{equation*}

\item $Mo_\cE(P)\left(L(0,-1,0,1),L(0,0,-1,1)\right)$:\ \par
Since $Mo_\cE(P)\left(L(0,-1,0,1),L(0,0,-1,1)\right) \cong Mo_\cE(P)\left(L(0,0,0,0),L(0,1,-1,0)\right)$, we consider $Mo_\cE(P)\left(L(0,0,0,0),L(0,1,-1,0)\right)$. The intersection of $L(0,0,0,0)$ and $L(0,1,-1,0)$ is expressed as
\begin{equation*}
i_1 = -\frac{st}{1+st},\ \ \ 
i_2 = \frac{t}{1+t} - \frac{st}{1+st}.
\end{equation*}
Then, we obtain the connected components
\begin{align*}
V_{(0,1,-1,0);(0,1)} &= \{(0,6)\}, & V_{(0,1,-1,0);(-1,0)}  &= E_5 ,\\
V_{(0,1,-1,0);(0,0)} &= E_2 ,      & V_{(0,1,-1,0);(-1,-1)} &= \{(6,0)\},
\end{align*}
and the corresponding gradient vector field
\begin{equation*}
2\pi\left(- \frac{st}{1+st} - i_1\right)\frac{\del}{\del x^1} + 2\pi\left(\frac{t}{1+t} - \frac{st}{1+st} - i_2 \right)\frac{\del}{\del x^2}.
\end{equation*}
For $(i_1,i_2)=(0,1)$ and $(-1,-1)$, the corresponding stable manifolds is as follows:
\begin{equation*}
S_{(6,0)} = \{(X,0)\},\ \ \ S_{(0,6)}=\{(X,6)\},
\end{equation*}
respectively. These connected components are degree one and do not satisfy the condition (M2). For $(i_1,i_2)=(0,0)$ and $(-1,0)$, the dimensions of the stable manifolds are not constant. Now, we decompose $V_{(0,1,-1,0);(0,0)}$ and $V_{(0,1,-1,0);(-1,0)}$ into the upper half and the lower half:
\begin{equation*}
\begin{split}
V_{(0,1,-1,0);(0,0)} &= V_{(0,1,-1,0);(0,0)}^- \cup V_{(0,1,-1,0);(0,0)}^+ ,\\
V_{(0,1,-1,0);(0,0)}^- &:= \{(X,Y)\in P\ |\ X+Y=2,\ 0\leq Y\leq1\},\\
V_{(0,1,-1,0);(0,0)}^+ &:= \{(X,Y)\in P\ |\ X+Y=2,\ 1<Y\leq2\},\\
V_{(0,1,-1,0);(-1,0)} &= V_{(0,1,-1,0);(-1,0)}^- \cup V_{(0,1,-1,0);(-1,0)}^+, \\
V_{(0,1,-1,0);(-1,0)}^- &:= \{(X,Y)\in P\ |\ X+Y=8,\ 2\leq Y\leq4\},\\
V_{(0,1,-1,0);(-1,0)}^+ &:= \{(X,Y)\in P\ |\ X+Y=8,\ 4<Y\leq6\}.
\end{split}
\end{equation*}
Then,we have $\dim S_v = 0$ (resp. $\dim S_v =1$) if $v\in V_{(0,1,-1,0);(0,0)}^+$ or $v\in V_{(0,1,-1,0);(-1,0)}^-$ (resp. $v\in V_{(0,1,-1,0);(0,0)}^-$ or $v\in V_{(0,1,-1,0);(-1,0)}^+$).\par
For the space of opposite directional morphisms, the degree of each connected components is similar to these above. Thus, we have
\begin{equation*}
\begin{split}
Mo_\cE(P)\left(L(0,-1,0,1),L(0,0,-1,1)\right) &= 0,\\
Mo_\cE(P)\left(L(0,0,-1,1),L(0,-1,0,1)\right) &= 0
\end{split}
\end{equation*}

\item $Mo_\cE(P)\left(L(0,-1,0,1),L(0,0,0,1)\right)$:\ \par
Since $Mo_\cE(P)\left(L(0,-1,0,1),L(0,0,0,1)\right) \cong Mo_\cE(P)\left(L(0,0,0,0),L(0,1,0,0)\right)$, we consider $Mo_\cE(P)\left(L(0,0,0,0),L(0,1,0,0)\right)$. The intersection of $L(0,0,0,0)$ and $L(0,1,0,0)$ is expressed as
\begin{equation*}
i_1 = 0,\ \ \ 
i_2 = \frac{t}{1+t}.
\end{equation*}
Then, we obtain the connected components
\begin{align*}
V_{(0,1,0,0);(0,0)} &= E_2 \cup E_3 ,& V_{(0,1,0,0);(0,1)}  &= E_5 \cup E_6,
\end{align*}
and the corresponding gradient vector field
\begin{equation*}
2\pi\left( \frac{t}{1+t} - i_2 \right)\frac{\del}{\del x^2}.
\end{equation*}
Since the degrees of these connected components are zero, we have
\begin{equation*}
\begin{split}
Mo_\cE(P)\left(L(0,-1,0,1),L(0,0,0,1)\right) &\cong Mo_\cE(P)\left(L(0,0,0,0),L(0,1,0,0)\right)\\
&= \C\cdot V_{(0,1,0,0);(0,0)} \oplus \C\cdot V_{(0,1,0,0);(0,1)}.
\end{split}
\end{equation*}\par
On the other hand, the space $Mo_\cE(P)\left(L(0,0,0,2),L(-1,0,0,1)\right)$ is trivial. In fact, the degree of $V_{(-1,0,0,-1);I}$ is greater than 0 and none of the point $v_{(-1,0,0,-1);I}\in V_{(-1,0,0,-1);I}$ is an interior point of $P$. Thus, we have
\begin{equation*}
Mo_\cE(P)\left(L(0,0,0,2),L(-1,0,0,1)\right) = 0.
\end{equation*}

\item $Mo_\cE(P)\left(L(0,-1,0,1),L(0,0,0,2)\right)$:\ \par
Since $Mo_\cE(P)\left(L(0,-1,0,1),L(0,0,0,2)\right) \cong Mo_\cE(P)\left(L(0,0,0,0),L(0,1,0,1)\right)$, we consider $Mo_\cE(P)\left(L(0,0,0,0),L(0,1,0,1)\right)$. The intersection of $L(0,0,0,0)$ and $L(0,1,0,1)$ is expressed as
\begin{equation*}
i_1 = \frac{st}{1+st+t},\ \ \ 
i_2 = \frac{t}{1+t} + \frac{st+t}{1+st+t}.
\end{equation*}
Then, we obtain the connected components
\begin{align*}
V_{(0,1,0,1);(0,2)} &= E_6 ,\\
V_{(0,1,0,1);(0,1)} &= \{(0,4)\}, & V_{(0,1,0,1);(1,2)} &= \{(6,2)\} , \\
V_{(0,1,0,1);(0,0)} &= E_2  ,     & V_{(0,1,0,1);(1,1)} &= \{(6,0)\},
\end{align*}
and the corresponding gradient vector field
\begin{equation*}
2\pi\left( \frac{st}{1+st+t} - i_1\right)\frac{\del}{\del x^1} + 2\pi\left(\frac{t}{1+t} + \frac{st+t}{1+st} - i_2 \right)\frac{\del}{\del x^2}.
\end{equation*}
Since the degrees of these connected components are zero, we have
\begin{equation*}
\begin{split}
Mo_\cE(P)\left(L(0,-1,0,1),L(0,0,0,2)\right) &\cong Mo_\cE(P)\left(L(0,0,0,0),L(0,1,0,1)\right)\\
&= \bigoplus_{\substack{0 \leq i_1 \leq 1 \\0 \leq i_1 \leq i_2 \leq 2}}\C\cdot V_{(0,1,0,1);(i_1,i_2)}.
\end{split}
\end{equation*}\par
On the other hand, the space $Mo_\cE(P)\left(L(0,0,0,2),L(0,-1,0,1)\right)$ is trivial. In fact, the degree of $V_{(0,-1,0,-1);I}$ is greater than 0 and none of the point $v_{(0,-1,0,-1);I}\in V_{(0,-1,0,-1);I}$ is an interior point of $P$. Thus, we have
\begin{equation*}
Mo_\cE(P)\left(L(0,0,0,2),L(0,-1,0,1)\right) = 0.
\end{equation*}

\item $Mo_\cE(P)\left(L(0,0,-1,1),L(0,0,0,1)\right)$:\ \par
Since $Mo_\cE(P)\left(L(0,0,-1,1),L(0,0,0,1)\right) \cong Mo_\cE(P)\left(L(0,0,0,0),L(0,0,1,0)\right)$, we consider $Mo_\cE(P)\left(L(0,0,0,0),L(0,0,1,0)\right)$. The intersection of $L(0,0,0,0)$ and $L(0,0,1,0)$ is expressed as
\begin{equation*}
i_1 = i_2 = \frac{st}{1+st}.
\end{equation*}
Then, we obtain the connected components
\begin{align*}
V_{(0,0,1,0);(0,0)} &= E_1 \cup E_2, & V_{(0,0,1,0);(1,1)}  &= E_4 \cup E_5 
\end{align*}
and the corresponding gradient vector field
\begin{equation*}
2\pi\left(\frac{st}{1+st} - i_1\right)\frac{\del}{\del x^1} + 2\pi\left(\frac{st}{1+st} - i_2 \right)\frac{\del}{\del x^2}.
\end{equation*}
Since the degrees of these connected components are zero, we have
\begin{equation*}
\begin{split}
Mo_\cE(P)\left(L(0,0,-1,1),L(0,0,0,1)\right) &\cong Mo_\cE(P)\left(L(0,0,0,0),L(0,0,1,0)\right)\\
&= \C\cdot V_{(0,0,1,0);(0,0)} \oplus \C\cdot V_{(0,0,1,0);(1,1)}.
\end{split}
\end{equation*}\par
On the other hand, the space $Mo_\cE(P)\left(L(0,0,0,1),L(0,0,-1,1)\right)$ is trivial. In fact, the degree of $V_{(0,0,-1,0);I}$ is greater than 0 and none of the point $v_{(0,0,-1,0);I}\in V_{(0,0,-1,0);I}$ is an interior point of $P$. Thus, we have
\begin{equation*}
Mo_\cE(P)\left(L(0,0,0,1),L(0,0,-1,1)\right) = 0.
\end{equation*}

\item $Mo_\cE(P)\left(L(0,0,-1,1),L(0,0,0,2)\right)$:\ \par
Since $Mo_\cE(P)\left(L(0,0,-1,1),L(0,0,0,2)\right) \cong Mo_\cE(P)\left(L(0,0,0,0),L(0,0,1,1)\right)$, we consider $Mo_\cE(P)\left(L(0,0,0,0),L(0,0,1,1)\right)$. The intersection of $L(0,0,0,0)$ and $L(0,0,1,1)$ is expressed as
\begin{equation*}
i_1 = \frac{st}{1+st} + \frac{st}{1+st+t},\ \ \ 
i_2 = \frac{st}{1+st} + \frac{st+t}{1+st+t}.
\end{equation*}
Then, we obtain the connected components
\begin{align*}
V_{(0,0,1,1);(0,1)} &= \{(0,6)\}, & V_{(0,0,1,1);(1,2)} &= \{(2,6)\},  \\
V_{(0,0,1,1);(0,0)} &= E_2,       & V_{(0,0,1,1);(1,1)} &= \{(4,0)\},  &  V_{(0,0,1,1);(2,2)} &= E_4,
\end{align*}
and the corresponding gradient vector field
\begin{equation*}
2\pi\left( \frac{st}{1+st} + \frac{st}{1+st+t} - i_1\right)\frac{\del}{\del x^1} + 2\pi\left(\frac{st}{1+st} + \frac{st+t}{1+st+t} - i_2 \right)\frac{\del}{\del x^2}.
\end{equation*}
Since the degrees of these connected components are zero, we have
\begin{equation*}
\begin{split}
Mo_\cE(P)\left(L(0,0,-1,1),L(0,0,0,2)\right) &\cong Mo_\cE(P)\left(L(0,0,0,0),L(0,0,1,1)\right)\\
&=  \bigoplus_{\substack{0 \leq i_1 \leq 1 \\0 \leq i_1 \leq i_2 \leq 2}}\C\cdot V_{(0,0,1,1);(i_1,i_2)}.
\end{split}
\end{equation*}\par
On the other hand, the space $Mo_\cE(P)\left(L(0,0,0,2),L(0,0,-1,1)\right)$ is trivial. In fact, the degree of $V_{(0,0,-1,-1);I}$ is greater than 0 and none of the point $v_{(0,0,-1,-1);I}\in V_{(0,0,-1,-1);I}$ is an interior point of $P$. Thus, we have
\begin{equation*}
Mo_\cE(P)\left(L(0,0,0,2),L(0,0,-1,1)\right) = 0.
\end{equation*}

\end{itemize}
Summarizing the above results, we obtain the following.

\begin{lem}
Each generator of the space of morphisms in $Mo_\cE(P)$ belongs to $\del P$ and its degree is zero.
\end{lem}

\subsection{Proof of the main theorem for $\check{X}=\Bll$.}\label{composite2}
In order to prove the main theorem (Theorem \ref{main}), we construct a quasi-isomorphism $\iota:Mo_\cE(P)\to\cV_\cE'$ explicitly. For the objects, we assign $L(a,b,c,d)$ to $\widetilde{\cO}(a,b,c,d)$. For the morphisms, we obtain the following.

\begin{lem}\label{lem3-2}
The basis $\ee_{(a,b,c,d);I}$ of $H^0(\cV'_\cE(\widetilde{\cO}(a_1,b_1,c_1,d_1),\widetilde{\cO}(a_1+a,b_1+b,c_1+c,d_1+d)))$ are expressed as the form
\begin{equation*}
\ee_{(a,b,c,d);I}(x)=e^{-f_I}e^{\ii Iy},
\end{equation*}
where $e^{-f_I}$ is continuous on $P$ and smooth on $B$. Furthermore, the function $f_I$ satisfies
\begin{equation}\label{dfI-3}
df_I = \sum_{j=1}^2\frac{\del f_I}{\del x_j}dx_j,\ \ \ \frac{\del f_I}{\del x_j} = \frac{y^j_{(a,b,c)} - 2\pi i_j}{2\pi},
\end{equation}
in $B$ and $\min_{x\in P}f_I=0$. In particular, we have 
\begin{equation}\label{VI3}
\{x\in P\ |\ f_I(x)=0\}=V_{(a,b,c,d);I}.
\end{equation}
Thus, the correspondence $\iota:V_{(a,b,c,d);I}\mapsto \ee_{(a,b,c,d);I}$ gives a quasi-isomorphism
\begin{equation*}
\begin{split}
\iota : Mo_\cE(P)(&L(a_1,b_1,c_1,d_1),L(a_1+a,b_1+b,c_1+c,d_1+d))\\
&\to \cV'_\cE(\widetilde{\cO}(a_1,b_1,c_1,d_1),\widetilde{\cO}(a_1+a,b_1+b,c_1+c,d_1+d))
\end{split}
\end{equation*}
of complexes.
\end{lem}
\begin{proof}
By the expression (\ref{pre-e2}), the basis $\ee_{(a,b,c,d);I}$ of $H^0(\cV'_\cE(\widetilde{\cO}(a_1,b_1,c_1,d_1),\widetilde{\cO}(a_1+a,b_1+b,c_1+c,d_1+d)))$ are expressed as the form
\begin{equation*}
\ee_{(a,b,c,d);I} =c_{(a,b,c,d);(i_1,i_2)}(1+s)^{\frac{-a}{2}}(1+t)^{\frac{-b}{2}}(1+st)^{\frac{-c}{2}}(1+st+t)^{\frac{-d}{2}}s^\frac{i_1}{2}t^\frac{i_2}{2}e^{\ii(i_1y_1 + i_2y_2)}
\end{equation*}
where $s = e^{2 x_1},\ t= e^{2 x_2}$ and $c_{(a,b,c,d);(i_1,i_2)}$ is a constant. Then, the functions $e^{-f_I}$ and $f_I$ are given by
\begin{equation}\label{fI3}
\begin{split}
e^{-f_I} &= c_{(a,b,c,d);(i_1,i_2)}(1+s)^{\frac{-a}{2}}(1+t)^{\frac{-b}{2}}(1+st)^{\frac{-c}{2}}(1+st+t)^{\frac{-d}{2}}s^\frac{i_1}{2}t^\frac{i_2}{2},\\
f_I  &= \log\left((1+s)^{\frac{a}{2}}(1+t)^{\frac{b}{2}}(1+st)^{\frac{c}{2}}(1+st+t)^{\frac{d}{2}}s^\frac{-i_1}{2}t^\frac{-i_2}{2}\right) + \mbox{const}.
\end{split}
\end{equation}
By this expression, the function $f_I$ satisfy (\ref{dfI-3}). For $(a,b,c,d)=(-1,0,0,1)$, since $i_1=0$ and $i_2=1$, we have
\begin{equation*}
e^{-f_I} = c_{(-1,0,0,1);I}(1+s)^{\frac{1}{2}}(1+st+t)^{\frac{-1}{2}}t^{\frac{1}{2}} = c_{(-1,0,0,1);I}\left(\frac{1}{1+\frac{1}{st+t}}\right)^{\frac{1}{2}}.
\end{equation*}
For $(a,b,c,d)=(0,-1,0,1),(0,0,-1,1)$, we have
\begin{align*}
e^{-f_I} &= c_{(0,-1,0,1);I}(1+t)^{\frac{1}{2}}(1+st+t)^{\frac{-1}{2}} = c_{(0,-1,0,1);I}\left(\frac{1}{1+\frac{st}{1+t}}\right)^{\frac{1}{2}},\\
e^{-f_I} &= c_{(0,0,-1,1);I}(1+st)^{\frac{1}{2}}(1+st+t)^{\frac{-1}{2}}t^{\frac{1}{2}} = c_{(0,0,-1,1);I}\left(\frac{1}{1+\frac{t}{1+st}}\right)^{\frac{1}{2}},
\end{align*}
since $i_1=i_2=0$. In these cases, the function $e^{-f_I}$ is continuous on $P$. For $a,b,c,d\geq0$, the degree of the denominator with respect to $s$ or $t$ is higher than that of the numerator, and the function $e^{-f_I}$ is also continuous on $P$. Furthermore, we obtain (\ref{VI3}) by evaluating $f_I$ in (\ref{fI3}) in each case directly. Since the DG structure of $Mo_\cE(P)$ is minimal, we obtain the last statement.
\end{proof}

\begin{lem}\label{lem3-3}
The correspondence $\iota$ forms a functor from $Mo_\cE(P)$ to $\cV'_\cE$.
\end{lem}
\begin{proof}
By lemma \ref{lem3-2}, the identity morphism in $Mo_\cE(P)$ is maped to that of in $\cV_\cE'$.\par
The composition in $\cV_\cE'$ or $H^0(\cV_\cE')$ is simply a product between functions, i.e.,
\begin{equation}\label{prod-e-pre3}
e^{-f_{(\alpha,\beta,\gamma,\delta);I}}e^{\ii Iy} \otimes e^{-f_{(\alpha',\beta',\gamma',\delta');J}}e^{\ii Jy} \longmapsto e^{-\left(f_{(\alpha,\beta,\gamma,\delta);I}+f_{(\alpha',\beta',\gamma',\delta');J}\right)}e^{\ii (I+J)y},
\end{equation}
where $\alpha:=a_2-a_1,\ \beta:=b_2-b_1,\ \gamma:=c_2-c_1,\ \delta:=d_2-d_1,\ \alpha':=a_3-a_2,\ \beta':=b_3-b_2,\ \gamma':=c_3-c_2,\ \delta':=d_3-d_2$.
Here, let $v$ be a point where the function$f_{(\alpha,\beta,\gamma,\delta);I}+f_{(\alpha',\beta',\gamma',\delta');J}$ is minimum. Then, we can rewrite (\ref{prod-e-pre3}) by \begin{equation}\label{prod-e3}
\ee_{(\alpha,\beta,\gamma,\delta);I} \otimes \ee_{(\alpha',\beta',\gamma',\delta');J} \mapsto e^{-\left(f_{(\alpha,\beta,\gamma,\delta);I}(v) + f_{(\alpha',\beta',\gamma',\delta');J}(v)\right)}\ee_{(\alpha+\alpha',\beta+\beta',\gamma+\gamma',\delta+\delta');I+J},
\end{equation}\par
On the other hand, the composition in $Mo_\cE(P)$ is given by
\begin{equation*}
V_{(\alpha,\beta,\gamma,\delta);I}\otimes V_{(\alpha',\beta',\gamma',\delta');J} \mapsto e^{-A(\Gamma)}V_{(\alpha+\alpha',\beta+\beta',\gamma+\gamma',\delta+\delta');I+J},
\end{equation*}
where $\Gamma$ is the unique gradient tree starting at two points $v_{(\alpha,\beta,\gamma,\delta);I}\in V_{(\alpha,\beta,\gamma,\delta);I}$ and $v_{(\alpha',\beta',\gamma',\delta');J}\in V_{(\alpha',\beta',\gamma',\delta');J}$ and ending at a point $v_{(\alpha+\alpha',\beta+\beta',\gamma+\gamma',\delta+\delta');I+J}\in V_{(\alpha+\alpha',\beta+\beta',\gamma+\gamma',\delta+\delta');I+J}$, and  $A(\Gamma)$ is the symplectic area associated to the lift of $\Gamma$.
 Since the Lagrangian $L(a,b,c,d)$ is locally the graph of $df_{(a,b,c,d);I}$ and $f_{(a,b,c,d);I}=0$ on $V_{(a,b,c,d);I}$, we have
\begin{equation}\label{area3}
A(\Gamma) = f_{(\alpha,\beta,\gamma,\delta);I}(v_{(\alpha+\alpha',\beta+\beta',\gamma+\gamma',\delta+\delta');I+J}) + f_{(\alpha',\beta',\gamma',\delta');J}(v_{(\alpha+\alpha',\beta+\beta',\gamma+\gamma',\delta+\delta');I+J}).
\end{equation}
By subsection \ref{morse homotopy2}, for each generator $V$ of the space of the morphisms in $Mo_\cE(P)$, the gradient trajectories are lines on the adjacent edges of $V$. Using this fact, for $V_{(\alpha,\beta,\gamma,\delta);I},\ V_{(\alpha',\beta',\gamma',\delta');J}$ and $V_{(\alpha+\alpha',\beta+\beta',\gamma+\gamma',\delta+\delta');I+J}$, the gradient trees is as follows:
\begin{itemize}
\item In the case of  $V_{(\alpha,\beta,\gamma,\delta);I} \cap V_{(\alpha',\beta',\gamma',\delta');J} = \emptyset$, there exists a gradient tree if $v_{(\alpha,\beta,\gamma,\delta);I}$ and $v_{(\alpha',\beta',\gamma',\delta');J}$ are in the same edge. Then, the image of the root edge is $\{v_{(\alpha+\alpha',\beta+\beta',\gamma+\gamma',\delta+\delta');I+J}\}$.

\item In the case of $V_{(\alpha,\beta,\gamma,\delta);I} \cap V_{(\alpha',\beta',\gamma',\delta');J} \neq \emptyset$, there exists a gradient tree only if $v_{(\alpha,\beta,\gamma,\delta);I} = v_{(\alpha',\beta',\gamma',\delta');J} = v_{(\alpha+\alpha',\beta+\beta',\gamma+\gamma',\delta+\delta');I+J}$. This gradient tree is trivial and the symplectic area equals zero.
\end{itemize}
In Appendix \ref{appendix2}, we give the list of compositions of morphisms and gradient trees. Since $v_{(\alpha+\alpha',\beta+\beta',\gamma+\gamma',\delta+\delta');I+J}$ coincide with the point $v$ in (\ref{prod-e3}), the correspondence $\iota:V_{(a,b,c,d);I}\mapsto \ee_{(a,b,c,d);I}$ is compatible with the composition of morphisms:
\begin{equation*}
\fm_2 \circ (\iota\otimes\iota)\left(V_{(\alpha,\beta,\gamma,\delta);I}\otimes V_{(\alpha',\beta',\gamma',\delta');J}\right) =  \iota \circ \fm_2 \left(V_{(\alpha,\beta,\gamma,\delta);I}\otimes V_{(\alpha',\beta',\gamma',\delta');J}\right).
\end{equation*}
\end{proof}
By Lemma \ref{lem3-2} and Lemma \ref{lem3-3}, the correspondence $\iota:Mo_\cE(P)\to\cV_\cE'$ is the quasi-isomorphism between DG categories. On the other hand, the DG category $DG_\cE(\Bll)$ is quasi-isomorphic to $\cV_\cE'$ by subsection \ref{HV'}. Thus, the proof of the main theorem (Theorem \ref{main}) for $\Bll$ is completed.

\appendix
\section{}
Here, we give the lists of compositions of morphisms in $Mo_{\cE}(P)$ and non-trivial gradient trees. In the case of $\Bl$,  the compositions of $Z_I$ and $W_J$ coincide with $V_{I+J}$, where $Z_I$ is a generator of $Mo_\cE(P)(L(a,b,c),L(a',b',c'))$, $W_J$ is a generator of $Mo_\cE(P)(L(a',b',c'),L(a'',b'',c''))$ and $V_K$ is a generator of $Mo_\cE(P)(L(a,b,c),L(a'',b'',c''))$. In the case of $\Bll$, we replace $(a,b,c)$ above by $(a,b,c,d)$.

\subsection{Compositions of morphisms in $Mo_{\cE}(P)$ for $\Bl$.}\label{appendix1}
\begin{itemize}
\item $L(0,0,0) \to L(0,-1,1)) \to L(0,0,1))$
\[
\begin{tikzpicture}
\draw[dashed](0,0)--  (4,0)--   (4,2) node[right=10pt]{$\bigotimes$}--  (2,4)--  (0,4)-- node[auto=left]{$Z_{(0,0)}$} cycle;
\draw[ultra thick] (0,0) -- (0,4) -- (2,4);
\end{tikzpicture}
\ \ \ 
\begin{tikzpicture}
\draw[dashed](0,0)-- node[auto=left]{$W_{(0,0)}$} (4,0)--   (4,2) node[right=10pt]{$\longmapsto$}-- node[auto=left]{$W_{(0,1)}$}  (2,4)--  (0,4)--  cycle;
\draw[ultra thick] (0,0) -- (4,0);
\draw[ultra thick] (0,4) -- (2,4) -- (4,2);
\end{tikzpicture}
\ \ \ 
\begin{tikzpicture}
\draw[dashed](0,0) node[above right]{$V_{(0,0)}$}--  (4,0)--   (4,2)--   (2,4)--node[auto=left]{$V_{(0,1)}$}  (0,4)--  cycle;
\draw[ultra thick] (0,4) -- (2,4);
\fill[black](0,0)circle(0.06);
\end{tikzpicture}
\]
There is the only trivial gradient tree.

\item $L(0,0,0) \to L(-1,0,1)) \to L(0,0,1))$
\[
\begin{tikzpicture}
\draw[dashed](0,0)-- node[auto=left]{$Z_{(0,0)}$} (4,0)--   (4,2) node[right=10pt]{$\bigotimes$}--  (2,4)--  (0,4)-- cycle;
\draw[ultra thick] (0,0) -- (4,0) -- (4,2);
\end{tikzpicture}
\ \ \ 
\begin{tikzpicture}
\draw[dashed](0,0)--  (4,0)--node[auto=left]{$W_{(1,0)}$}   (4,2) node[right=10pt]{$\longmapsto$}--   (2,4)--  (0,4)--node[auto=left]{$W_{(0,0)}$}  cycle;
\draw[ultra thick] (0,0) -- (0,4);
\draw[ultra thick] (4,0) -- (4,2) -- (2,4);
\end{tikzpicture}
\ \ \ 
\begin{tikzpicture}
\draw[dashed](0,0) --  (4,0)--node[auto=left]{$V_{(1,0)}$}   (4,2)--   (2,4)--  (0,4)--  cycle;
\draw[ultra thick] (4,0) -- (4,2);
\fill[black](0,0)circle(0.06) node[above right]{$V_{(0,0)}$};
\end{tikzpicture}
\]
There is the only trivial gradient tree.

\item $L(0,0,0) \to L(0,-1,1)) \to L(0,0,2))$
\[
\begin{tikzpicture}
\draw[dashed](0,0)--  (4,0)--   (4,2) node[right=10pt]{$\bigotimes$}--  (2,4)--  (0,4)-- node[auto=left]{$Z_{(0,0)}$} cycle;
\draw[ultra thick] (0,0) -- (0,4) -- (2,4);
\end{tikzpicture}
\ \ \ 
\begin{tikzpicture}
\draw[dashed](0,0) --  (4,0)--   (4,2) node[right=10pt]{$\longmapsto$}--   (2,4)--node[auto=left]{$W_{(0,2)}$}  (0,4)--  cycle;
\fill[black](0,0)circle(0.06) node[above right]{$W_{(0,0)}$};
\fill[black](4,0)circle(0.06) node[above left]{$W_{(1,0)}$};
\fill[black](4,2)circle(0.06) node[left]{$W_{(1,1)}$};
\fill[black](0,2)circle(0.06) node[right]{$W_{(0,1)}$};
\draw[ultra thick] (0,4) -- (2,4);
\end{tikzpicture}
\ \ \ 
\begin{tikzpicture}
\draw[dashed](0,0) --  (4,0)--   (4,2)--   (2,4)--node[auto=left]{$V_{(0,2)}$}  (0,4)--  cycle;
\fill[black](0,0)circle(0.06) node[above right]{$V_{(0,0)}$};
\fill[black](3,3)circle(0.06) node[below left]{$V_{(1,1)}$};
\fill[black](2,0)circle(0.06) node[above]{$V_{(1,0)}$};
\fill[black](0,2)circle(0.06) node[right]{$V_{(0,1)}$};
\draw[ultra thick] (0,4) -- (2,4);
\end{tikzpicture}
\]
For $Z_{(0,0)}\otimes W_{(0,j)}\mapsto V_{(0,j)}\ (j=0,1,2)$ there is the only trivial gradient tree. For the other compositions, there exists non-trivial gradient trees which are as following :
\[
\begin{tikzpicture}
\draw[dashed](0,0)--  (4,0)--   (4,2) --  (2,4)--  (0,4)-- node[auto=right]{$Z_{(0,0)}$} cycle;
\draw[ultra thick] (0,0) -- (0,4) -- (2,4);
\fill[black](4,0)circle(0.06) node[right]{$W_{(1,0)}$};
\fill[black](4,2)circle(0.06) node[right]{$W_{(1,1)}$};
\fill[black](3,3)circle(0.06) node[above right]{$V_{(1,1)}$};
\fill[black](2,0)circle(0.06) node[below]{$V_{(1,0)}$};
\draw (0,0) -- (4,0);
\draw [arrows = {-Stealth[scale=1.5]}] (0,0) -- (1.2,0);
\draw [arrows = {-Stealth[scale=1.5]}] (4,0) -- (2.8,0);
\draw (4,2) -- (2,4);
\draw [arrows = {-Stealth[scale=1.5]}] (2,4) -- (2.6,3.4);
\draw [arrows = {-Stealth[scale=1.5]}] (4,2) -- (3.4,2.6);
\end{tikzpicture}
\]

\item $L(0,0,0) \to L(-1,0,1)) \to L(0,0,2))$
\[
\begin{tikzpicture}
\draw[dashed](0,0)-- node[auto=left]{$Z_{(0,0)}$} (4,0)--   (4,2) node[right=10pt]{$\bigotimes$}--  (2,4)--  (0,4)-- cycle;
\draw[ultra thick] (0,0) -- (4,0) -- (4,2);
\end{tikzpicture}
\ \ \ 
\begin{tikzpicture}
\draw[dashed](0,0) --  (4,0)-- node[auto=left]{$W_{(2,0)}$}  (4,2) node[right=10pt]{$\longmapsto$}--   (2,4)--  (0,4)--  cycle;
\fill[black](0,0)circle(0.06) node[above right]{$W_{(0,0)}$};
\fill[black](0,4)circle(0.06) node[below right]{$W_{(0,1)}$};
\fill[black](2,4)circle(0.06) node[below]{$W_{(1,1)}$};
\fill[black](2,0)circle(0.06) node[above]{$W_{(1,0)}$};
\draw[ultra thick] (4,0) -- (4,2);
\end{tikzpicture}
\ \ \ 
\begin{tikzpicture}
\draw[dashed](0,0) --  (4,0)-- node[auto=left]{$V_{(2,0)}$}  (4,2)--   (2,4)--  (0,4)--  cycle;
\fill[black](0,0)circle(0.06) node[above right]{$V_{(0,0)}$};
\fill[black](3,3)circle(0.06) node[below left]{$V_{(1,1)}$};
\fill[black](2,0)circle(0.06) node[above]{$V_{(1,0)}$};
\fill[black](0,2)circle(0.06) node[right]{$V_{(0,1)}$};
\draw[ultra thick] (4,0) -- (4,2);
\end{tikzpicture}
\]
For $Z_{(0,0)}\otimes W_{(j,0)}\mapsto V_{(0,j)}\ (j=0,1,2)$, there exists a trivial gradient tree. For the other compositions, there exists non-trivial gradient trees which are as following :
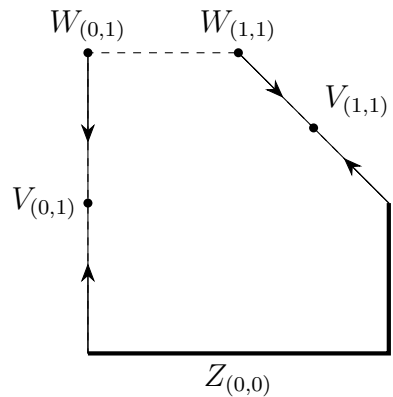
\[
\begin{tikzpicture}
\draw[dashed](0,0)--node[auto=right]{$Z_{(0,0)}$}  (4,0)--   (4,2) --  (2,4)--  (0,4)--  cycle;
\draw[ultra thick] (0,0) -- (4,0) -- (4,2);
\fill[black](0,4)circle(0.06) node[above]{$W_{(0,1)}$};
\fill[black](2,4)circle(0.06) node[above]{$W_{(1,1)}$};
\fill[black](3,3)circle(0.06) node[above right]{$V_{(1,1)}$};
\fill[black](0,2)circle(0.06) node[left]{$V_{(0,1)}$};
\draw (0,0) -- (0,4);
\draw [arrows = {-Stealth[scale=1.5]}] (0,0) -- (0,1.2);
\draw [arrows = {-Stealth[scale=1.5]}] (0,4) -- (0,2.8);
\draw (4,2) -- (2,4);
\draw [arrows = {-Stealth[scale=1.5]}] (2,4) -- (2.6,3.4);
\draw [arrows = {-Stealth[scale=1.5]}] (4,2) -- (3.4,2.6);
\end{tikzpicture}
\]
\item $L(0,0,0) \to L(0,0,1)) \to L(0,0,2))$
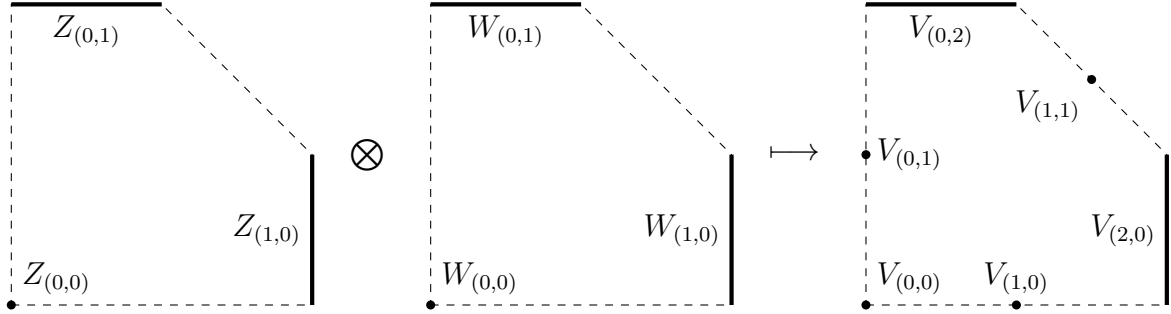
\[
\begin{tikzpicture}
\draw[dashed](0,0) node[above right]{$Z_{(0,0)}$}--  (4,0)-- node[auto=left]{$Z_{(1,0)}$}  (4,2) node[right=10pt]{$\bigotimes$}--   (2,4)--node[auto=left]{$Z_{(0,1)}$}  (0,4)--  cycle;
\draw[ultra thick] (4,0) -- (4,2);
\draw[ultra thick] (0,4) -- (2,4);
\fill[black](0,0)circle(0.06);
\end{tikzpicture}
\ \ \ 
\begin{tikzpicture}
\draw[dashed](0,0) node[above right]{$W_{(0,0)}$}--  (4,0)-- node[auto=left]{$W_{(1,0)}$}  (4,2) node[right=10pt]{$\longmapsto$}--   (2,4)--node[auto=left]{$W_{(0,1)}$}  (0,4)--  cycle;
\draw[ultra thick] (4,0) -- (4,2);
\draw[ultra thick] (0,4) -- (2,4);
\fill[black](0,0)circle(0.06);
\end{tikzpicture}
\ \ \ 
\begin{tikzpicture}
\draw[dashed](0,0) --  (4,0)-- node[auto=left]{$V_{(2,0)}$}  (4,2)--   (2,4)--node[auto=left]{$V_{(0,2)}$}  (0,4)--  cycle;
\fill[black](0,0)circle(0.06) node[above right]{$V_{(0,0)}$};
\fill[black](3,3)circle(0.06) node[below left]{$V_{(1,1)}$};
\fill[black](2,0)circle(0.06) node[above]{$V_{(1,0)}$};
\fill[black](0,2)circle(0.06) node[right]{$V_{(0,1)}$};
\draw[ultra thick] (4,0) -- (4,2);
\draw[ultra thick] (0,4) -- (2,4);
\end{tikzpicture}
\]
For $Z_K\otimes W_K\mapsto V_{2K}\ (K=(0,0),(0,1),(1,0))$ there is the only trivial gradient tree. For the other compositions, there exists non-trivial gradient trees which are as following :

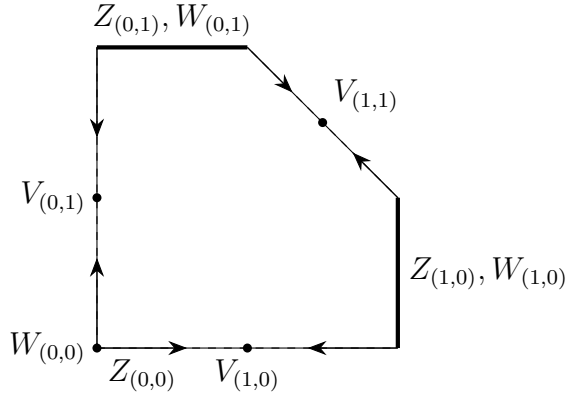
\[
\begin{tikzpicture}
\draw[dashed](0,0) --  (4,0)-- node[auto=right]{$Z_{(1,0)},W_{(1,0)}$}  (4,2)--   (2,4)--node[auto=right]{$Z_{(0,1)},W_{(0,1)}$}  (0,4)--  cycle;
\fill[black](0,0)circle(0.06) node[below right]{$Z_{(0,0)}$}node[left]{$W_{(0,0)}$};
\fill[black](3,3)circle(0.06) node[above right]{$V_{(1,1)}$};
\fill[black](2,0)circle(0.06) node[below]{$V_{(1,0)}$};
\fill[black](0,2)circle(0.06) node[left]{$V_{(0,1)}$};
\draw[ultra thick] (4,0) -- (4,2);
\draw[ultra thick] (0,4) -- (2,4);
\draw (0,0) -- (4,0);
\draw [arrows = {-Stealth[scale=1.5]}] (0,0) -- (1.2,0);
\draw [arrows = {-Stealth[scale=1.5]}] (4,0) -- (2.8,0);
\draw (0,0) -- (0,4);
\draw [arrows = {-Stealth[scale=1.5]}] (0,0) -- (0,1.2);
\draw [arrows = {-Stealth[scale=1.5]}] (0,4) -- (0,2.8);
\draw (4,2) -- (2,4);
\draw [arrows = {-Stealth[scale=1.5]}] (2,4) -- (2.6,3.4);
\draw [arrows = {-Stealth[scale=1.5]}] (4,2) -- (3.4,2.6);
\end{tikzpicture}
\]

\item $L(0,-1,1) \to L(0,0,1)) \to L(0,0,2))$
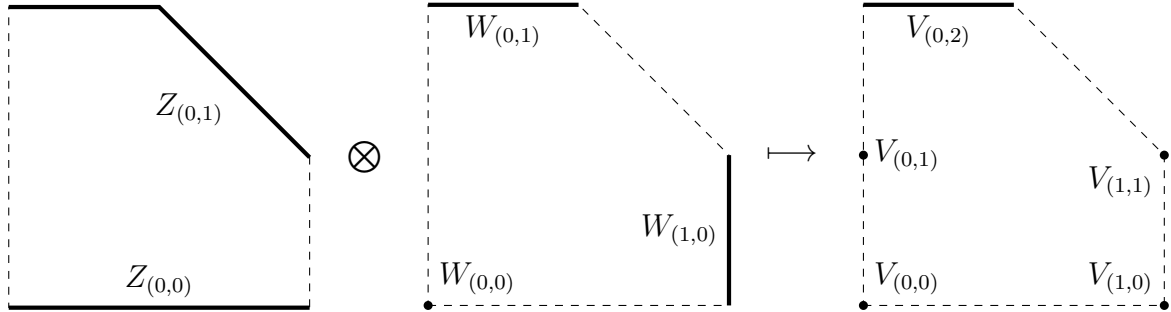
\[
\begin{tikzpicture}
\draw[dashed](0,0)-- node[auto=left]{$Z_{(0,0)}$} (4,0)--   (4,2) node[right=10pt]{$\bigotimes$}-- node[auto=left]{$Z_{(0,1)}$}  (2,4)--  (0,4)--  cycle;
\draw[ultra thick] (0,0) -- (4,0);
\draw[ultra thick] (0,4) -- (2,4) -- (4,2);
\end{tikzpicture}
\ \ \ 
\begin{tikzpicture}
\draw[dashed](0,0) node[above right]{$W_{(0,0)}$}--  (4,0)-- node[auto=left]{$W_{(1,0)}$}  (4,2) node[right=10pt]{$\longmapsto$}--   (2,4)--node[auto=left]{$W_{(0,1)}$}  (0,4)--  cycle;
\draw[ultra thick] (4,0) -- (4,2);
\draw[ultra thick] (0,4) -- (2,4);
\fill[black](0,0)circle(0.06);
\end{tikzpicture}
\ \ \ 
\begin{tikzpicture}
\draw[dashed](0,0) --  (4,0)--   (4,2)--   (2,4)--node[auto=left]{$V_{(0,2)}$}  (0,4)--  cycle;
\fill[black](0,0)circle(0.06) node[above right]{$V_{(0,0)}$};
\fill[black](4,2)circle(0.06) node[below left]{$V_{(1,1)}$};
\fill[black](4,0)circle(0.06) node[above left]{$V_{(1,0)}$};
\fill[black](0,2)circle(0.06) node[right]{$V_{(0,1)}$};
\draw[ultra thick] (0,4) -- (2,4);
\end{tikzpicture}
\]
For $Z_{(0,0)}\otimes W_{(0,1)}\mapsto V_{(0,1)}$ and $Z_{(0,1)}\otimes W_{(0,0)}\mapsto V_{(0,1)}$ there is non-trivial gradient trees and for the other compositions, there exists the only trivial gradient tree. Non-trivial gradient trees are as following :
\[
\begin{tikzpicture}
\draw[dashed](0,0)-- node[auto=right]{$Z_{(0,0)}$} (4,0)--   (4,2)  --   (2,4) --node[auto=right]{$W_{(0,1)}$}  (0,4)--  cycle;
\draw[ultra thick] (0,0) -- (4,0);
\draw[ultra thick] (0,4) -- (2,4) ;
\fill[black](0,2)circle(0.06) node[left]{$V_{(0,1)}$};
\draw (0,0) -- (0,4);
\draw [arrows = {-Stealth[scale=1.5]}] (0,0) -- (0,1.2);
\draw [arrows = {-Stealth[scale=1.5]}] (0,4) -- (0,2.8);
\end{tikzpicture}
\ \ \ 
\begin{tikzpicture}
\draw[dashed](0,0)--  (4,0)--   (4,2)  --   (2,4) node[above right]{$Z_{(0,1)}$}--  (0,4)--  cycle;
\draw[ultra thick] (0,4) -- (2,4) -- (4,2);
\fill[black](0,0)circle(0.06) node[below]{$W_{(0,0)}$};
\fill[black](0,2)circle(0.06) node[left]{$V_{(0,1)}$};
\draw (0,0) -- (0,4);
\draw [arrows = {-Stealth[scale=1.5]}] (0,0) -- (0,1.2);
\draw [arrows = {-Stealth[scale=1.5]}] (0,4) -- (0,2.8);
\end{tikzpicture}
\]

\item $L(-1,0,1) \to L(0,0,1)) \to L(0,0,2))$
\[
\begin{tikzpicture}
\draw[dashed](0,0)--  (4,0)--node[auto=left]{$Z_{(1,0)}$}   (4,2) node[right=10pt]{$\bigotimes$}--   (2,4)--  (0,4)--node[auto=left]{$Z_{(0,0)}$}  cycle;
\draw[ultra thick] (0,0) -- (0,4);
\draw[ultra thick] (4,0) -- (4,2) -- (2,4);
\end{tikzpicture}
\ \ \ 
\begin{tikzpicture}
\draw[dashed](0,0) node[above right]{$W_{(0,0)}$}--  (4,0)-- node[auto=left]{$W_{(1,0)}$}  (4,2) node[right=10pt]{$\longmapsto$}--   (2,4)--node[auto=left]{$W_{(0,1)}$}  (0,4)--  cycle;
\draw[ultra thick] (4,0) -- (4,2);
\draw[ultra thick] (0,4) -- (2,4);
\fill[black](0,0)circle(0.06);
\end{tikzpicture}
\ \ \ 
\begin{tikzpicture}
\draw[dashed](0,0) --  (4,0)-- node[auto=left]{$V_{(2,0)}$}  (4,2)--   (2,4)--  (0,4)--  cycle;
\fill[black](0,0)circle(0.06) node[above right]{$V_{(0,0)}$};
\fill[black](2,4)circle(0.06) node[below]{$V_{(1,1)}$};
\fill[black](2,0)circle(0.06) node[above]{$V_{(1,0)}$};
\fill[black](0,4)circle(0.06) node[below right]{$V_{(0,1)}$};
\draw[ultra thick] (4,0) -- (4,2);
\end{tikzpicture}
\]
For $Z_{(0,0)}\otimes W_{(1,0)}\mapsto V_{(1,0)}$ and $Z_{(1,0)}\otimes W_{(0,0)}\mapsto V_{(1,0)}$ there is non-trivial gradient trees and for the other compositions, there exists the only trivial gradient tree. Non-trivial gradient trees are as following :
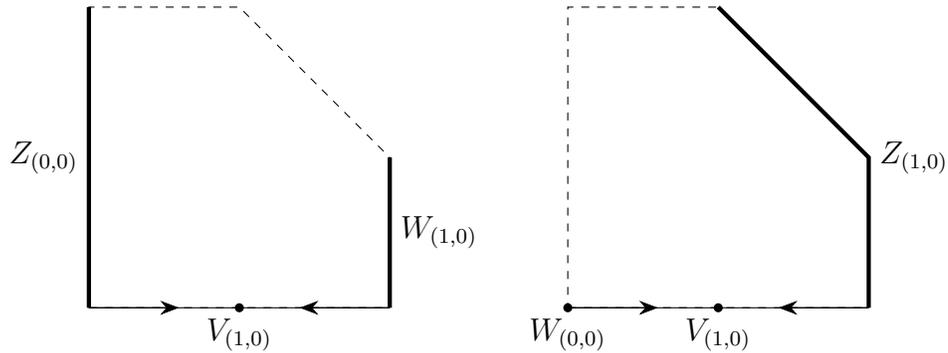
\[
\begin{tikzpicture}
\draw[dashed](0,0) --  (4,0)--node[auto=right]{$W_{(1,0)}$}   (4,2)  --   (2,4)--  (0,4)--node[auto=right]{$Z_{(0,0)}$}  cycle;
\draw[ultra thick] (0,0) -- (0,4);
\draw[ultra thick] (4,0) -- (4,2) ;
\fill[black](2,0)circle(0.06) node[below]{$V_{(1,0)}$};
\draw (0,0) -- (4,0);
\draw [arrows = {-Stealth[scale=1.5]}] (0,0) -- (1.2,0);
\draw [arrows = {-Stealth[scale=1.5]}] (4,0) -- (2.8,0);
\end{tikzpicture}
\ \ \ 
\begin{tikzpicture}
\draw[dashed](0,0) --  (4,0)--   (4,2) node[right]{$Z_{(1,0)}$} --   (2,4)--  (0,4)--  cycle;
\draw[ultra thick] (4,0) -- (4,2) -- (2,4);
\fill[black](0,0)circle(0.06) node[below]{$W_{(0,0)}$};
\fill[black](2,0)circle(0.06) node[below]{$V_{(1,0)}$};
\draw (0,0) -- (4,0);
\draw [arrows = {-Stealth[scale=1.5]}] (0,0) -- (1.2,0);
\draw [arrows = {-Stealth[scale=1.5]}] (4,0) -- (2.8,0);
\end{tikzpicture}
\]
\end{itemize}

\subsection{Compositions of morphisms in $Mo_{\cE}(P)$ for $\Bll$.}\label{appendix2}

\begin{itemize}
\item $L(0,0,0,0) \to L(-1,0,0,1) \to  L(0,0,0,1)$\ \par
\[
\begin{tikzpicture}
\draw[dashed](0,1.5)--  (1.5,0) --  (4.5,0)--  (4.5,1.5)node[above=20pt,right=1pt]{$\bigotimes$} -- node[auto=left]{$Z_{(0,1)}$} (1.5,4.5)--  (0,4.5)--  cycle;
\draw[ultra thick] (4.5,0)--(4.5,1.5)--(1.5,4.5)--(0,4.5);
\end{tikzpicture}
\ \ \ 
\begin{tikzpicture}
\draw[dashed](0,1.5)--  (1.5,0) -- node[auto=left]{$W_{(1,0)}$} (4.5,0)--  (4.5,1.5)node[above=20pt,right=1pt]{$\longmapsto$}--  (1.5,4.5)--  (0,4.5)-- node[auto=left]{$W_{(0,0)}$} cycle;
\draw[ultra thick] (0,1.5)--(0,4.5)--(1.5,4.5);
\draw[ultra thick] (1.5,0)--(4.5,0)--(4.5,1.5);
\end{tikzpicture}
\ \ \ 
\begin{tikzpicture}
\draw[dashed](0,1.5)--  (1.5,0) --  (4.5,0)-- node[auto=left]{$V_{(1,1)}$} (4.5,1.5)--  (1.5,4.5)-- node[auto=left]{$V_{(0,1)}$} (0,4.5)--  cycle;
\draw[ultra thick] (4.5,0)--(4.5,1.5);
\draw[ultra thick] (0,4.5)--(1.5,4.5);
\end{tikzpicture}
\]
There is the only trivial gradient tree.

\item $L(0,0,0,0) \to L(0,-1,0,1) \to  L(0,0,0,1)$\ \par
\[
\begin{tikzpicture}
\draw[dashed](0,1.5)--  (1.5,0) --  (4.5,0)--  (4.5,1.5)node[above=20pt,right=1pt]{$\bigotimes$}--  (1.5,4.5)--  (0,4.5)-- node[auto=left]{$Z_{(0,0)}$} cycle;
\draw[ultra thick] (1.5,0)--(0,1.5)--(0,4.5)--(1.5,4.5);
\end{tikzpicture}
\ \ \ 
\begin{tikzpicture}
\draw[dashed](0,1.5)-- node[auto=left]{$W_{(0,0)}$} (1.5,0) --  (4.5,0)--  (4.5,1.5)node[above=20pt,right=1pt]{$\longmapsto$}-- node[auto=left]{$W_{(0,1)}$} (1.5,4.5)--  (0,4.5)--  cycle;
\draw[ultra thick] (0,1.5)--(1.5,0)--(4.5,0);
\draw[ultra thick] (4.5,1.5)--(1.5,4.5)--(0,4.5);
\end{tikzpicture}
\ \ \ 
\begin{tikzpicture}
\draw[dashed](0,1.5)-- node[auto=left]{$V_{(0,0)}$} (1.5,0) --  (4.5,0)--  (4.5,1.5)--  (1.5,4.5)-- node[auto=left]{$V_{(0,1)}$} (0,4.5)--  cycle;
\draw[ultra thick] (0,1.5)--(1.5,0);
\draw[ultra thick] (0,4.5)--(1.5,4.5);
\end{tikzpicture}
\]
There is the only trivial gradient tree.

\item $L(0,0,0,0) \to L(0,0,-1,1) \to  L(0,0,0,1)$\ \par
\[
\begin{tikzpicture}
\draw[dashed](0,1.5)--  (1.5,0) -- node[auto=left]{$Z_{(0,0)}$} (4.5,0)--  (4.5,1.5)node[above=20pt,right=1pt]{$\bigotimes$}--  (1.5,4.5)--  (0,4.5)--  cycle;
\draw[ultra thick] (0,1.5)--(1.5,0)--(4.5,0)--(4.5,1.5);
\end{tikzpicture}
\ \ \ 
\begin{tikzpicture}
\draw[dashed](0,1.5)-- node[auto=left]{$W_{(0,0)}$} (1.5,0) --  (4.5,0)--  (4.5,1.5)node[above=20pt,right=1pt]{$\longmapsto$}-- node[auto=left]{$W_{(1,1)}$} (1.5,4.5)--  (0,4.5)--  cycle;
\draw[ultra thick] (1.5,0)--(0,1.5)--(0,4.5);
\draw[ultra thick] (4.5,0)--(4.5,1.5)--(1.5,4.5);
\end{tikzpicture}
\ \ \ 
\begin{tikzpicture}
\draw[dashed](0,1.5)-- node[auto=left]{$V_{(0,0)}$} (1.5,0) --  (4.5,0)-- node[auto=left]{$V_{(1,1)}$} (4.5,1.5)--  (1.5,4.5)--  (0,4.5)--  cycle;
\draw[ultra thick] (0,1.5)--(1.5,0);
\draw[ultra thick] (4.5,0)--(4.5,1.5);
\end{tikzpicture}
\]
There is the only trivial gradient tree.

\item $L(0,0,0,0) \to L(-1,0,0,1) \to  L(0,0,0,2)$\ \par
\[
\begin{tikzpicture}
\draw[dashed](0,1.5)--  (1.5,0) --  (4.5,0)--  (4.5,1.5)node[above=20pt,right=1pt]{$\bigotimes$} -- node[auto=left]{$Z_{(0,1)}$} (1.5,4.5)--  (0,4.5)--  cycle;
\draw[ultra thick] (4.5,0)--(4.5,1.5)--(1.5,4.5)--(0,4.5);
\end{tikzpicture}
\ \ \ 
\begin{tikzpicture}
\draw[dashed](0,1.5)--  (1.5,0) --  (4.5,0)-- node[auto=left]{$W_{(2,1)}$} (4.5,1.5)node[above=20pt,right=1pt]{$\longmapsto$}--  (1.5,4.5)-- node[auto=left]{$W_{(0,1)}$} (0,4.5)--  cycle;
\draw[ultra thick] (4.5,0)--(4.5,1.5);
\draw[ultra thick] (0,4.5)--(1.5,4.5);
\fill[black](0,1.5)circle(0.06) node[right]{$W_{(0,0)}$};
\fill[black](1.5,0)circle(0.06) node[above right]{$W_{(1,0)}$};
\fill[black](3,3)circle(0.06) node[below]{$W_{(1,1)}$};
\end{tikzpicture}
\ \ \ 
\begin{tikzpicture}
\draw[dashed](0,1.5)--  (1.5,0) --  (4.5,0)-- node[auto=left]{$V_{(2,2)}$} (4.5,1.5)--  (1.5,4.5)-- node[auto=left]{$V_{(0,2)}$} (0,4.5)--  cycle;
\draw[ultra thick] (4.5,0)--(4.5,1.5);
\draw[ultra thick] (0,4.5)--(1.5,4.5);
\fill[black](0,3)circle(0.06) node[right]{$V_{(0,1)}$};
\fill[black](3,0)circle(0.06) node[above]{$V_{(1,1)}$};
\fill[black](3,3)circle(0.06) node[below]{$V_{(1,2)}$};
\end{tikzpicture}
\]
For $Z_{(0,1)}\otimes W_{(j,1)}\mapsto V_{(j,2)}\ (j=0,1,2)$ there is the only trivial gradient tree. For the other compositions, there exists non-trivial gradient trees which are as following :
\[
\begin{tikzpicture}
\draw[dashed](0,1.5)--  (1.5,0) --  (4.5,0)--  (4.5,1.5) -- node[auto=right]{$Z_{(0,1)}$} (1.5,4.5)--  (0,4.5)--  cycle;
\draw[ultra thick] (4.5,0)--(4.5,1.5);
\draw[ultra thick] (0,4.5)--(1.5,4.5);
\fill[black](0,3)circle(0.06) node[left]{$V_{(0,1)}$};
\fill[black](3,0)circle(0.06) node[below]{$V_{(1,1)}$};
\fill[black](0,1.5)circle(0.06) node[left]{$W_{(0,0)}$};
\fill[black](1.5,0)circle(0.06) node[below]{$W_{(1,0)}$};
\draw[ultra thick] (4.5,0)--(4.5,1.5)--(1.5,4.5)--(0,4.5);
\draw (1.5,0) -- (4.5,0);
\draw [arrows = {-Stealth[scale=1.5]}] (1.5,0) -- (2.5,0);
\draw [arrows = {-Stealth[scale=1.5]}] (4.5,0) -- (3.5,0);
\draw (0,1.5) -- (0,4.5);
\draw [arrows = {-Stealth[scale=1.5]}] (0,1.5) -- (0,2.5);
\draw [arrows = {-Stealth[scale=1.5]}] (0,4.5) -- (0,3.5);
\end{tikzpicture}
\]

\item $L(0,0,0,0) \to L(0,-1,0,1) \to  L(0,0,0,2)$\ \par
\[
\begin{tikzpicture}
\draw[dashed](0,1.5)--  (1.5,0) --  (4.5,0)--  (4.5,1.5)node[above=20pt,right=1pt]{$\bigotimes$}--  (1.5,4.5)--  (0,4.5)-- node[auto=left]{$Z_{(0,0)}$} cycle;
\draw[ultra thick] (1.5,0)--(0,1.5)--(0,4.5)--(1.5,4.5);
\end{tikzpicture}
\ \ \ 
\begin{tikzpicture}
\draw[dashed](0,1.5)-- node[auto=left]{$W_{(0,0)}$} (1.5,0) --  (4.5,0)--  (4.5,1.5)node[above=20pt,right=1pt]{$\longmapsto$}--  (1.5,4.5)-- node[auto=left]{$W_{(0,2)}$} (0,4.5)--  cycle;
\draw[ultra thick] (0,1.5)--(1.5,0);
\draw[ultra thick] (0,4.5)--(1.5,4.5);
\fill[black](4.5,1.5)circle(0.06) node[left]{$W_{(1,2)}$};
\fill[black](4.5,0)circle(0.06) node[above left]{$W_{(1,1)}$};
\fill[black](0,3)circle(0.06) node[right]{$W_{(0,1)}$};
\end{tikzpicture}
\ \ \ 
\begin{tikzpicture}
\draw[dashed](0,1.5)-- node[auto=left]{$V_{(0,0)}$} (1.5,0) --  (4.5,0)--  (4.5,1.5)--  (1.5,4.5)-- node[auto=left]{$V_{(0,2)}$} (0,4.5)--  cycle;
\draw[ultra thick] (0,1.5)--(1.5,0);
\draw[ultra thick] (0,4.5)--(1.5,4.5);
\fill[black](0,3)circle(0.06) node[right]{$V_{(0,1)}$};
\fill[black](3,0)circle(0.06) node[above]{$V_{(1,1)}$};
\fill[black](3,3)circle(0.06) node[below]{$V_{(1,2)}$};
\end{tikzpicture}
\]
For $Z_{(0,0)}\otimes W_{(0,j)}\mapsto V_{(0,j)}\ (j=0,1,2)$ there is the only trivial gradient tree. For the other compositions, there exists non-trivial gradient trees which are as following :
\[
\begin{tikzpicture}
\draw[dashed](0,1.5)-- (1.5,0) --  (4.5,0)--  (4.5,1.5)--  (1.5,4.5)--  (0,4.5)-- node[left]{$Z_{(0,0)}$}  cycle;
\draw[ultra thick] (0,1.5)--(1.5,0);
\draw[ultra thick] (0,4.5)--(1.5,4.5);
\fill[black](3,0)circle(0.06) node[below]{$V_{(1,1)}$};
\fill[black](3,3)circle(0.06) node[above right]{$V_{(1,2)}$};
\fill[black](4.5,1.5)circle(0.06) node[right]{$W_{(1,2)}$};
\fill[black](4.5,0)circle(0.06) node[below]{$W_{(1,1)}$};
\draw[ultra thick] (1.5,0)--(0,1.5)--(0,4.5)--(1.5,4.5);
\draw (1.5,0) -- (4.5,0);
\draw [arrows = {-Stealth[scale=1.5]}] (1.5,0) -- (2.5,0);
\draw [arrows = {-Stealth[scale=1.5]}] (4.5,0) -- (3.5,0);
\draw (4.5,1.5) -- (1.5,4.5);
\draw [arrows = {-Stealth[scale=1.5]}] (4.5,1.5) -- (3.5,2.5);
\draw [arrows = {-Stealth[scale=1.5]}] (1.5,4.5) -- (2.5,3.5);
\end{tikzpicture}
\]

\item $L(0,0,0,0) \to L(0,0,-1,1) \to  L(0,0,0,2)$\ \par
\[
\begin{tikzpicture}
\draw[dashed](0,1.5)--  (1.5,0) -- node[auto=left]{$Z_{(0,0)}$} (4.5,0)--  (4.5,1.5)node[above=20pt,right=1pt]{$\bigotimes$}--  (1.5,4.5)--  (0,4.5)--  cycle;
\draw[ultra thick] (0,1.5)--(1.5,0)--(4.5,0)--(4.5,1.5);
\end{tikzpicture}
\ \ \ 
\begin{tikzpicture}
\draw[dashed](0,1.5)-- node[auto=left]{$W_{(0,0)}$} (1.5,0) --  (4.5,0)-- node[auto=left]{$W_{(2,2)}$} (4.5,1.5)node[above=20pt,right=1pt]{$\longmapsto$}--  (1.5,4.5)--  (0,4.5)--  cycle;
\draw[ultra thick] (0,1.5)--(1.5,0);
\draw[ultra thick] (4.5,0)--(4.5,1.5);
\fill[black](1.5,4.5)circle(0.06) node[below]{$W_{(1,2)}$};
\fill[black](0,4.5)circle(0.06) node[below right]{$W_{(0,1)}$};
\fill[black](3,0)circle(0.06) node[above]{$W_{(1,1)}$};
\end{tikzpicture}
\ \ \ 
\begin{tikzpicture}
\draw[dashed](0,1.5)-- node[auto=left]{$V_{(0,0)}$} (1.5,0) --  (4.5,0)-- node[auto=left]{$V_{(2,2)}$} (4.5,1.5)--  (1.5,4.5)--  (0,4.5)--  cycle;
\draw[ultra thick] (0,1.5)--(1.5,0);
\draw[ultra thick] (4.5,0)--(4.5,1.5);
\fill[black](0,3)circle(0.06) node[right]{$V_{(0,1)}$};
\fill[black](3,0)circle(0.06) node[above]{$V_{(1,1)}$};
\fill[black](3,3)circle(0.06) node[below]{$V_{(1,2)}$};
\end{tikzpicture}
\]
For $Z_{(0,0)}\otimes W_{(j,j)}\mapsto V_{(j,j)}\ (j=0,1,2)$ there is the only trivial gradient tree. For the other compositions, there exists non-trivial gradient trees which are as following :
\[
\begin{tikzpicture}
\draw[dashed](0,1.5)--  (1.5,0) -- node[auto=right]{$Z_{(0,0)}$} (4.5,0)--  (4.5,1.5)--  (1.5,4.5)--  (0,4.5)--  cycle;
\draw[ultra thick] (0,1.5)--(1.5,0);
\draw[ultra thick] (4.5,0)--(4.5,1.5);
\fill[black](0,3)circle(0.06) node[left]{$V_{(0,1)}$};
\fill[black](3,3)circle(0.06) node[above right]{$V_{(1,2)}$};
\fill[black](1.5,4.5)circle(0.06) node[right]{$W_{(1,2)}$};
\fill[black](0,4.5)circle(0.06) node[left]{$W_{(0,1)}$};
\draw[ultra thick] (0,1.5)--(1.5,0)--(4.5,0)--(4.5,1.5);
\draw (0,1.5) -- (0,4.5);
\draw [arrows = {-Stealth[scale=1.5]}] (0,1.5) -- (0,2.5);
\draw [arrows = {-Stealth[scale=1.5]}] (0,4.5) -- (0,3.5);
\draw (4.5,1.5) -- (1.5,4.5);
\draw [arrows = {-Stealth[scale=1.5]}] (4.5,1.5) -- (3.5,2.5);
\draw [arrows = {-Stealth[scale=1.5]}] (1.5,4.5) -- (2.5,3.5);
\end{tikzpicture}
\]

\item $L(0,0,0,0) \to L(0,0,0,1)  \to  L(0,0,0,2)$\ \par
\[
\begin{tikzpicture}
\draw[dashed](0,1.5)-- node[auto=left]{$Z_{(0,0)}$} (1.5,0) --  (4.5,0)-- node[auto=left]{$Z_{(1,1)}$} (4.5,1.5)node[above=20pt,right=1pt]{$\bigotimes$}--  (1.5,4.5)-- node[auto=left]{$Z_{(0,1)}$} (0,4.5)--  cycle;
\draw[ultra thick] (0,1.5)--(1.5,0);
\draw[ultra thick] (4.5,0)--(4.5,1.5);
\draw[ultra thick] (0,4.5)--(1.5,4.5);
\end{tikzpicture}
\ \ \ 
\begin{tikzpicture}
\draw[dashed](0,1.5)-- node[auto=left]{$W_{(0,0)}$} (1.5,0) --  (4.5,0)-- node[auto=left]{$W_{(1,1)}$} (4.5,1.5)node[above=20pt,right=1pt]{$\longmapsto$}--  (1.5,4.5)-- node[auto=left]{$W_{(0,1)}$} (0,4.5)--  cycle;
\draw[ultra thick] (0,1.5)--(1.5,0);
\draw[ultra thick] (4.5,0)--(4.5,1.5);
\draw[ultra thick] (0,4.5)--(1.5,4.5);
\end{tikzpicture}
\ \ \ 
\begin{tikzpicture}
\draw[dashed](0,1.5)-- node[auto=left]{$V_{(0,0)}$} (1.5,0) --  (4.5,0)-- node[auto=left]{$V_{(2,2)}$} (4.5,1.5)--  (1.5,4.5)-- node[auto=left]{$V_{(0,2)}$} (0,4.5)--  cycle;
\draw[ultra thick] (0,1.5)--(1.5,0);
\draw[ultra thick] (4.5,0)--(4.5,1.5);
\draw[ultra thick] (0,4.5)--(1.5,4.5);
\fill[black](0,3)circle(0.06) node[right]{$V_{(0,1)}$};
\fill[black](3,0)circle(0.06) node[above]{$V_{(1,1)}$};
\fill[black](3,3)circle(0.06) node[below]{$V_{(1,2)}$};
\end{tikzpicture}
\]
For $Z_{K}\otimes W_{K}\mapsto V_{2K}\ (K=(0,0),(0,1),(1,1))$ there is the only trivial gradient tree. For the other compositions, there exists non-trivial gradient trees which are as following :
\[
\begin{tikzpicture}
\draw[dashed](0,1.5)-- node[auto=right]{$Z_{(0,0)},W_{(0,0)}$} (1.5,0) --  (4.5,0)-- node[auto=right]{$Z_{(1,1)},W_{(1,1)}$} (4.5,1.5)--  (1.5,4.5)-- node[auto=right]{$Z_{(0,1)},W_{(0,1)}$} (0,4.5)--  cycle;
\draw[ultra thick] (0,1.5)--(1.5,0);
\draw[ultra thick] (4.5,0)--(4.5,1.5);
\draw[ultra thick] (0,4.5)--(1.5,4.5);
\fill[black](0,3)circle(0.06) node[left]{$V_{(0,1)}$};
\fill[black](3,0)circle(0.06) node[below]{$V_{(1,1)}$};
\fill[black](3,3)circle(0.06) node[above right]{$V_{(1,2)}$};
\draw (4.5,1.5) -- (1.5,4.5);
\draw [arrows = {-Stealth[scale=1.5]}] (4.5,1.5) -- (3.5,2.5);
\draw [arrows = {-Stealth[scale=1.5]}] (1.5,4.5) -- (2.5,3.5);
\draw (0,1.5) -- (0,4.5);
\draw [arrows = {-Stealth[scale=1.5]}] (0,1.5) -- (0,2.5);
\draw [arrows = {-Stealth[scale=1.5]}] (0,4.5) -- (0,3.5);
\draw (1.5,0) -- (4.5,0);
\draw [arrows = {-Stealth[scale=1.5]}] (1.5,0) -- (2.5,0);
\draw [arrows = {-Stealth[scale=1.5]}] (4.5,0) -- (3.5,0);
\end{tikzpicture}
\]

\item $L(-1,0,0,1) \to L(0,0,0,1) \to  L(0,0,0,2)$\ \par
\[
\begin{tikzpicture}
\draw[dashed](0,1.5)--  (1.5,0) -- node[auto=left]{$Z_{(1,0)}$} (4.5,0)--  (4.5,1.5)node[above=20pt,right=1pt]{$\bigotimes$}--  (1.5,4.5)--  (0,4.5)-- node[auto=left]{$Z_{(0,0)}$} cycle;
\draw[ultra thick] (0,1.5)--(0,4.5)--(1.5,4.5);
\draw[ultra thick] (1.5,0)--(4.5,0)--(4.5,1.5);
\end{tikzpicture}
\ \ \ 
\begin{tikzpicture}
\draw[dashed](0,1.5)-- node[auto=left]{$W_{(0,0)}$} (1.5,0) --  (4.5,0)-- node[auto=left]{$W_{(1,1)}$} (4.5,1.5)node[above=20pt,right=1pt]{$\longmapsto$}--  (1.5,4.5)-- node[auto=left]{$W_{(0,1)}$} (0,4.5)--  cycle;
\draw[ultra thick] (0,1.5)--(1.5,0);
\draw[ultra thick] (4.5,0)--(4.5,1.5);
\draw[ultra thick] (0,4.5)--(1.5,4.5);
\end{tikzpicture}
\ \ \ 
\begin{tikzpicture}
\draw[dashed](0,1.5)--  (1.5,0) --  (4.5,0)-- node[auto=left]{$V_{(2,1)}$} (4.5,1.5)--  (1.5,4.5)-- node[auto=left]{$V_{(0,1)}$} (0,4.5)--  cycle;
\draw[ultra thick] (4.5,0)--(4.5,1.5);
\draw[ultra thick] (0,4.5)--(1.5,4.5);
\fill[black](0,1.5)circle(0.06) node[right]{$V_{(0,0)}$};
\fill[black](1.5,0)circle(0.06) node[above right]{$V_{(1,0)}$};
\fill[black](3,3)circle(0.06) node[below]{$V_{(1,1)}$};
\end{tikzpicture}
\]
For $Z_{(0,0)}\otimes W_{(1,1)}\mapsto V_{(1,1)}$ and $Z_{(1,0)}\otimes W_{(0,1)}\mapsto V_{(1,1)}$ there is a non-trivial gradient tree and for the other compositions there exists the only trivial gradient tree. Non-trivial gradient trees are as following :
\[
\begin{tikzpicture}
\draw[dashed](0,1.5)--  (1.5,0) --  (4.5,0)-- node[auto=right]{$W_{(1,1)}$}  (4.5,1.5)--  (1.5,4.5)--  (0,4.5)-- node[auto=right]{$Z_{(0,0)}$} cycle;
\draw[ultra thick] (0,1.5)--(0,4.5)--(1.5,4.5);
\draw[ultra thick] (4.5,0)--(4.5,1.5);
\fill[black](3,3)circle(0.06) node[above right]{$V_{(1,1)}$};
\draw (4.5,1.5) -- (1.5,4.5);
\draw [arrows = {-Stealth[scale=1.5]}] (4.5,1.5) -- (3.5,2.5);
\draw [arrows = {-Stealth[scale=1.5]}] (1.5,4.5) -- (2.5,3.5);
\end{tikzpicture}
\ \ \ 
\begin{tikzpicture}
\draw[dashed](0,1.5)--  (1.5,0) -- node[auto=right]{$Z_{(1,0)}$} (4.5,0)--  (4.5,1.5)--  (1.5,4.5)-- node[auto=right]{$W_{(0,1)}$}  (0,4.5)--  cycle;
\draw[ultra thick] (1.5,0)--(4.5,0)--(4.5,1.5);
\draw[ultra thick] (0,4.5)--(1.5,4.5);
\fill[black](3,3)circle(0.06) node[above right]{$V_{(1,1)}$};
\draw (4.5,1.5) -- (1.5,4.5);
\draw [arrows = {-Stealth[scale=1.5]}] (4.5,1.5) -- (3.5,2.5);
\draw [arrows = {-Stealth[scale=1.5]}] (1.5,4.5) -- (2.5,3.5);
\end{tikzpicture}
\]

\item $L(0,-1,0,1) \to L(0,0,0,1) \to  L(0,0,0,2)$\ \par
\[
\begin{tikzpicture}
\draw[dashed](0,1.5)-- node[auto=left]{$Z_{(0,0)}$} (1.5,0) --  (4.5,0)--  (4.5,1.5)node[above=20pt,right=1pt]{$\bigotimes$}-- node[auto=left]{$Z_{(0,1)}$} (1.5,4.5)--  (0,4.5)--  cycle;
\draw[ultra thick] (0,1.5)--(1.5,0)--(4.5,0);
\draw[ultra thick] (4.5,1.5)--(1.5,4.5)--(0,4.5);
\end{tikzpicture}
\ \ \ 
\begin{tikzpicture}
\draw[dashed](0,1.5)-- node[auto=left]{$W_{(0,0)}$} (1.5,0) --  (4.5,0)-- node[auto=left]{$W_{(1,1)}$} (4.5,1.5)node[above=20pt,right=1pt]{$\longmapsto$}--  (1.5,4.5)-- node[auto=left]{$W_{(0,1)}$} (0,4.5)--  cycle;
\draw[ultra thick] (0,1.5)--(1.5,0);
\draw[ultra thick] (4.5,0)--(4.5,1.5);
\draw[ultra thick] (0,4.5)--(1.5,4.5);
\end{tikzpicture}
\ \ \ 
\begin{tikzpicture}
\draw[dashed](0,1.5)-- node[auto=left]{$V_{(0,0)}$} (1.5,0) --  (4.5,0)--  (4.5,1.5)--  (1.5,4.5)-- node[auto=left]{$V_{(0,2)}$} (0,4.5)--  cycle;
\draw[ultra thick] (0,1.5)--(1.5,0);
\draw[ultra thick] (0,4.5)--(1.5,4.5);
\fill[black](4.5,1.5)circle(0.06) node[left]{$V_{(1,2)}$};
\fill[black](4.5,0)circle(0.06) node[above left]{$V_{(1,1)}$};
\fill[black](0,3)circle(0.06) node[right]{$V_{(0,1)}$};
\end{tikzpicture}
\]
For $Z_{(0,0)}\otimes W_{(0,1)}\mapsto V_{(0,1)}$ and $Z_{(0,1)}\otimes W_{(0,0)}\mapsto V_{(0,1)}$ there is a non-trivial gradient tree and for the other compositions there exists the only trivial gradient tree. Non-trivial gradient trees are as following :
\[
\begin{tikzpicture}
\draw[dashed](0,1.5)-- node[auto=right]{$Z_{(0,0)}$} (1.5,0) --  (4.5,0)--  (4.5,1.5)--  (1.5,4.5)--node[auto=right]{$W_{(0,1)}$}  (0,4.5)--  cycle;
\draw[ultra thick] (0,1.5)--(1.5,0)--(4.5,0);
\draw[ultra thick] (0,4.5)--(1.5,4.5);
\fill[black](0,3)circle(0.06) node[left]{$V_{(0,1)}$};
\draw (0,1.5) -- (0,4.5);
\draw [arrows = {-Stealth[scale=1.5]}] (0,1.5) -- (0,2.5);
\draw [arrows = {-Stealth[scale=1.5]}] (0,4.5) -- (0,3.5);
\end{tikzpicture}
\ \ \ 
\begin{tikzpicture}
\draw[dashed](0,1.5)--node[auto=right]{$W_{(0,0)}$} (1.5,0) --  (4.5,0)--  (4.5,1.5)-- node[auto=right]{$Z_{(0,1)}$} (1.5,4.5)--  (0,4.5)--  cycle;
\draw[ultra thick] (4.5,1.5)--(1.5,4.5)--(0,4.5);
\draw[ultra thick] (0,1.5)--(1.5,0);
\fill[black](0,3)circle(0.06) node[left]{$V_{(0,1)}$};
\draw (0,1.5) -- (0,4.5);
\draw [arrows = {-Stealth[scale=1.5]}] (0,1.5) -- (0,2.5);
\draw [arrows = {-Stealth[scale=1.5]}] (0,4.5) -- (0,3.5);
\end{tikzpicture}
\]

\item $L(0,0,-1,1) \to L(0,0,0,1) \to  L(0,0,0,2)$\ \par
\[
\begin{tikzpicture}
\draw[dashed](0,1.5)-- node[auto=left]{$Z_{(0,0)}$} (1.5,0) --  (4.5,0)--  (4.5,1.5)node[above=20pt,right=1pt]{$\bigotimes$}-- node[auto=left]{$Z_{(1,1)}$} (1.5,4.5)--  (0,4.5)--  cycle;
\draw[ultra thick] (1.5,0)--(0,1.5)--(0,4.5);
\draw[ultra thick] (4.5,0)--(4.5,1.5)--(1.5,4.5);
\end{tikzpicture}
\ \ \ 
\begin{tikzpicture}
\draw[dashed](0,1.5)-- node[auto=left]{$W_{(0,0)}$} (1.5,0) --  (4.5,0)-- node[auto=left]{$W_{(1,1)}$} (4.5,1.5)node[above=20pt,right=1pt]{$\longmapsto$}--  (1.5,4.5)-- node[auto=left]{$W_{(0,1)}$} (0,4.5)--  cycle;
\draw[ultra thick] (0,1.5)--(1.5,0);
\draw[ultra thick] (4.5,0)--(4.5,1.5);
\draw[ultra thick] (0,4.5)--(1.5,4.5);
\end{tikzpicture}
\ \ \ 
\begin{tikzpicture}
\draw[dashed](0,1.5)-- node[auto=left]{$V_{(0,0)}$} (1.5,0) --  (4.5,0)-- node[auto=left]{$V_{(2,2)}$} (4.5,1.5)--  (1.5,4.5)--  (0,4.5)--  cycle;
\draw[ultra thick] (0,1.5)--(1.5,0);
\draw[ultra thick] (4.5,0)--(4.5,1.5);
\fill[black](1.5,4.5)circle(0.06) node[below]{$V_{(1,2)}$};
\fill[black](0,4.5)circle(0.06) node[below right]{$V_{(0,1)}$};
\fill[black](3,0)circle(0.06) node[above]{$V_{(1,1)}$};
\end{tikzpicture}
\]
For $Z_{(0,0)}\otimes W_{(1,1)}\mapsto V_{(1,1)}$ and $Z_{(1,1)}\otimes W_{(0,0)}\mapsto V_{(1,1)}$ there are non-trivial gradient trees and for the other compositions there exists the only trivial gradient tree. Non-trivial trees are as following :
\[
\begin{tikzpicture}
\draw[dashed](0,1.5)-- node[auto=right]{$Z_{(0,0)}$} (1.5,0) --  (4.5,0)-- node[auto=right]{$W_{(1,1)}$} (4.5,1.5)--  (1.5,4.5)--  (0,4.5)--  cycle;
\draw[ultra thick] (1.5,0)--(0,1.5)--(0,4.5);
\draw[ultra thick] (4.5,0)--(4.5,1.5);
\fill[black](3,0)circle(0.06) node[above]{$V_{(1,1)}$};
\draw (1.5,0) -- (4.5,0);
\draw [arrows = {-Stealth[scale=1.5]}] (1.5,0) -- (2.5,0);
\draw [arrows = {-Stealth[scale=1.5]}] (4.5,0) -- (3.5,0);
\end{tikzpicture}
\ \ \ 
\begin{tikzpicture}
\draw[dashed](0,1.5)-- node[auto=right]{$W_{(0,0)}$}  (1.5,0) --  (4.5,0)--  (4.5,1.5)-- node[auto=right]{$Z_{(1,1)}$} (1.5,4.5)--  (0,4.5)--  cycle;
\draw[ultra thick] (4.5,0)--(4.5,1.5)--(1.5,4.5);
\draw[ultra thick] (0,1.5)--(1.5,0);
\fill[black](3,0)circle(0.06) node[above]{$V_{(1,1)}$};
\draw (1.5,0) -- (4.5,0);
\draw [arrows = {-Stealth[scale=1.5]}] (1.5,0) -- (2.5,0);
\draw [arrows = {-Stealth[scale=1.5]}] (4.5,0) -- (3.5,0);
\end{tikzpicture}
\]
\end{itemize}

\ \par
{\bf Declarations of interest} : none

\end{document}